\documentclass{amsart}
\usepackage[T1]{fontenc}
\usepackage[utf8]{inputenc}
\usepackage[ margin=1in]{geometry}
\usepackage{graphicx} 
\usepackage{enumitem}
\usepackage{tabularx}
\usepackage{array}
\usepackage{comment}
\usepackage{amsmath}
\usepackage{amsfonts}
\usepackage{amsthm, amssymb}
\usepackage{thmtools}
\usepackage{booktabs}
\usepackage[dvipsnames]{xcolor}
\usepackage{bbm}
\usepackage{dirtytalk}
\usepackage{mathrsfs}
\usepackage{subcaption}
\usepackage{mathtools}
\usepackage{float}
\usepackage{tikz}
\usepackage{algorithm}
\usepackage[noend]{algpseudocode}
\usepackage{makecell}
\usepackage{longtable}
\usepackage{fancyhdr}
\usetikzlibrary{arrows.meta,decorations.markings, calc, positioning}



\newcommand{\man}{M}

\definecolor{cyan}{cmyk}{1, 0.4, 0, 0}

\newcommand{\HK}{\operatorname{HK}}
\newcommand{\PT}{\operatorname{PT}}

\usepackage[round]{natbib}
\usepackage[colorlinks=true,linkcolor=blue,citecolor=blue]{hyperref}
\usepackage[
        noabbrev,
        capitalise,
        nameinlink,
    ]
    {cleveref}

\newtheorem{thm}{Theorem}[section] 

\newtheorem{prop}[thm]{Proposition}
\newtheorem{lemma}[thm]{Lemma}
\newtheorem{asmpt}[thm]{Assumption}
\newtheorem{corollary}[thm]{Corollary}
\newtheorem{defn}[thm]{Definition}
\newtheorem{rmk}[thm]{Remark}

\crefname{thm}{theorem}{theorems}
\Crefname{thm}{Theorem}{Theorems}
\crefname{prop}{proposition}{propositions}
\Crefname{prop}{Proposition}{Propositions}

\title[Differential-Geometric Equivalence of HK and Cone-Wasserstein Spaces]{On the Differential-Geometric Equivalence of Hellinger-Kantorovich and Cone-Wasserstein Spaces}

\author{Tristan Luca Saidi}
\address{Department of Statistics and Data Science, Carnegie Mellon University}
\email{tsaidi@andrew.cmu.edu}

\author{Gonzalo Mena}
\address{Department of Statistics and Data Science, Carnegie Mellon University}
\email{gmena@andrew.cmu.edu}

\author{Florian Gunsilius}
\address{Department of Economics, Emory University}
\email{fgunsil@emory.edu}

\begin{document}

\begin{abstract}
    The Hellinger-Kantorovich (HK) space provides a natural geometry for nonnegative measures with varying total mass, but its differential-geometric structure is less well understood than that of the closely related Wasserstein space of probability measures. In this paper, we take a step toward resolving this issue. We show that the cone representation of the HK geometry via the Wasserstein metric preserves the local Riemannian geometry along a class of lifted geodesics. Specifically, we give a constructive procedure that produces a Wasserstein geodesic on the cone along which the HK Riemannian geometry is preserved pointwise, yielding an explicit isometry of tangent spaces between HK geodesics and their Wasserstein lifts. This connection makes many Wasserstein-geometric tools available for HK computations. Concretely, we use it to approximate parallel transport on HK space by lifting to the cone and applying recently developed Wasserstein parallel transport tools, circumventing the high-dimensional PDE arising from the HK covariant derivative. We also derive closed-form expressions for the covariant derivative and parallel transport on Euclidean metric cones, using the theory of warped-product manifolds. Finally, we present simulations illustrating the behavior of parallel geodesics in HK space, which reveal that the HK geometry couples spatial and mass variation through the geometry of the cone---a feature with nontrivial implications for applied use of the framework.
\end{abstract}

\maketitle


\section{Introduction}

Many modern data analytic problems involve operating with positive measures whose total mass is meaningful and may vary across time. Examples include problems involving imaging data with varying intensity, domain adaptation problems dealing with outliers, or inference problems in biology where cells proliferate and die. In such settings, classical optimal transport using probability measures is too rigid, as the Wasserstein metric captures displacement well but it enforces exact mass conservation. The Hellinger-Kantorovich distance (also known as the Wasserstein--Fisher--Rao distance) has been studied as it resolves this mismatch by combining transport and reaction. Seminal work demonstrated that the Hellinger--Kantorovich geometry enjoys many analogous properties and characterizations that Wasserstein geometry does -- in particular, Hellinger-Kantorovich geometry admits both a static entropy transport formulation and a dynamic continuity-reaction formulation \citep{liero2016optimal, liero2018optimal, kondratyev2016newoptimaltransportdistance, chizat2018interpolating, sejourne2023unbalanced, monsaingeon2021new}. 

The theory of optimal transport geometry has been developed extensively over the past decade. Computationally, scalable Sinkhorn-type algorithms have made unbalanced optimal transport tractable beyond toy problems, which has spurred its adoption in various application domains \citep{chizat2016scaling, sejourne2023unbalanced}. In particular, unbalanced optimal transport methods now play a central role in imaging and the analysis of biological dynamics, and have become standard tools in single-cell trajectory inference \citep{schiebinger2019optimal, zhang2024learning}. These methodological and application-driven developments make an understanding of the differential-geometric properties of the Hellinger-Kantorovich space even more valuable, which has been relatively understudied in comparison to its balanced counterpart. In the balanced setting, the geometry of the quadratic Wasserstein space has been rigorously and extensively developed. The Otto calculus in combination with the PDE formulations have provided a formal Riemannian structure on the Wasserstein space that supports gradient flows and second-order calculus \citep{ambrosio2005gradient, ambrosio2012user, gigli2012second, chewi2024statisticaloptimaltransport}. This Riemannian structure has been key in the development and analysis of methods for sampling, local linearization and parallel transport \citep{chewi2023log, cloninger2025linearized, cai2022linearized, saidi2026wassersteinparalleltransportpredicting}.

For the Hellinger--Kantorovich geometry, an analogous differential theory has only been partially studied. In particular, the cone representation of HK/WFR characterizes the \textit{metric} structure with Wasserstein distance over a suitable metric cone -- geometric properties of that cone have been studied in detail \citep{liero2016optimal, laschos2019geometric}. Beyond metric geometry, \cite{clancy2021interpolating} and \citet{clancy2022wasserstein} partially developed a Riemannian treatment of HK/WFR for applications in measure-valued splines, while others have derived explicit logarithmic and exponential maps for HK geometry in the pursuit of tangent space embeddings for data-analytic procedures \citep{cai2022linearized}. From a statistical perspective, recent work by \citet{ponnoprat2026minimaxoptimalestimationtransportgrowth} analyze the minimax properties of estimating transport-growth pairs, objects closely related to HK tangents, in the Gaussian-Hellinger geometry. Finally, \citet{gallouet2025regularity} show that the regularity of unbalanced optimal transport is inherited from the regularity of optimal transport, and they also show that the Ma-Trudinger-Wang condition for general costs on the associated metric cone implies the same condition for the cost on the original space.  Despite these developments, a basic obstruction still remains: the cone correspondence between Hellinger-Kantorovich geometry and Wasserstein geometry has only been exploited at the level of distances, \textit{geodesics} and cost functions, but higher-order differential objects of the Hellinger-Kantorovich space are still difficult to access both computationally and conceptually. 

The goal of this paper is to close that gap. In particular, we show that under suitable regularity conditions, the metric-cone representation of the Hellinger-Kantorovich geometry can be upgraded from a \textit{metric} equivalence to a local differential-geometric equivalence along an entire geodesic. Specifically, we construct an explicit lifting procedure that associates to any sufficiently regular Hellinger-Kantorovich geodesic a Wasserstein geodesic on the cone that admits explicit lifting and projection operators at the level of tangent objects that are isometric inverses. We show that this stronger equivalence allows one to import the tools of Wasserstein geometry to characterize and compute higher order differential objects for the Hellinger-Kantorovich space. In particular, we use this equivalence to tractably approximate Hellinger-Kantorovich parallel transport without solving its associated parallel transport PDE. 

Our contributions in this work are threefold. Firstly, we give an explicit lifting procedure that produces optimal cone lifts of suitably regular Hellinger-Kantorovich geodesics that preserves the local-differential geometric structure of the Hellinger-Kantorovich space.  Secondly, we instantiate this lifting procedure to leverage recently developed tools for Wasserstein geometry to compute Hellinger-Kantorovich parallel transport \textit{without} solving the PDE arising from the covariant derivative. Finally, we derive closed form expressions for the covariant derivative and parallel transport on Euclidean metric cones, which are of independent geometric interest. Conceptually, our results strengthen the utility of the cone representation in unbalanced transport by expanding the equivalence from one of metric structure to one of differential structure. 

The rest of the paper is organized as follows. In \Cref{sec: setup and background} we review the formal Riemannian structure of the Wasserstein and Hellinger-Kantorovich space, and in \Cref{sec: HK geometry via cone} we review the cone representation of Hellinger--Kantorovich geometry and introduce our lifting and projection operators. We then develop an explicit characteristic-based lifting procedure and prove its isometric and optimality properties. In \Cref{sec: HK parallel transport} we apply this framework to parallel transport, showing that the pullback cone Wasserstein covariant derivative and the Hellinger-Kantorovich covariant derivative coincide, and deriving an approximation theorem for Hellinger-Kantorovich parallel transport via cone transport. In \Cref{sec: cone parallel transport} we compute the cone covariant derivative and obtain closed-form formulas for parallel transport on Euclidean metric cones. The remaining sections contain proofs, implementation details, and simulations illustrating the properties of parallel geodesics in the Hellinger-Kantorovich space.

\section{Setup and Background} \label{sec: setup and background}

In this section, we describe the construction of Wasserstein space and its Riemannian structure, and we list key notation in \Cref{sec:notation}.  We will then use this construction to analogously characterize the Hellinger-Kantorovich geometry on general positive measures. The constructions described in this section have been described and studied in detail in \cite{clancy2021interpolating, chewi2024statisticaloptimaltransport, liero2018optimal, liero2016optimal, chizat2018interpolating, kondratyev2016newoptimaltransportdistance}. We remark that our ultimate goal in this work is to establish a stronger notion of equivalence between $(\mathfrak{M}_+^\Gamma(\Omega), \HK)$ and $(\mathcal{P}_2(\mathfrak{C}_\Omega), W_\mathfrak{C})$, where $\Omega$ is a compact subset of $\mathbb{R}^d$, and $\mathfrak{C}_\Omega$ is its associated metric cone. However, for the sake of generality, in this section we will present existing results regarding the Wasserstein geometry of $\mathcal{P}_2(\man)$, the space of probability measures with finite second moment on a $C^\infty$, boundaryless and complete Riemannian manifold $\man$. In subsequent sections, we will then instantiate $\man$ as an open and smooth extension of the cone $\mathfrak{C}_\Omega$ in order to characterize the geometry of $\mathfrak{M}_+^\Gamma(\Omega)$.

A minor technical point is worth making at the outset. The standard Riemannian-geometric description of the Wasserstein space is most cleanly stated when the underlying manifold $\man$ is smooth, complete and without boundary \citep{gigli2012second}; for the Hellinger-Kantorovich geometry, \citet{liero2016optimal} consider the space of positive measures over a compact subset $\Omega$ of $\mathbb{R}^d$. The latter construction can encounter a boundary issue, so throughout the Hellinger-Kantorovich portions of the following section we work on a localized class of measures contained in a compact subset $\Gamma \Subset \Omega^\circ$ to ensure that all measures are bounded away from $\partial\Omega$. We then choose an open neighborhood $U$ of $\Gamma$ such that $U \Subset \Omega^\circ $ and use test potentials $\varphi \in C^\infty_c(U)$ for our tangent space and weak solution characterizations.

\subsection{Optimal Transport} Define $\mathcal{P}_{2}(\man)$ to be the set of probability measures with finite second moment over a complete, connected, boundaryless and $C^{\infty}$ Riemannian manifold $\man$ with metric tensor $g$. The $2$-Wasserstein distance between two probability measures $\mu, \nu \in \mathcal{P}_{2}(\man)$ is defined by
\begin{equation}
    W_2(\mu, \nu) = \inf_{\gamma \in \Gamma_{\mu, \nu}}\left(\int_{\man \times \man} d_M(x,y)^2\,\gamma(dx, dy)\right)^{1/2} \label{eq: 2 wasserstein dist}
\end{equation}
where $\Gamma_{\mu, \nu}$ denotes the set of couplings of $\mu$ and $\nu$. The $2$-Wasserstein distance (which we will henceforth refer to as the Wasserstein distance) is indeed a metric, which renders $(\mathcal{P}_{2}(\man), W_2)$ a metric space. A key result is that of Brenier, who showed that for $\man = \mathbb{R}^d$, the optimal coupling takes the form $(X, \nabla \varphi(X)), X \sim \mu$ for some convex $\varphi$ when $\mu$ is absolutely continuous with respect to the Lebesgue measure.

\begin{thm}[Brenier]
\label{brenier}
    Let $\mu, \nu \in \mathcal{P}_2(\mathbb{R}^d)$ be probability measures such that $\mu$ has a density, and let $X \sim \mu$. If $\gamma^*$ is optimal for \cref{eq: 2 wasserstein dist} with $\man = \mathbb{R}^d$, then there exists a convex function $\varphi:\mathbb{R}^d \rightarrow \mathbb{R}$ such that $(X, \nabla\varphi(X))\sim \gamma^*$. 
\end{thm}

This theorem guarantees that if the source measure has a density, then the optimal transport coupling can be written as the source measure $\mu$ and $\nabla \varphi_\#\mu$, the pushforward of the source under some convex function -- note that the pushforward of a measure $\mu$ under a map $T$ is simply the measure $(T_\#\mu) (B) = \mu(T^{-1}(B))$ for all measurable $B$. The function $\nabla\varphi$ from \Cref{brenier} is often referred to as the \textit{Brenier map}. Brenier's theorem was later generalized to Riemannian manifolds, which requires generalizing the notion of convexity and concavity to non-Euclidean settings. In particular, given a function $\psi: M \rightarrow \mathbb{R} \cup \{\pm \infty\}$ its \textit{infimal convolution} $\psi^c$ with a cost function $c$ is defined by 
\[\psi^c(y) = \inf_{x \in M}\left\{c(x,y) - \psi(x)\right\}.\]
We say $\psi$ is $c$-concave if and only if $\psi^{cc} = (\psi^c)^c = \psi.$

\begin{thm}[Brenier-McCann, \citet{mccann2001polar}]
\label{brenier-mcann}
    Let $(M, g)$ be a complete Riemannian manifold and let $\mu, \nu \in \mathcal{P}_2(\man)$ with $\mu \ll \operatorname{vol}_g$. Then there exists a $c$-concave function $\psi: \man \rightarrow \mathbb{R}$ with $c(x,y) = \frac 12d_M(x,y)^2$ such that the optimal plan (in the sense of \Cref{eq: 2 wasserstein dist}) is induced by a $\mu$-a.e. unique map
    \[T(x) = \exp_x\left(-\nabla\psi(x)\right), \quad \gamma^* = (\operatorname{id}, T)_\# \mu\]
    where $\nabla$ is the Riemannian gradient and $\exp_p(v)$ is the Riemannian exponential map. Indeed, when $\man = \mathbb{R}^d$ we recover the Brenier map $T(x) = x - \nabla \psi(x) = \nabla \varphi(x)$ where $\varphi(x) = \frac{1}{2}\|x\|_2^2-\psi(x)$ is convex. 
\end{thm}

This concludes the \textit{static} formulation of optimal transport. The dynamic formulation is an equivalent characterization of Wasserstein geometry that borrows ideas from the theory of fluid mechanics. 

\subsection{Dynamic Formulation and Unbalanced Transport}

As alluded to, an equivalent perspective on optimal transport comes from fluid mechanics, which arrives at the same metric structure through a different formulation. In particular, let $(v_t)_{t \geq 0}$ be a time dependent family of vector fields over $\man$, and consider the ODE $\dot X_t = v_t(X_t)$. Let $\mu_t$ denote the law of $X_t$, where $X_0 \sim \mu_0$ and $X_t$ evolves according to the ODE described previously. Then, the dynamics of $\mu_t$ obey the so-called \textit{continuity equation},
\begin{equation}
    \partial_t\mu_t + \operatorname{div}_g (\mu_t v_t) = 0 \label{eq: balanced continuity eqn}
\end{equation}
in the distributional (or weak) sense, where $\operatorname{div}_g$ is the Riemannian divergence. Distributional solutions to the continuity equation are described in \Cref{def: weak solutions}.  
\begin{defn}[Distributional Solutions to \Cref{eq: balanced continuity eqn}, \citet{ambrosio2012user}]
\label{def: weak solutions}
    We say that a family of pairs $(\mu_t, v_t)$ solves the continuity equation on $(0, 1)$ weakly (also known as in the distributional sense) if for any bounded and Lipschitz test function $\varphi \in C^1_c((0, 1)\times \man)$ we have
    \[\int_0^1\int_{\man} \partial_t\varphi \, d\mu_tdt + \int_0^1\int_{\man} \langle\nabla\varphi,  v_t\rangle_g \, d\mu_tdt = 0\]
    where $\nabla$ is the Riemannian gradient.
\end{defn}
It turns out, one can identify optimal transport maps with vector fields satisfying the continuity equation that are optimal in a certain sense. This gives rise to the \textit{dynamic} formulation of optimal transport. Before stating this connection, we need to define a notion of continuity on the curve of measures with which we will bridge the continuity equation and the Wasserstein distance. 
\begin{defn}[Weakly Continuous Curve of Measures]
    A weakly continuous curve of measures is a map $\mu: [0, 1] \rightarrow \mathfrak{M}(\man)$ 
    from the interval to the space of measures where the measures $\mu_t$ evolve continuously with respect to the weak topology. In particular, for every bounded and continuous test function $f$, the mapping
    \[t \mapsto \int f \,d\mu_t \quad \text{is continuous in $t$.}\] 
\end{defn}
With this definition of weak continuity of curves of measures, we are now in a position to bridge the continuity equation and the Wasserstein distance. This connection is due to the celebrated \textit{Benamou-Brenier} theorem, stated below.

\begin{thm}[Benamou-Brenier, \citet{ ambrosio2012user, chewi2024statisticaloptimaltransport}] \label{benamou-brenier}
    Let $\mu_0, \mu_1 \in \mathcal{P}_{2}(\man)$ be absolutely continuous with respect to $\operatorname{vol}_g$. Then 
    \begin{equation*}
        W_2^2(\mu_0, \mu_1) = \inf \left\{\int_{0}^1\|v_t\|^2_{L^2(\widetilde\mu_t)} dt \right\}
    \end{equation*}
    where the infimum is taken over all weakly continuous distributional solutions to the continuity equation $(\tilde \mu_t, v_t)$ such that $\widetilde \mu_0 = \mu_0$ and $\widetilde \mu_1 = \mu_1$. Moreover, if $M = \mathbb{R}^d$ then the optimal curve $(\mu_t)_{t\geq 0}$ is unique and is described by $X_t \sim \mu_t$, where $X_t = (1-t)X_0 + tX_1$ and $(X_0, X_1) \sim \gamma^* \in \Gamma_{\mu_0, \mu_1}$ with $\gamma^*$ being an optimal coupling.
\end{thm}

The Benamou-Brenier theorem (\Cref{benamou-brenier}) bridges the static and dynamic formulations of optimal transport, proving that they are indeed equivalent characterizations of optimal transport geometry. In classical balanced transport, however, it is only possible to transport a measure $\mu$ to a measure $\nu$ if they have the same total measure. In many applications, it can be useful to consider situations where total mass can change over time. The dynamical formulation stated above, which considers minimum cost transport between probability measures, can be augmented to allow for \textit{unbalanced} transport: transport where total mass can change. To do this, the \textit{continuity-reaction} equation has been proposed, which allows for the creation and destruction of mass. In particular, the \textit{continuity-reaction} equation (also known as the \textit{inhomogeneous} continuity equation) is given by
    \begin{equation}
    \partial_t\mu_t + \nabla \cdot (\mu_tv_t) = 4\beta_t\mu_t. \label{eq: unbalanced continuity eqn}
\end{equation}

\begin{defn}[Distributional Solutions to \Cref{eq: unbalanced continuity eqn}, \citet{liero2016optimal}] \label{def: weak solns cre}
    
    We say that a family of triplets $(\mu_t, v_t, \beta_t)$ solve the continuity-reaction equation on $(0,1)$ in the distributional sense if for any bounded and Lipschitz test function  $\varphi \in C^1_c((0,1)\times \Omega)$ we have
    \[\int_0^1\int_{\Omega}\left( \partial_t \varphi_t + \langle \nabla \varphi_t, v_t\rangle + 4\beta_t\varphi_t\right) \, d\mu_t dt = 0.\]
\end{defn}
At first glance, the weak characterization appears to have a sign flip for the reaction component $4\beta_t \varphi_t$. But one can see why this sign flip occurs by integrating the (sign-flipped) reaction continuity equation against a test function $\varphi_t \in C^\infty_c((0,1) \times \Omega)$ and applying integration by parts in time and space,
\begin{multline*}
    -\int \varphi_t \partial_t\mu_t - \int \varphi_t\,\nabla \cdot(\mu_t v_t) + \int 4\beta_t\varphi_t\mu_t \\= \left[\int \partial_t\varphi_t d\mu_t + \int \langle \nabla \varphi_t, v_t\rangle \, d\mu_t + \int 4\beta_t\varphi_t\,d\mu_t\right] - \int_0^1\int_{\partial\Omega}  \langle v_t, n\rangle\rho_t\varphi_t\, dS\,dt = 0
\end{multline*}
where $n$ is the surface normal and $\rho_t$ is the Lebesgue density of $\mu_t$. Note that boundary terms (in time) disappear due to the compact support of $\varphi_t$, but the spatial boundary terms do not. By characterizing the weak solutions by 
\[\int \partial_t\varphi_t d\mu_t + \int \langle \nabla \varphi_t, v_t\rangle \, d\mu_t + \int 4\beta_t\varphi_t\,d\mu_t = 0\]
we are implicitly encoding the no flux condition $\rho_t\langle v_t, n\rangle = 0$ on $\partial\Omega$. This continuity-reaction equation leads to an analogous notion of distance between general non-negative measures, commonly referred to as the \textit{Wasserstein-Fisher-Rao} (WFR) or the \textit{Hellinger-Kantorovich} (HK) distance. In fact, in \citet{liero2016optimal} it was shown that HK is indeed a metric on the space of non-negative Radon measures over $\Omega$. 

\begin{defn}[Hellinger-Kantorovich Distance, \cite{liero2016optimal, kondratyev2016newoptimaltransportdistance, chizat2018interpolating}]
    Let $\Omega\Subset \mathbb{R}^d$ and let $\mathfrak{M}_+(\Omega)$ denote the space of non-negative measures on $\Omega$. For any absolutely continuous $\mu_0, \mu_1 \in \mathfrak{M}_+(\Omega)$ we define the Hellinger-Kantorovich distance to be
    \begin{equation}
        \HK^2(\mu_0, \mu_1) = \inf\left\{\int_0^1 \left(\|\widetilde v_t\|_{L^2(\widetilde \mu_t)}^2 + 4\|\widetilde\beta_t\|_{L^2(\widetilde \mu_t)}^2\right)dt \right\}.
    \end{equation}
    where the infimum is taken over all weakly continuous distributional solutions to the continuity-reaction equation $(\widetilde \mu_t, \widetilde v_t, \widetilde \beta_t)$ such that $\widetilde \mu_0 = \mu_0$ and $\widetilde \mu_1 = \mu_1$.
\end{defn}

\subsection{Riemannian Structure of Wasserstein and Hellinger-Kantorovich}  The dynamic formulations of balanced and unbalanced optimal transport are key concepts that allow one to formalize the Riemannian structure on the space of measures and probability measures. As we will see, the continuity and continuity-reaction equations indicate that we can view the space of velocity fields (Wasserstein) and velocity-reaction fields (Hellinger-Kantorovich) as a vector space of infinitesimal perturbations to a measure.
\medskip

\noindent \textbf{The Tangent Space.} The tangent space in optimal transport geometry indeed is the vector space of infinitesimal perturbations to a measure. We will start with formalizing the Wasserstein tangent space. In doing so, one would find that to each perturbation of a measure one can associate an infinite number of vector fields that produce it, as adding a divergence-free field does not change the marginal behavior of $\mu_t$. Thus, we need to establish a selection principle. To do so, we will define a notion of a \say{derivative} in a general metric space.

\begin{defn}[Metric Derivative]
    Let $(\mathcal{X}, d)$ be a metric space and $(x_t)_{t \geq 0}$ be a curve in $\mathcal{X}$. The metric derivative of the curve at time $t$ is given by
    \begin{equation*}
        |\dot x|(t) \triangleq \lim_{s\rightarrow t}\frac{d(x_s, x_t)}{|s - t|}
    \end{equation*}
    provided that the limit exists. 
\end{defn}

Now, for a pair of probability measures $\mu, \nu \in \mathcal{P}_{2}(\man)$ where $\mu$ admits a density, we will write $T_{\mu \rightarrow \nu}$ for the Brenier (or Brenier-McCann) map from $\mu$ to $\nu$ and we will write $|\dot \mu|$ for the metric derivative of a curve in $\mathcal{P}_{2}(\man)$ with respect to the Wasserstein metric.

\begin{thm}[\citet{ambrosio2005gradient}]
    Let $(M, g)$ be a smooth and complete Riemannian manifold without boundary and let $(\mu_t)_{t\geq 0}$ be an absolutely continuous curve, i.e. $\mu_t$ admits a Riemannian density and $|\dot \mu_t|$ exists for all $t \geq 0$. Then for every family of vector fields $(v_t)_{t\geq 0}$ for which \cref{eq: balanced continuity eqn} holds, it holds that $|\dot \mu_t|(t) \leq \|v_t\|_{L^2(\mu_t)}$ for all $t \geq 0$. Moreover, there exists a unique family $(v_t)_{t\geq 0}$ such that \cref{eq: balanced continuity eqn} holds and for which $|\dot \mu_t|(t) = \|v_t\|_{L^2(\mu_t)}$ for every $t \geq 0$. This family is such that
    \begin{equation*}
        v_t = \lim_{h \rightarrow0^+}\left(h^{-1}\log_x(T_{\mu_t\rightarrow \mu_{t+h}}(x))\right) \qquad \text{ in $L^2(\mu_t)$}
    \end{equation*}
    and
    \begin{equation*}
        v_t = \arg\inf \left\{\|\tilde{v}_t\|_{L^2(\mu_t)}^2 \, \Big|\, \partial_t\mu_t + \operatorname{div}_g(\tilde v_t\mu_t) = 0\right\}
    \end{equation*}
    where $\log$ is the Riemannian logarithmic map, i.e. the inverse of the Riemannian exponential map.  
\end{thm}

Note that this result -- coupled with Brenier's theorem -- indicates that the velocity field that coincides (in an $L^2(\mu_t)$ sense) with the minimal norm solution and the metric derivative of the path is a limit of gradients. This will be our selection principle: we will choose the representative for a perturbation $\partial_t\mu_t$ to be the minimal-norm velocity field generating it. Now we can formally define the Wasserstein tangent space. 

\begin{defn}[Wasserstein Tangent Space, \citet{ambrosio2012user}]
    Let $\mu \in \mathcal{P}_2(\man)$. We define the tangent space to $\mathcal{P}_2(\man)$ at $\mu$ to be 
    \begin{equation}
        T_\mu \mathcal{P}_2(\man) = \overline{\left\{\nabla \varphi \, | \, \varphi \in C_c^{\infty}(\man)\right\}}^{L^2(\mu)}
    \end{equation}
    where $\overline{\{\cdot \}}^{L^2(\mu)}$ denotes the $L^2(\mu)$ closure and $C_c^{\infty}$ denotes the set of compactly supported and smooth maps. We also endow $T_\mu \mathcal{P}_2(\man)$ with the metric tensor
    \begin{equation}
        \langle \nabla \varphi_1, \nabla \varphi_2\rangle_{\mu} \triangleq \int_\man \langle\nabla\varphi_1, \nabla\varphi_2 \rangle_g \,d\mu 
    \end{equation}
    where $\langle \cdot, \cdot \rangle_g$ is the inner product defined by the metric tensor $g$.
\end{defn}

We are now in a position to introduce analogous structure on the Hellinger Kantorovich space. Starting from the continuity-reaction equation, one can show that the minimum norm solution is unique and is the gradient of a potential. We can then use this to construct the tangent space and the metric tensor as we did for the Wasserstein space.
\begin{prop}[\citet{clancy2021interpolating}, Proposition 27]
    Suppose $(\mu_t, v_t, \beta_t)$ satisfy \Cref{def: weak solns cre}. Then there exists a unique and minimal norm solution to the problem,
    \begin{equation*}
        (\overline v_t, \overline \beta_t) = \arg\inf\left\{\|\tilde v_t\|_{L^2(\mu_t)}^2 + 4\|\tilde \beta_t\|_{L^2(\mu_t)}^2  \, \Big| \, \nabla \cdot (\tilde v_t \mu_t) - 4\tilde \beta_t \mu_t = \nabla \cdot (v_t\mu_t) - 4\beta_t\mu_t\right\}
    \end{equation*}
    with $(\overline v_t, \overline \beta_t) \in \overline{\left\{(\nabla \varphi, \varphi) \, | \, \varphi \in C^{\infty}_c(\mathbb{R}^d)\right\}}^{L^2(\mu_t) \times L^2(\mu_t)}$.
\end{prop}

Note that \citet{clancy2021interpolating} formulates the HK tangent space and covariant derivative on $\mathfrak{M}_+(\mathbb{R}^d)$, while \citet{liero2016optimal} work on ($\mathfrak{M}_+(\Omega)$, HK) -- despite this discrepancy, the two descriptions coincide on any class of curves supported in a fixed interior region $\Gamma \Subset \Omega^\circ$. We now make this localization precise. Let $\Gamma \Subset U \Subset \Omega^\circ$
with $U$ open and define
\[
\mathfrak{M}_+^\Gamma(\Omega)
\triangleq
\{\mu\in \mathfrak{M}_+(\Omega): \operatorname{supp}\mu\subset \Gamma\}.
\]
As discussed earlier, throughout this section we work on the localized class $\mathfrak{M}_+^\Gamma(\Omega)$ and use test potentials in $C^\infty_c(U)$ for a fixed open neighborhood $U$ of $\Gamma$ contained in $\Omega^\circ$.
 
\begin{defn}[Hellinger-Kantorovich Tangent Space]
    For $\mu \in \mathfrak{M}_+^\Gamma(\Omega)$, we define the tangent space to be 
    \begin{equation}
        T_\mu \mathfrak{M}_+^{\Gamma} = \overline{\left\{(\nabla \varphi, \varphi) \, | \, \varphi \in C_c^{\infty}(U)\right\}}^{L^2(\mu) \times L^2(\mu)}.
    \end{equation}
    We also endow $T_\mu \mathfrak{M}_+^{\Gamma}$ with the metric tensor
    \begin{equation}
        \langle s_1,  s_2\rangle_{\mu} \triangleq \int \langle v_1, v_2\rangle + 4\beta_1\beta_2 \,  d\mu
    \end{equation}
    for any $s_1, s_2\in T_\mu \mathfrak{M}_+^{\Gamma}$, with $s_i = (v_i, \beta_i)$. 
\end{defn}

\noindent \textbf{Connections and the Covariant Derivative.} 
On an abstract manifold $(\man, g)$ that isn't embedded in an ambient space, we have no obvious way to compare vectors in the tangent space at two points $p, q \in \man$, $p \neq q$. Therefore, we need a way of connecting the separate vector spaces $T_p\man$ and $T_q\man$ when $p \neq q$. One can achieve this by defining a rule $\nabla$ for differentiating vector fields against each other on $\man$ in a way that preserves the structure of the metric $g$.

\begin{defn}[Affine connections and the covariant derivative, \citet{lee2018introduction}]
    An affine connection on a smooth manifold \(\man\) is a rule \(\nabla\) that assigns to each pair of smooth vector fields \(X,Y\) another smooth vector field \(\nabla_XY\), satisfying the following properties. For all smooth vector fields \(X,Y,Z\), smooth functions \(f,h\in C^\infty(\man)\), and constants \(a,b\in\mathbb R\),
    \begin{enumerate}
        \item \textit{\(C^\infty(\man)\)-linearity in the first argument:} $\nabla_{fX+hZ}Y=f\nabla_XY+h\nabla_ZY.$
        \item \textit{\(\mathbb R\)-linearity in the second argument:} $\nabla_X(aY+bZ)=a\nabla_XY+b\nabla_XZ.$
        \item \textit{Leibniz rule in the second argument:} $\nabla_X(fY)=X(f)Y+f\nabla_XY,$
        
    \end{enumerate}
    where \(X(f)=g(X,\nabla f)\) is the derivative of \(f\)
    in the direction \(X\) and $\nabla f$ is the Riemannian gradient of $f$. The vector field \(\nabla_XY\) is called the covariant derivative of \(Y\)
    in the direction \(X\).
\end{defn}

While this gives us a way to connect tangent spaces, the covariant derivative might distort the geometry induced by the metric $g$ -- i.e. it might not be \say{metric compatible}. A fundamental result of Riemannian geometry, however, is that for any Riemannian manifold $(\man, g)$ there exists a \textit{unique} connection that is torsion free and is metric compatible. 

\begin{thm}[The fundamental theorem of Riemannian Geometry, \cite{petersen2006riemannian}]
    If $\man$ is endowed with a Riemannian metric $g$, then there exists a unique connection $\nabla$ called the \textit{Levi-Civita connection} that is
    \begin{enumerate}
        \item \textit{Torsion free:} $\nabla_XY - \nabla_YX = [X,Y]$, where $[X,Y]$ is the Lie bracket of $X$ and $Y$. 
        \item \textit{Metric compatible:} $Xg(Y, Z) = g(\nabla_XY, Z) + g(Y, \nabla_XZ).$ 
    \end{enumerate}
\end{thm}
In the case of $(\mathcal{P}_2(\man), W_2)$ and $(\mathfrak{M}_+^{\Gamma}(\Omega), \HK)$ we can explicitly construct the covariant derivative, but one needs to manually verify that they are torsion free and metric compatible as neither spaces are finite-dimensional Riemannian manifolds.

\begin{prop}[Wasserstein Covariant Derivative, \citet{clancy2021interpolating, gigli2012second}]
\label{Wasserstein covariant deriv}
    Let $\man$ be a $C^\infty$, complete and boundaryless manifold endowed with its Levi-Civita connection $\nabla^\man$. Let $(\mu_t)_{t\geq 0}$ be a curve through $\mathcal{P}_2(\man)$ with tangent field $\nabla \varphi_t$ solving \cref{eq: balanced continuity eqn}, and therefore driving the dynamics of $\mu_t$. Also let $v_t$ be an absolutely continuous tangent vector field along $\mu_t$ and let $\Pi_{\mu_t}$ be the orthogonal projection onto $T_{\mu_t}\mathcal{P}_2(\man)$ in $L^2(\mu_t)$. Then the differential operator $\nabla_{(\nabla\varphi_t)}^{W_2}$ given by
    \begin{equation*}
        \mathbf{D}_t^{W_2}v_t = \partial_t v_t + \nabla^{\man}_{( \nabla \varphi_t)} v_t, \qquad \nabla_{(\nabla\varphi_t)}^{W_2}v_t = \Pi_{\mu_t}\left(\mathbf{D}^{W_2}_tv_t \right)
    \end{equation*}
    is a valid covariant derivative, is torsion free, and is metric compatible. The operator $\mathbf{D}_t^{W_2}$ is referred to as the total derivative. 
\end{prop}
This result establishes a closed form PDE describing a differential operator with the desired characteristics of the Levi-Civita connection. We also have an analogous covariant derivative for the Hellinger-Kantorovich space.

\begin{prop}[Hellinger-Kantorovich Covariant Derivative, \citet{clancy2021interpolating}]
\label{HK covariant deriv}
    Let $\Gamma \Subset U \Subset \Omega ^\circ$ with $U$ open and $\Omega \Subset \mathbb{R}^d$ and let $(\mu_t)_{t\geq 0}$ be a curve of measures supported on $\Gamma$ with tangent field $(\nabla \varphi_t, \varphi_t)$. Also let $(v_t, \beta_t)$ be a sufficiently regular tangent vector field along $\mu_t$ and let $\Pi_{\mu_t}$ be the orthogonal projection onto $T_{\mu_t}\mathfrak{M}_+^\Gamma$. Then the differential operator $\nabla_{(\nabla\varphi_t, \varphi_t)}^{\HK}$ given by
    \begin{equation*}
        \mathbf{D}_t^{\HK} \begin{pmatrix}
            v_t \\
            \beta_t
        \end{pmatrix} = \begin{pmatrix}
            \partial_t v_t + \nabla v_t \cdot \nabla \varphi_t + 2\varphi_tv_t+2\beta_t\nabla\varphi_t \\
            \partial_t \beta_t + \frac{1}{2}\langle \nabla \beta_t, \nabla \varphi_t\rangle + 2\varphi_t\beta_t
        \end{pmatrix}, \qquad \nabla_{(\nabla\varphi_t, \varphi_t)}^{\HK} \begin{pmatrix}
            v_t \\
            \beta_t
        \end{pmatrix} = \Pi_{\mu_t}\left(\mathbf{D}_t^{\HK}\begin{pmatrix}
            v_t \\ \beta_t
        \end{pmatrix}\right)
    \end{equation*}
    is a valid covariant derivative, is torsion free, and is metric compatible. The operator $\mathbf{D}_t^{\HK}$ is referred to as the total derivative.
\end{prop}

\noindent \textbf{The Logarithmic Entropy Transport Functional.}
To compute the Hellinger-Kantorovich distance in practice, we will use an equivalent characterization that has been studied in the literature. Let $\pi_j = \Pi^j_\#\pi$ be the marginals of the joint measure $\pi \in \mathfrak{M}_+(\Omega \times \Omega)$ and define the Hellinger-Kantorovich entropy-transport functional as
\begin{equation} \label{eq: let functional}
    E(\pi; \mu_0, \mu_1) = \int_{\Omega}F\left(\frac{d\pi_0}{d\mu_0}\right)d\mu_0 + \int_{\Omega}F\left(\frac{d\pi_1}{d\mu_1}\right)d\mu_1 + \int_{\Omega \times \Omega}c(\|x_0 - x_1\|_2)d\pi
\end{equation}
where
\begin{equation*}
    c(L) = \begin{cases}
        -2\log(\cos(L)) & L < \pi/2 \\
        \infty & L \geq \pi/2
    \end{cases}
\end{equation*}
and $F(\rho) = \rho\log\rho - \rho + 1.$ \citet{liero2016optimal} show that 
\begin{align} \label{eq: let hk equivalence}
    \HK(\mu_0, \mu_1) &= \inf\left\{E(\pi;\mu_0, \mu_1)^{1/2} \, |\, \pi  \in \mathfrak{M}_+(\Omega \times \Omega), \pi_j  \ll \mu_j\right\}.
\end{align}
As we will see in future sections, this characterization of Hellinger-Kantorovich geometry will enable tractable procedures for computing distances and other objects of interest. 

\begin{rmk}[Diameter bound] \label{diameter bound}
    Throughout the paper we assume that all measures $\mu_i$ are supported inside a set $\Omega$ with diameter strictly bounded by $\pi/2$ -- in practice, this condition can be met simply by rescaling. This assumption guarantees that we are in the \textit{reaction-transport} regime only, and $\mu_i \ll \pi_i^*$ from \Cref{eq: let functional}. Furthermore, by admissibility of $\pi^*$ we also necessarily have mutual absolute continuity $\pi_i \sim \mu_i$.  
\end{rmk}
\begin{proof}
    The property $\mu_i \ll \pi_i^*$  follows directly from the optimality conditions characterized in \citet[Theorem 6.3.b]{liero2018optimal}. 
\end{proof}

\noindent \textbf{The Exponential and Logarithmic Maps.} For $\mathcal{P}_{2}(\man)$, the exponential map is defined to be $\mathbf{exp}_{\mu}(u) \triangleq (\exp(u))_\#\mu$ where $\exp(\cdot)$ is the exponential map of $\man$ and $u \in L^2(\mu)$. In the case of $\man = \mathbb{R}^d$, this operation is trivial -- we have
\begin{equation}\label{eq: wasserstein euclidean expmap}
    \mathbf{exp}_\mu(u) = (\operatorname{id} + u)_\# \mu
\end{equation}
where $\operatorname{id}$ is the identity map, $x \mapsto x$. In the case of the logarithmic map, things are analogous. We define $(\mathbf{log}_\mu\nu) (x) \triangleq \log_x(T_{\mu \rightarrow \nu}(x))$ where $T_{\mu \rightarrow \nu}$ is the Brenier map from $\mu$ to $\nu$. Note that we will denote $\overline{\log}$ as the normalized vector (with respect to the metric tensor). Again, in the Euclidean case this reduces to
\begin{equation}\label{eq: wasserstein euclidean logmap}
    (\mathbf{log}_\mu\nu )(x) = (T_{\mu \rightarrow \nu} - \operatorname{id})(x).
\end{equation}
In the case of Hellinger-Kantorovich, the exponential map and logarithmic maps are more challenging to describe and compute. When one assumes that $\mu_i \ll \pi_i$ where $\pi$ is the optimized coupling in \Cref{eq: let functional} (which means we are in the reaction-transport only regime), they admit a closed form as derived by \citet{cai2022linearized}. 

\begin{prop}[Corollary of Proposition 4.8, \citet{cai2022linearized}]
\label{HK exponential map closed form}
    Let $\mu_0, \mu_1 \in \mathfrak{M}_+(\Omega)$ and let $s_0 = (v_0, \beta_0)$ be a tangent field at $\mu_0$. Suppose further that for the optimal transport plan $\pi$ of the logarithmic entropy transport functional problem between $\mu_0$ and $\mu_1$, we have that $\mu_i \ll \pi_i$. Set $a_t = t\|v_0\|_2$, $b_t = 1 + 2t\beta_0$, $S_t = \{x \in \Omega \, | \, (a_t(x), b_t(x)) = (0, 0) \}$, $q_t = \sqrt{a_t^2 + b_t^2}$ and $\varphi_t = \text{atan2}(a_t, b_t).$ Finally, let 
    \[T_t ( x) = x + \begin{cases}
        \frac{v_0(x)}{\|v_0(x)\|_2}\cdot \varphi_t(x) & \text{if } v_0(x) \neq 0 \\
        0 & \text{o.w.}
    \end{cases}.\]
    Then it holds that the curve given by $\mu_t = \mathbf{exp}_{\mu_0}(ts_0)$ where 
    \[\mathbf{exp}_{\mu_0}(ts_0) \triangleq (T_t)_\#(q_t(x)^2 \mu_0|_{\Omega \setminus S_t})\]
    is a constant speed HK geodesic between $\mu_0$ and $\mu_1$.  

\end{prop}

\begin{prop}[Proposition 4.1, \citet{cai2022linearized} ]
\label{HK logarithmic map closed form}
    Let $\mu_0, \mu_1 \in \mathfrak{M}_+(\Omega)$ and let $\pi = (\operatorname{id}, T)_\#\pi_0$ be the minimizer in \Cref{eq: let functional} for some measurable $T: \Omega \rightarrow\Omega$ and $\pi_i = \Pi^i_\# \pi$. Consider the Lebesgue decompositions of $\mu_0$ and $\mu_1$ with respect to the marginals of $\pi$, i.e. $\mu_0 = u_0 \pi_0$ and $\mu_1 = u_1 [(T)_\# \pi_0] = u_1\pi_1$. Then 
        \[
        v_0(x) =
        \begin{cases}
        \dfrac{T(x)-x}{\|T(x)-x\|_2}\,
        \sqrt{\dfrac{u_{1}(T(x))}{u_0(x)}}\,
        \sin\bigl(\|T(x)-x\|_2\bigr),
        & T(x)\neq x,\\
        0, & T(x)=x,
        \end{cases}
    \]
    and
    \[
        \beta_0(x)
        =
        \frac{1}{2}\left(
        \sqrt{\dfrac{u_{1}(T(x))}{u_0(x)}}\,
        \cos\bigl(\|T(x)-x\|_2\bigr)-1
        \right).
    \]
    satisfy $\mathbf{exp}_{\mu_0}(v_0, \beta_0) = \mu_1$.
\end{prop}

These results indicate that if the minimizing coupling $\pi$ from \Cref{eq: let functional} is supported on the assignment $(\operatorname{id}, T)$ for some measureable $T$, we can compute the exponential and logarithmic map explicitly. Fortunately, as we will discuss in \Cref{clancy coupling supported on a map}, this condition is guaranteed by absolute continuity of $\mu_0, \mu_1$ and $\mu_i \ll \pi_i$ (which follows from the diameter bound from \Cref{diameter bound}).  Intuitively, this means that all mass growth and destruction stems from the reaction component of the continuity-reaction equation, and it prevents the creation of mass from nothing (ruling out the Hellinger-only regime). In many applications of interest, like genomics for example, this is a reasonable and an arguably desirable property.    

\section{Hellinger-Kantorovich Geometry via Cone Lifting} \label{sec: HK geometry via cone}

While Hellinger-Kantorovich geometry allows for a richer notion of transport, the tradeoff is that it induces more computational challenges. The tangent space, the covariant derivative and the exponential/logarithmic maps are all more challenging to describe and compute than their Wasserstein counterparts. Remarkably, \citet{liero2016optimal} show that Hellinger-Kantorovich geometry can be captured (in a delicate sense) from $W_2$ geometry on an augmented base domain. In particular, this construction entails lifting compactly supported positive measures from the base space $\Omega \Subset \mathbb{R}^d$ to probability measures on $\mathfrak{C}_{\Omega}$, the metric cone of the base space. Under a particular set of lifts, it turns out that the Wasserstein geometry on the cone recovers Hellinger-Kantorovich geometry on the original space.

That being said, the connection between Hellinger-Kantorovich geometry and Wasserstein geometry on the cone has only been demonstrated at the level of distances and geodesic paths. In this section, we detail a specific lifting procedure that allows one to strengthen this connection through an explicit and tractable isometric map between the Hellinger-Kantorovich tangent space and a subset of the cone Wasserstein tangent space. The lifting procedure is simple to analyze theoretically and computationally tractable; we demonstrate its utility in a later section by showing that the pullback Wasserstein Levi-Civita connection induced by this lift coincides with the Hellinger-Kantorovich Levi-Civita connection. This will allow us to use recently developed tools for Wasserstein geometry to easily approximate Hellinger-Kantorovich parallel transport, avoiding the need to solve a high-dimensional parallel transport PDE.

\subsection{Metric Cones}

We will start by giving a brief description of metric cones. For a more detailed treatment of metric geometry we recommend \citet{burago2001course}. For a more detailed treatment of warped-product manifolds, which are a class of manifolds that are closely related to metric cones, we recommend \citet{o1983semi}.  

\begin{defn}[Metric cone, \citet{burago2001course}]
\label{def: metric cone}
    For a Riemannian manifold $(\man, g)$, the metric cone is defined to be
    \[\mathfrak{C}_\man = \man \times \mathbb{R}_{+} / (M \times \{0\})\]
    where a point in $\mathfrak{C}_\man$ is written as $(x, r)$. In particular, all points $\man \times \{0\}$ are identified with one point called the cone \textit{apex}, denoted $\mathfrak{o}$. The metric tensor on the cone is given by $g_{(x,r)} = dr^2 + r^2g_x$, while the metric (as one would have in a metric space) is given by 
    \begin{equation*}
    d_{\mathfrak{C}}(z_0, z_1)^2 = r_0^2 + r_1^2 - 2r_0r_1\cos(d_{\man}(x_0, x_1) \land \pi)
    \end{equation*}
    for $z_i = (x_i, r_i)$.
\end{defn}

For metric cones of Euclidean space, we are lucky in the sense that geodesics are available in closed form. This is made explicit and precise below.

\begin{prop}[Geodesics on $\mathfrak{C}_{\Omega}$, $\Omega \Subset  \mathbb{R}^d$, \citet{liero2016optimal}]
    Suppose $z_0 = (x_0, r_0)$ and $z_1 = (x_1, r_1)$ satisfy $\|x_1 - x_0\|_2 < \pi/2$ and $r_0, r_1 \in [r_{\min}, r_{\max}]$ for some $0 < r_{\min} \le r_{\max} < \infty$. Then the function $Z(s; \cdot, \cdot ):\mathfrak{C}_\Omega \times \mathfrak{C}_\Omega \rightarrow \mathfrak{C}_\Omega$ given by 
    \begin{equation} \label{eq: cone geodesics}
        Z(s; z_0, z_1) = \left[X\left(s; z_0, z_1\right) , R\left(s; z_0, z_1\right)\right]
    \end{equation}
    where
    \begin{align*}
        R(s; z_0, z_1)^2 &= (1-s)^2r_0^2 + s^2r_1^2 + 2s(1-s)r_0r_1\cos( \|x_0 - x_1\|_2) \\
        \\
        X(s; z_0, z_1) &= (1 - \rho(s; z_0, z_1))x_0 + \rho(s; z_0, z_1)x_1 \\
        \\[0.05ex]\rho(s; z_0, z_1) &= \frac{1}{\|x_0 - x_1\|_2}\arccos\left(\frac{(1-s)r_0 + sr_1\cos(\|x_0 - x_1\|_2)}{R(s; z_0, z_1)}\right)
    \end{align*}
    is a unique constant-speed geodesic between $z_0$ and $z_1$ where $d_{\mathfrak{C}}(\gamma(s), \gamma(t)) = |t-s|d_{\mathfrak{C}}(z_0, z_1)$. 
\end{prop}

From this closed form expression, we can directly compute the exponential and logarithmic map on $\mathfrak{C}_\Omega$, the proof for which we provide in \Cref{sec: proof of cone exponential and log map}.
\begin{prop}[Exponential and Logarithmic maps on $\mathfrak{C}_\Omega$]
\label{cone exponential and log map}
    Let $z_i = (x_i, r_i) \in \mathfrak{C}_{\Omega}$ and let $(v_x, v_r) \in T_{z_0}\mathfrak{C}_\Omega$ with $\|v_x\|_2 > 0$. Then for $0 < \|x_0 - x_1\|_2 < \pi$ we have
    \begin{equation} \label{eq: cone exp map}
        \exp^{\mathfrak{C}}_{z_0}((v_x, v_r)) = \left(x_1, r_1\right)
    \end{equation}
    with 
    \begin{equation*}
        x_1 = x_0 + \theta\frac{v_x}{\|v_x\|_2}, \quad r_1 = \sqrt{\|v_x\|_2^2r_0^2 + (v_r+r_0)^2}, \quad \theta = \operatorname{atan2}(r_0\|v_x\|_2, v_r + r_0)
    \end{equation*}
    and 
    \begin{equation} \label{eq: cone log map}
        \log_{z_0}^\mathfrak{C}(z_1) = \left[\frac{\sin(\|x_1 - x_0\|_2)r_1(x_1 - x_0)}{\|x_1 - x_0\|_2r_0}, r_1\cos(\|x_1 - x_0\|_2) - r_0\right].
    \end{equation}
\end{prop}

Finally, we provide a derivation of the covariant derivative and closed form parallel transport on $\mathfrak{C}_\Omega$ in \Cref{sec: cone parallel transport}. To obtain these objects we appeal to the theory of \textit{warped-product} manifolds \citep{o1983semi}, which are a class of product manifolds to which the cone space $\mathfrak{C}_\Omega$ belongs. 

\subsection{Lifting Measures, Geodesics and Tangent Vectors.} \label{lifting}

Having established the structure of a metric cone, we will now discuss how one lifts and projects measures to and from it. In particular, given a measure $\lambda \in \mathcal{P}_2(\mathfrak{C}_{\Omega})$ we can project it to a measure in $\mathfrak{M}_+(\Omega)$ with the map $\mathfrak{P}$ characterized by
\begin{equation} \label{eq: measure projection}
    \int_{\Omega} \phi(x)d\mathfrak{P}\lambda(x) = \int_{\mathfrak{C}_\Omega} r^2 \phi(x) d\lambda(x, r) \qquad \text{for all test functions }\phi \in C^0(\Omega).
\end{equation}
This characterization elucidates the motivation for the cone construction: the radial coordinate provides an additional degree of freedom that encodes mass variation. In this sense, the projection from cone measures to base measures amounts to averaging over the radial degree of freedom, with the radial coordinate determining how much mass is assigned to each base point.
\begin{rmk}[Choice of notation]
    Throughout the paper, we use the notation $\mathcal{P}_2(\mathfrak{C}_\Omega)$ for the space of finite nonnegative Radon measures on $\mathfrak{C}_\Omega$ with finite second moment with respect to the cone metric. In particular, despite the notation, elements of $\mathcal P_2(\mathfrak C_\Omega)$ are not
    required to have unit total mass. That is,
    \[ \mathcal P_2(\mathfrak C_\Omega) \triangleq \left\{ \lambda \in \mathfrak M_+(\mathfrak C_\Omega) : \int_{\mathfrak C_\Omega} d_{\mathfrak C}(z,\mathfrak o)^2\,d\lambda(z)<\infty \right\}\]
    where $d_{\mathfrak{C}}((x,r), \mathfrak{o}) = r$.
\end{rmk}
We choose this notation to emphasize the fact that the Wasserstein distance $W_{\mathfrak{C}}$ is evaluated only between measures of the same total mass. For each fixed $m > 0$, the space 
\[\mathcal{P}_{2, m}(\mathfrak{C}_\Omega) \triangleq \left\{\lambda \in  \mathcal P_2(\mathfrak C_\Omega)
        : \lambda(\mathfrak C_\Omega)=m\right\}\]
has a geometry equivalent to that of probability measures on the cone under the normalization $\lambda \mapsto \lambda /m$. More precisely, if $\lambda, \eta \in \mathcal P_{2,m}(\mathfrak C_\Omega)$ then $W_{\mathfrak C}^2(\lambda,\eta) = m\,W_{\mathfrak C}^2(\lambda/m,\eta/m)$ and thus the fixed-mass geometry differs from the unit-mass geometry only by a constant rescaling of the Riemannian metric. The differential-geometric objects used in subsequent sections are insensitive to this
constant rescaling. The tangent spaces, viewed as spaces of velocity fields,
are the same after normalization; the weighted Helmholtz projection onto
gradient fields is unchanged; and the Levi-Civita connection and parallel transport are the same for any $m > 0$. Only metric quantities such as squared distances and squared tangent norms, for example, acquire the multiplicative factor $m$. Again, this convention is convenient because the cone lifts of measures in $\mathfrak M_+(\Omega)$
need not themselves be probability measures, even though their projections
satisfy
\[ \int_\Omega \phi(x)\,d\mathfrak P\lambda(x) = \int_{\mathfrak C_\Omega} r^2\phi(x)\,d\lambda(x,r). \]

Now observe that, given a base measure $\mu$, there are infinitely many valid lifts of it to $\mathfrak{C}_\Omega$. \citet{liero2016optimal} show that one can replicate Hellinger-Kantorovich geometry on $\Omega$ through Wasserstein geometry on the cone space $\mathfrak{C}_\Omega$ by minimizing over all possible lifts. 

\begin{thm}[Correspondence of $(\mathcal{P}_2(\mathfrak{C}_\Omega), W_2)$ and $(\mathfrak{M}_+(\Omega), \HK)$, \citet{liero2016optimal} Theorem 3.6]
\label{HK cone variational prob}
    For any measures $\mu_0, \mu_1 \in \mathfrak{M}_+(\Omega)$ we have 
    \begin{equation} \label{eq: cone wasserstein and hk equivalence}
    \HK(\mu_0, \mu_1) = \min\left\{W_{\mathfrak{C}}(\lambda_0, \lambda_1) \,\Big|\, \lambda_i \in \mathcal{P}_2(\mathfrak{C}_\Omega),  \mathfrak{P}\lambda_0 = \mu_0, \mathfrak{P}\lambda_1 = \mu_1\right\}
\end{equation}
where $W_\mathfrak{C}$ is the 2-Wasserstein distance with respect to the cone metric $d_\mathfrak{C}$, which takes the value $\infty$ if $\lambda_0(\mathfrak{C}_\Omega) \neq \lambda_1(\mathfrak{C}_\Omega).$
\end{thm}
This remarkable result allows us to characterize the local metric structure of the Hellinger-Kantorovich space through Wasserstein geometry under a suitable cone lift. However, minimizing over all admissible lifts in \cref{eq: cone wasserstein and hk equivalence} is not computationally tractable. Fortunately, one can leverage \textit{another} equivalent formulation of the Hellinger-Kantorovich distance to obtain explicit lifts that are optimal in the sense of \cref{eq: cone wasserstein and hk equivalence}. In particular, \citet{liero2016optimal} show that the logarithmic entropy transport functional gives us a way to compute optimal lifts in the sense of \cref{eq: cone wasserstein and hk equivalence}. In particular, assume that $\pi \in \mathfrak{M}_+(\Omega \times \Omega)$ is a minimizer of $E(\cdot\,; \mu_0, \mu_1)$ from \Cref{eq: let functional} and for the marginals $\pi_i$ consider the Lebesgue decomposition $\mu_i = u_i\pi_i + \mu_i^\perp$. Then the transport plan $\gamma_{\pi} \in \mathfrak{M}_+(\mathfrak{C}_{\Omega} \times \mathfrak{C}_{\Omega})$ defined by
\begin{multline}\label{eq: optimal lift}
    \gamma_{\pi}(dz_0, dz_1) = \delta_{\sqrt{u_0(x_0)}}(dr_0)\delta_{\sqrt{u_1(x_1)}}(dr_1)\pi(dx_0, dx_1) \\
    + \delta_1(dr_0)\mu_0^\perp(dx_0)\delta_{\mathfrak{o}}(dz_1) + \delta_1(dr_1)\mu_1^\perp(dx_1)\delta_{\mathfrak{o}}(dz_0) 
\end{multline}
and the associated lifts $\lambda_i = \Pi_{\#}^i\gamma_\pi$ are optimal for \cref{eq: cone wasserstein and hk equivalence}. This means we can solve the logarithmic-entropy-transport functional approach and use the optimal plan to construct appropriate lifts of the original measures onto the cone, described in detail in \Cref{alg:let-lift}. This is important, as \Cref{eq: let hk equivalence} is implementable with standard unbalanced optimal transport solvers. We also have the following characterization of the optimal coupling $\pi$ and $\gamma_{\pi}$ under the assumption that we are in the \textit{reaction-transport} regime. 

\begin{rmk}[$\gamma_\pi$ Supported on a Map, \citet{clancy2022wasserstein}]
\label{clancy coupling supported on a map}
    Let $\pi$ be the optimal coupling from \Cref{alg:let-lift}, and suppose that $\mu_0, \mu_1$ are absolutely continuous with respect to Lebesgue and are supported in a set of diameter strictly less than $\pi/2$ (and thus, by \Cref{diameter bound} we know that $\mu_i^\perp = 0$). It then holds that $\pi$ is supported on a map $T$ and the optimal coupling $\gamma_\pi$ of the optimal lifts $\lambda_0$ and $\lambda_1$ is supported on the assignment
    \[(x_0, r_0(x_0)) \mapsto (T(x_0), r_1(T(x_0))) \]
    implying that the radial conditional laws of $\lambda_0$ and $\lambda_1$ are deterministic.
\end{rmk}

\begin{algorithm}[H]
    \begin{algorithmic}[1]
        \caption{Lift measures from  $(\mathfrak{M}_+(\Omega), \HK)$ to $(\mathcal{P}_{2}(\mathfrak{C}_\Omega), W_2)$}
        \label{alg:let-lift}
        \Require Measures $\mu_0, \mu_1 \in \mathfrak{M}_+(\Omega)$.  
        \State Solve logarithmic entropy transport functional minimization,
        \begin{align*} 
            \HK(\mu, \nu) &= \inf\left\{E(\pi;\mu, \nu)^{1/2} \, |\, \pi  \in \mathfrak{M}_+(\Omega \times \Omega), \,\Pi^1_{\#}\pi  \ll \mu, \, \Pi^2_{\#}\pi \ll \nu\right\}
        \end{align*}
        for $E$ defined in \cref{eq: let functional}, and let $\pi^*$ denote the optimal coupling for the problem above. 
        \State Compute the Lebesgue decompositions $\mu_i = u_i\pi_i^* + \mu_i^\perp$.
        \State Lift $\mu_i$ to $\lambda_i$ via $\lambda_i = \Pi^i_{\#}\gamma_{\pi^*}$ with $\gamma_{\pi^*}$ defined in \cref{eq: optimal lift}.
        \State \Return $\lambda_0, \lambda_1  \in \mathcal{P}_{2}(\mathfrak{C}_{\Omega})$. 
    \end{algorithmic}
\end{algorithm}

\noindent \textbf{Geodesic Projections.} Note that earlier indicate that optimal lifts are dependent on the source and target measures $\mu_0, \mu_1 \in \mathfrak{M}_+(\Omega)$ -- this indicates that we need to use caution when constructing algorithms on $\mathfrak{C}_\Omega$ to emulate Hellinger-Kantorovich geometry on $\Omega$. Fortunately, it holds that for a pair $\mu_0, \mu_1$, the $W_\mathfrak{C}$ geodesic between optimal lifts $\lambda_0, \lambda_1$ of $\mu_0, \mu_1$ projects back to a Hellinger-Kantorovich geodesic on $\Omega$. 

\begin{prop}[\citet{liero2016optimal}, Corollary 4.4]
\label{cone geodesics project to hk geodesics}
    For any optimal lift $\lambda_0, \lambda_1$ of $\mu_0, \mu_1$, an optimal plan $\gamma \in \mathcal{P}_2(\mathfrak{C}_\Omega \times \mathfrak{C}_\Omega)$ in the sense of $W_\mathfrak{C}$ induces a geodesic path in $(\mathfrak{M}_+(\Omega), \HK)$ connecting $\mu_0$ and $\mu_1$ given by
    \[\mu_t = \mathfrak{P}\lambda_t \qquad \text{with} \qquad \lambda_t = Z(t; \cdot, \cdot)_\#\gamma\]
    where $Z(\cdot; z_0, z_1)$ is the geodesic interpolator described in \Cref{eq: cone geodesics}.
\end{prop}
This result indicates that we can interpolate along a Hellinger-Kantorovich geodesic by lifting the source and target measures to $\mathfrak{C}_\Omega$, solving for the optimal plan, performing Wasserstein interpolation on $\mathfrak{C}_\Omega$, and projecting back via $\mathfrak{P}$. We provide a example of this interpolation procedure applied to two Gaussian measures of differing total mass in \Cref{fig:interpolation}.
\medskip

\begin{figure}[htbp]
     \centering
     \begin{subfigure}{\textwidth}
         \centering
         \includegraphics[width=\linewidth]{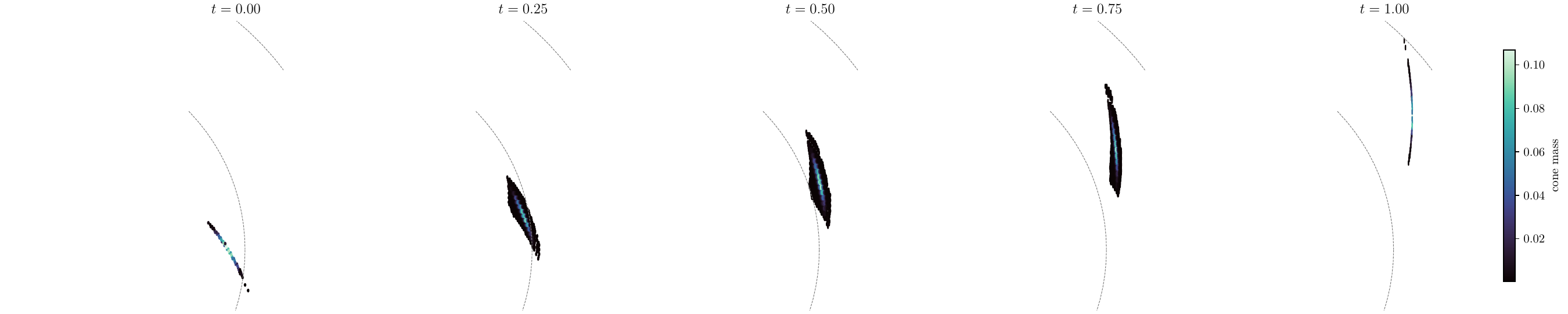}
         \caption{Wasserstein interpolation of $(\lambda_0, \lambda_1) = \operatorname{LETLift}(\mu_0, \mu_1)$ (\Cref{alg:let-lift}) on $\mathfrak{C}_\Omega$  plotted in polar coordinates, where the grey curves denote the radial coordinate lines. Observe that $\lambda_t(r\,|\,x)$ is not deterministic, which prevents the tangent projection operator $\mathcal{P}_{\lambda_t}$ from being an isometry.}
     \end{subfigure}
     
     \vspace{1em} 

     \begin{subfigure}{\textwidth}
         \centering
         \includegraphics[width=\linewidth]{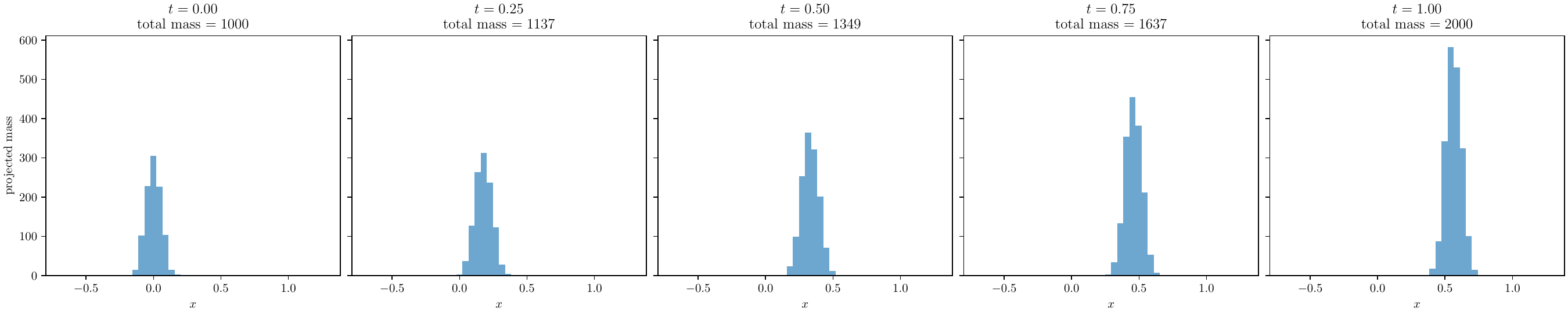}
         \caption{Hellinger-Kantorovich geodesic between $\mu_0, \mu_1$ obtained by projecting $\lambda_t$ from panel (a).}
     \end{subfigure}
     \caption{Visualization of Hellinger-Kantorovich interpolation using the lifting procedure described in \Cref{alg:let-lift}, where grey lines denote radial coordinate lines. The measures $\mu_0, \mu_1$ consist of $n_0 = 1000$ and $n_1 = 2000$ draws from Gaussian probability measures with offset means. Note that we do \textit{not} normalize the empirical measures, and thus $\mu_0(\mathbb{R})\neq \mu_1(\mathbb{R})$.}
     \label{fig:interpolation}
\end{figure}

\noindent \textbf{Lifting Tangent Vector Fields.} Fortunately, lifting tangent vectors in a manner that is faithful to the dynamical formulation of Hellinger-Kantorovich geometry is tractable too. We'll define the vector field lifting operator as follows, 
\begin{align*}
    \mathcal{L}_{\mu, \lambda}: T_{\mu}\mathfrak{M}_+^\Gamma \rightarrow S_{\lambda} \subset T_{\lambda}\mathcal{P}_2(\mathfrak{C}_{\Omega}) \quad \text{given by} \quad \mathcal{L}_{\mu, \lambda}[(v, \beta)](x,r) = (v(x), 2\beta(x)r)
\end{align*}
where $\mathfrak{P}\lambda = \mu$ and $S_{\lambda}$ is a subset of the Wasserstein tangent space at the measure $\lambda$. For the rest of the paper, we may opt to drop the operator subscripts as they do not effect the functional form of the projection -- they merely denote the $L^2$ space in which the lifted objects live. One can see that $\mathcal{L}_{\mu, \lambda}(T_{\mu}\mathfrak{M}_+^\Gamma) \subset T_{\lambda} \mathcal{P}_2(\mathfrak{C}_\Omega)$ by observing that any Hellinger-Kantorovich tangent of the form $(\nabla \varphi, \varphi)$ with $\varphi \in C^\infty_c(U)$ is lifted to a tangent that is a cone gradient field, $(\nabla \varphi(x), 2\varphi(x)r) = \nabla_{\mathfrak{C}}(r^2 \varphi)$, a fact that we formalize in \Cref{lifted image}.
\begin{lemma} \label{lifted image}
    Suppose the support of $\lambda$ is radially supported in the range $[r_{\min}, r_{\max}]$ for some $0 < r_{\min} \leq r_{\max} < \infty$. Then the image of the lifting operator $\mathcal{L}_{\mu, \lambda}$ is a subset of $T_\lambda \mathcal{P}_2(\mathfrak{C}_\Omega)$, i.e. $\mathcal{L}_{\mu, \lambda}(T_\mu \mathfrak{M}_+^\Gamma) \subset T_\lambda \mathcal{P}_2(\mathfrak{C}_\Omega)$. 
\end{lemma}
\begin{proof}
    Define $S_\lambda \triangleq \mathcal{L}_{\mu, \lambda}(T_\mu \mathfrak{M}_+^\Gamma)$ and $u \triangleq (\nabla \varphi, \varphi)$ where $\varphi \in C^\infty_c(U)$ -- observe that $u$ is necessarily a HK tangent. Then the lifted tangent is given by $\mathcal{L}_{\mu, \lambda} [u] = (\nabla \varphi(x), 2\varphi(x)r).$ Now choose a $C^\infty_c((0,\infty))$ bump function $\chi$ such that $\chi \equiv 1$ on $[r_{\min}, r_{\max}]$, and observe that 
    \[\mathcal{L}_{\mu, \lambda}[u] = \nabla_{\mathfrak{C}}(\chi(r) \cdot  r^2\varphi(x)) \qquad \lambda \text{-almost everywhere,}\] rendering $\mathcal{L}_{\mu, \lambda}[u]$ a valid cone Wasserstein tangent. Now we'll consider HK tangents that are $L^2(\mu; T\Omega) \times L^2(\mu)$ limits of fields of the form $(\nabla \varphi, \varphi)$. In particular, let $(v, \beta) = \lim_{n \rightarrow \infty}(\nabla \varphi_n, \varphi_n)$ in  $L^2(\mu; T\Omega) \times L^2(\mu)$. We will show that the cone potential $\Phi_n(x,r) = \chi(r) r^2 \varphi_n(x)$ satisfies $\nabla_{\mathfrak{C}}\Phi_n(x,r) \rightarrow (v(x), 2\beta(x)r)$ in $L^2(\lambda)$. To see this, observe that 
    \begin{align*}
        \nabla_{\mathfrak{C}}\Phi_n(x,r) &= \left(\chi(r) \nabla \varphi_n(x), \chi'(r) r^2 \varphi_n(x) + 2\varphi_n(x) r \chi(r)\right)
    \end{align*}
    and thus
    \begin{align*}
        \left\|\mathcal{L}_{\mu, \lambda}(v,\beta) - \nabla_{\mathfrak{C}}\Phi_n\right\|_{L^2(\lambda)}^2 &= \int_{\mathfrak{C}_\Omega}r^2\|\chi(r) \nabla \varphi_n(x) - v(x)\|_2^2 + \left|\chi'(r) r^2 \varphi_n(x) + 2\varphi_n(x) r \chi(r) - 2\beta(x)r\right|^2 \, d\lambda(x,r) \\
        &= \int_{\mathfrak{C}_\Omega}r^2\|\nabla \varphi_n(x) - v(x)\|_2^2 + 4r^2\left|\varphi_n(x) - \beta(x)\right|^2 \, d\lambda(x,r) \\
        &= \int_\Omega \|\nabla \varphi_n - v\|_2^2 + 4\left|\varphi_n - \beta\right|^2 \, d\mu
        \rightarrow 0.
    \end{align*}
\end{proof}

We note that, as we will show in subsequent sections, the radial support condition required by \Cref{lifted image} follows directly from standard assumptions on the curve of measures under consideration (see \Cref{admissible class} for the exact conditions). With this result in mind, we are now in position to define a complementary vector field \textit{projection} operator that takes cone Wasserstein tangent fields to HK tangent fields. Let $\lambda \in \mathcal{P}_2(\mathfrak{C}_\Omega)$ be a measure with support bounded away from the cone apex $\mathfrak{o}$ and let $u(x,r) = a(x,r) + b(x,r)\partial_r \in L^2(\lambda; \mathfrak{C}_\Omega)$ be vector field over the cone. We'll define the map  
\[\mathcal{P}_{\lambda}: T_{\lambda}\mathcal{P}_2(\mathfrak{C}_\Omega) \rightarrow F_{\mathfrak{P}\lambda} \subset L^2(\mathfrak{P}\lambda; \Omega) \times L^2(\mathfrak{P}\lambda) \]
where $T_{\mathfrak{P}\lambda}\mathfrak{M}_+^\Gamma \subset F_{\mathfrak{P}\lambda}$ as follows. Let 
\[v(x) = \frac{\int_{0}^\infty r^2a(x,r) \, d\lambda(r\,|\,x)}{\int_0^\infty r^2 \,d\lambda(r\,| \,x)} \quad \text{and} \quad \beta(x) = \frac{\int_0^\infty rb(x,r) \,d\lambda(r\, | \, x)}{2\int_0^\infty r^2 \,d\lambda(r\,| \,x)}\]
and define
\[\mathcal{P}_\lambda[u](x) = (v(x), \beta(x)).\]
With this definition in place, we are ready to present the following result relating the lifting operator $\mathcal{L}$ and the projection operator $\mathcal{P}.$
\begin{thm}[Isometry of lifting and projection]
\label{isometry}
Let $\lambda \in \mathcal{P}_2(\mathfrak{C}_\Omega)$, set $\mu \triangleq \mathfrak{P}\lambda$ and suppose the support of $\lambda$ is radially supported in the range $[r_{\min}, r_{\max}]$ for some $0 < r_{\min} \leq r_{\max} < \infty$. Then the lifting operator
\[
\mathcal{L}_{\mu,\lambda} : T_{\mu}\mathfrak{M}_+^\Gamma
\to S_\lambda \subset T_\lambda \mathcal{P}_2(\mathfrak{C}_\Omega),
\qquad
\mathcal{L}_{\mu,\lambda}(v,\beta)(x,r) = (v(x),2\beta(x)r),
\]
is an isometry, where $T_{\mu}\mathfrak{M}_+^\Gamma$ is equipped with the
Hellinger--Kantorovich metric tensor and $S_\lambda \subset T_{\lambda}\mathcal{P}_2(\mathfrak{C}_\Omega)$ is equipped with the $W_2(\mathfrak{C}_\Omega)$ metric tensor. Suppose in addition that the conditional radial law of $\lambda$ given $x$ is deterministic and supported in $(0,\infty)$; equivalently, suppose there exist a Borel measure $\eta$ on $\Omega$ and a measurable map $r:\Omega\to(0,\infty)$ such that $\lambda = (x\mapsto (x,r(x)))_\#\eta.$ Then, for every $u=a+b\,\partial_r \in L^2(\lambda;T\mathfrak{C}_\Omega)$, the projection operator satisfies
\[
\mathcal{P}_\lambda[u](x)
=
\left(a(x,r(x)),\frac{b(x,r(x))}{2r(x)}\right)
\qquad \text{for }\mu\text{-a.e. }x.
\]
Moreover, $\mathcal{P}_\lambda$ is the isometric inverse of $\mathcal{L}_{\mu,\lambda}$, i.e.
\[
\mathcal{P}_\lambda\circ \mathcal{L}_{\mu,\lambda}\Big|_{T_\mu \mathfrak{M}_+^\Gamma}
=
\mathrm{Id}
\quad \text{on } T_\mu \mathfrak{M}_+^\Gamma,
\quad 
\text{and}
\quad
\mathcal{L}_{\mu,\lambda}\circ \mathcal{P}_\lambda \Big|_{S_{\lambda}}
=
\mathrm{Id}
\quad  \lambda\text{-a.e.}
\]
\end{thm}
We provide the proof of \Cref{isometry} in \Cref{sec: proof of isometry}. This result illustrates the following important fact: if we can construct lifts of geodesics where the conditional radial law of the measures are always \textit{deterministic}, then we have an explicit isometry between $T_{\mu_t}\mathfrak{M}_+^\Gamma$ and a subset of $T_{\lambda_t}\mathcal{P}_2(\mathfrak{C}_\Omega)$. In the next section we will describe a general lifting procedure that is optimal in the sense of \Cref{HK cone variational prob} \textit{and} satisfies this deterministic conditional radial law property.

\subsubsection{Method of Characteristics}

In this section we will characterize curves of measures by using the method of characteristics to solve for the Lagrangian paths and masses of individual particles. This will give us the foundation for producing cone lifts that have the deterministic conditional radial law property described above. \Cref{lifting by characteristics one} describes and proves the validity of a lifting procedure based on this principle. 

\begin{prop}[Lifting by Characteristics]
\label{lifting by characteristics one}
    Suppose that the continuity-reaction equation driven by $(v_t, \beta_t)_{t\in[0,1]}$ is uniquely solved by an weakly continuous HK geodesic $(\mu_t)_{t \in [0,1]}$ in $\mathfrak{M}_+^\Gamma(\Omega)$. Moreover, assume that $v_t$ is uniformly bounded on $[0,1] \times \Omega$ and admits a $\mu_0$-a.e. injective flow $X_t$ for all $t$, i.e.
    \[\partial_t X_t(x) = v_t(X_t(x)), \quad X_0(x) = x,\]
    and assume that the map $s \mapsto \beta_s(X_s(x)) $ is in $ L^1(0,1)$ for $\mu_0$ a.e. $x$ and $\beta_s \leq \beta_{\max} < \infty$. Finally, define  define $\eta_t \triangleq (X_t)_\# \mu_0$ and assume that for each $t$ there exists a measureable map $r_t:\Omega \rightarrow (0, \infty)$ unique $\eta_t$-a.e. such that
    \[r_t(X_t(x)) \triangleq \exp\left(2 \int_0^t \beta_s(X_s(x))\,ds\right) \quad \text{$\mu_0$-a.e.}\]
    Then for any Borel $A \subset \Omega$
    \[\mu_t(A) = \int_A r_t^2 \, d\eta_t\]
    and $\lambda_t \triangleq (x \mapsto (x, r_t(x)))_\#\eta_t$ is a valid lift of $\mu_t$, i.e. $\mu_t = \mathfrak{P}\lambda_t$. 
\end{prop}
We provide the proof of \Cref{lifting by characteristics one} in \Cref{sec: proof of lifting by characteristics one}. The result indicates that we can obtain a valid lift of a curve of measures $\mu_t$ by integrating the reaction component $\beta_t$ along the particle-wise Lagrangian paths; this accumulated reaction will encode the radial coordinate of the location that a particle gets mapped to on the cone. Importantly, this lift satisfies the hypotheses of \Cref{isometry} part (b), which guarantee that the vector field projection and lifting operators are \textit{isometric inverses} (when the domain is restricted appropriately). We describe this lifting procedure in \Cref{alg:isometric-lift}, and \Cref{lifting by characteristics two} proves that the proposed lifting procedure is indeed optimal in the sense of \Cref{HK cone variational prob}. To gain intuition for the lifting procedure, consider the particle-wise
interpretation of unbalanced transport. The lift tracks two coupled quantities.
First, it tracks a transported reference measure $\eta_t$, which records the
locations of the particles and is obtained by pushing forward the initial
reference measure $\eta_0$ along the flow generated by the spatial velocity
field $v_t$. Second, it tracks a radial factor $r_t$, which records how much
mass is attached to each transported particle. More precisely, the HK
interpolant satisfies $\mu_t = r_t^2$, or equivalently $d\mu_t/d\eta_t=r_t^2.$ Thus $\eta_t$ describes where the particles move, while $r_t^2$ describes how their masses grow or decay along the flow.

\begin{algorithm}[h]
\caption{Isometric lifting and interpolation procedure}
\label{alg:isometric-lift}
\begin{algorithmic}[1]
\Require Endpoints $\mu_0,\mu_1 \in \mathfrak{M}_+^\Gamma(\Omega)$, discretization level $N$.
\Ensure Discrete lifted path $(\lambda_i)_{i=0}^N$ on $\mathfrak{C}_{\Omega}$ and lifted tangent fields $(V_i)_{i=0}^{N-1}$.

\State Set $t_i=i/N$ for $i\in\{0,\dots,N\}$ and $\Delta t=1/N$.
\State Compute a discrete HK geodesic $(\mu_i)_{i=0}^N$ between $\mu_0$ and $\mu_1$ using \Cref{alg:let-lift}.
\State Initialize $\eta_0 = \mu_0$ and $r_0(y)\equiv 1$.
\State Set $\lambda_0 = (y\mapsto (y,r_0(y)))_\# \eta_0$.

\For{$i=0$ to $N-1$}
    \State Solve the local HK / LET problem between $\mu_i$ and $\mu_{i+1}$ to obtain optimal coupling $\pi$.
    \State Compute Lebesgue decompositions $\mu_{i+j} = u_{i+j} \pi_j$ for $j \in \{0,1\}$ where $\pi_j = \Pi^j_\# \pi$.
    \State Let $T_i$ be the Monge map (i.e. $\pi_1=(T_i)_\#\pi_0$) and write
    $\mu_i=u_i\,\pi_0$, $\mu_{i+1}= u_{i+1} \pi_1.$
    \State Define the local HK logarithmic components
    \[
        v_i(y)=
        \begin{cases}
        \dfrac{1}{\Delta t}\dfrac{T_i(y)-y}{\|T_i(y)-y\|_2}\,
        \sqrt{\dfrac{u_{i+1}(T_i(y))}{u_i(y)}}\,
        \sin\bigl(\|T_i(y)-y\|_2\bigr),
        & T_i(y)\neq y,\\[2ex]
        0, & T_i(y)=y,
        \end{cases}
    \]
    \[
        \beta_i(y)
        =
        \frac{1}{2\Delta t}\left(
        \sqrt{\dfrac{u_{i+1}(T_i(y))}{u_i(y)}}\,
        \cos\bigl(\|T_i(y)-y\|_2\bigr)-1
        \right).
    \]
    \State Define the local radial multiplier (\Cref{radial update})
    \[
        q_i(y)= \sqrt{\frac{u_{i+1}(T_i(y))}{u_i(y)}}.
    \]
    \State Update the transported reference measure
    with $\eta_{i+1}= (T_i)_\#\eta_i.$
    \State Update the radial function on current positions by $r_{i+1}(z)= q_i(T_i^{-1}(z))\,r_i(T_i^{-1}(z)).$ 
    \State Set $\lambda_{i+1}= (z\mapsto (z,r_{i+1}(z)))_\#\eta_{i+1}.$
    \State Define the lifted tangent field $V_i= \mathcal{L}_{\mu_i,\lambda_i}(v_i,\beta_i).$
\EndFor

\State \Return $(\lambda_i)_{i=0}^N$ and $(V_i)_{i=0}^{N-1}$.
\end{algorithmic}
\end{algorithm}

\begin{figure}[htbp]
     \centering
     \begin{subfigure}{\textwidth}
         \centering
         \includegraphics[width=\linewidth]{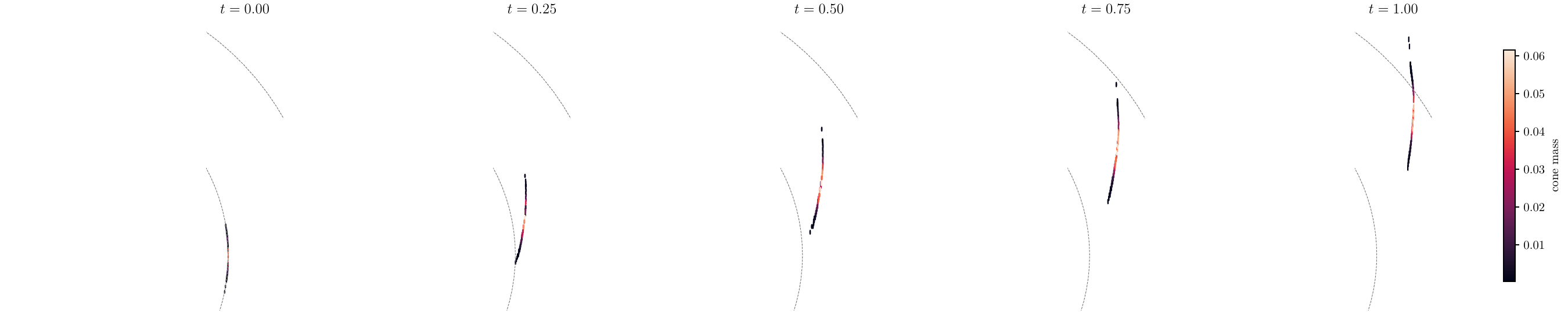}
         \caption{Wasserstein interpolation of $(\lambda_0, \lambda_1) = \operatorname{IsometricLift}(\mu_0, \mu_1)$ (\Cref{alg:isometric-lift}) on $\mathfrak{C}_\Omega$  plotted in polar coordinates. Observe that, unlike in panel (a) of \Cref{fig:interpolation}, the law associated with $\lambda_t(r\,|\,x)$ is deterministic and therefore $\mathcal{P}_{\lambda_t}$ is an isometry between $T_{\lambda_t}\mathcal{P}_2(\mathfrak{C}_\Omega)$ and $S_{\lambda_t} \subset T_{\mu_t}\mathfrak{M}_+^\Gamma$ (\Cref{isometry}).}
     \end{subfigure}
     
     \vspace{1em} 

     \begin{subfigure}{\textwidth}
         \centering
         \includegraphics[width=\linewidth]{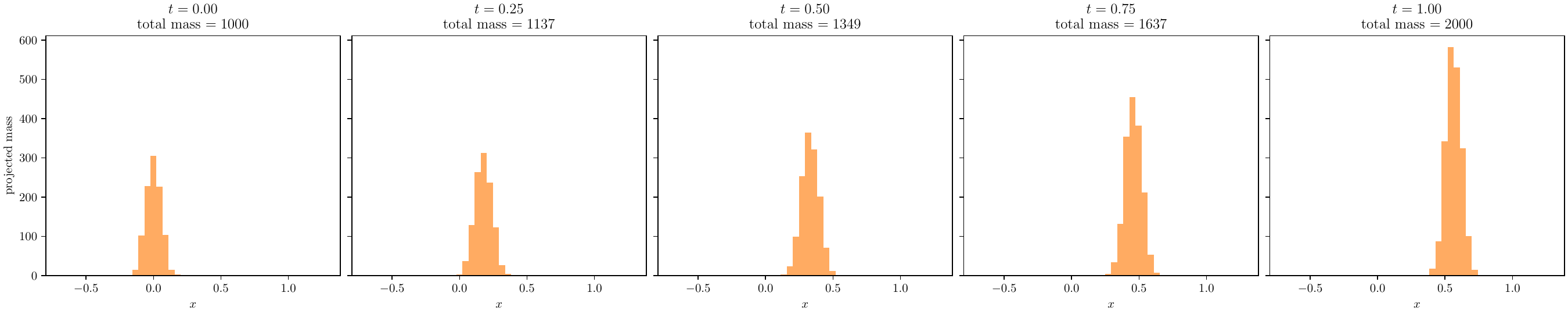}
         \caption{Hellinger-Kantorovich geodesic between $\mu_0, \mu_1$ obtained by projecting $\lambda_t$ from panel (a).}
     \end{subfigure}
     \caption{Visualization of Hellinger-Kantorovich interpolation using the lifting procedure described in \Cref{alg:isometric-lift} under the same settings as in \Cref{fig:cone_pt}. Observe that, unlike the LET lifted interpolant in \Cref{fig:cone_pt}, the conditional radial law of the measures $\lambda_t$ are deterministic for all $t$.}
     \label{fig:isometric_interpolation}
\end{figure}
Fortunately, in practice one does not need to approximate the integral to obtain the radial function $r_t$ along characteristics. Instead, one can simply leverage the following fact related to the logarithmic entropy transport functional problem. 

\begin{prop}[Radial update]
\label{radial update}
    Fix a step $i$ of \Cref{alg:isometric-lift}. Let $\pi$ be the local
    LET minimizer between $\mu_i$ and $\mu_{i+1}$, and write $\pi_j = \Pi^j_\#\pi$ for $j\in\{0,1\}$. Assume that $\pi$ is induced by an $\eta_i$-a.e. injective map $T_i$,
    so that
    \[
        \pi=(\operatorname{id},T_i)_\#\pi_0,
        \qquad
        \pi_1=(T_i)_\#\pi_0.
    \]
    Also suppose that we are in the reaction--transport regime, so that
    \[
        \mu_i=u_i\pi_0,
        \qquad
        \mu_{i+1}=u_{i+1}\pi_1,
    \]
    with $0<u_i<\infty$ on the relevant support. Then the radial coordinate in
    \Cref{alg:isometric-lift} satisfies
    \[
        \frac{r_{i+1}(T_i(x))}{r_i(x)}
        =
        \sqrt{\frac{u_{i+1}(T_i(x))}{u_i(x)}}
        \qquad \text{for $\eta_i$-a.e. $x$}.
    \]
\end{prop}
We provide a proof of \Cref{radial update} in \Cref{sec: proof of radial update}. With this fact, we can now provide an \textit{exact} algorithmic lifting procedure that avoids numerical integration. We provide the algorithm description in  \Cref{alg:isometric-lift}.

\begin{rmk}
    As alluded to, the procedure for obtaining the lifted measures and tangents described in \Cref{alg:isometric-lift} incurs no approximation error on the grid points; in particular, $\lambda_i$, the characteristic values $X_i, r_i$ and the lifted tangents $V_i$ are all exact given population measures $\mu_0, \mu_1$. 
\end{rmk}

Now we will show that this lifting procedure is useful in a very powerful sense. In particular, as is clear from the construction, the lifted measures have the property that the conditional radial law is \emph{deterministic}. As indicated by \Cref{isometry}, we know this ensures that the maps $\mathcal{P}_{\lambda_t}: T_{\lambda_t}\mathcal{P}_2(\mathfrak{C}_\Omega) \rightarrow T_{\mu_t}\mathfrak{M}_+^\Gamma$ and $\mathcal{L}_{\mu_t, \lambda_t}: T_{\mu_t}\mathfrak{M}_+^\Gamma \rightarrow S_{\lambda_t} \subset T_{\lambda_t}\mathcal{P}_2(\mathfrak{C}_\Omega)$ are isometric inverses.

\begin{thm}[Isometry and Optimality]
\label{lifting by characteristics two}
    Suppose the assumptions and definitions of \Cref{lifting by characteristics one} hold, the assumptions of \Cref{isometry} hold for all $t$, and further assume that $r_t(x)$ is nonzero for all $t$ and all $x$. Then the tangent lifting and projection maps \[\mathcal{L}_{\mu_t, \lambda_t}: T_{\mu_t}\mathfrak{M}_+^\Gamma  \rightarrow S_{\lambda_t}\subset T_{\lambda_t}\mathcal{P}_2(\mathfrak{C}_\Omega) \qquad \text{and} \qquad \mathcal{P}_{\lambda_t}: S_{\lambda_t} \rightarrow T_{\mu_t}\mathfrak{M}_+^\Gamma\] are isometric inverses. Moreover, the lifted tangent field $V_t \triangleq \mathcal{L}_{\mu_t, \lambda_t}(v_t, \beta_t)$ and the curve of measures $(\lambda_t)_{t \in [0,1]}$ satisfy the cone continuity equation, where $\lambda_t$ is both a $W_2$ geodesic on $\mathfrak{C}_\Omega$ and is an optimal lift, i.e.
    \[W_\mathfrak{C}(\lambda_0, \lambda_1) = \HK(\mu_0, \mu_1).\] 
\end{thm}
We note that the proof that $\mathcal{L}_{\mu_t, \lambda_t}$ and $\mathcal{P}_{\lambda_t}$ are isometric inverses is fairly straightforward and is a simple corollary of \Cref{isometry}. The proofs of the remaining statements are more challenging, and we therefore defer the proof to \Cref{sec: proof of lifting by characteristics two}. To the best of our knowledge, this result is the first to rigorously establish the existence and a explicit construction of an isometry between the Hellinger-Kantorovich tangent space and a subset of the cone Wasserstein tangent space along lifted Wasserstein geodesics. As we shall see in later sections, this richer connection between the spaces enables the use of tools from Wasserstein geometry when solving for or computing Hellinger-Kantorovich objects. As an example, this connection will allow us to compute parallel transport on the Hellinger-Kantorovich space by computing Wasserstein parallel transport on the cone using the procedure described in \citet{saidi2026wassersteinparalleltransportpredicting} and projecting back. Before instantiating that example, we will demonstrate that the lifted geodesics inherit regularity properties directly from the Hellinger-Kantorovich interpolating geodesic.

\subsection{Lifting Regularity}

In this section we will demonstrate that key regularity properties of the isometrically lifted procedure in \Cref{alg:isometric-lift} are inherited directly from standard assumptions on the HK geodesic $(\mu_t)$ on the base space. To the best of our knowledge, the same cannot be said for the lifted geodesic one obtains by taking the Wasserstein interpolation of the LET lifted endpoints in \Cref{alg:let-lift}. This inheritance of regularity is key, as it ensures that many differential objects and operators exist for lifted geodesics. An example that we will explore is parallel transport, which only exists along sufficiently regular Wasserstein geodesics. 

\begin{asmpt}[Admissible class]
\label{admissible class}
    Let $\Gamma \Subset U \Subset  \Omega^\circ$ with $U$ open and $ \Omega\Subset \mathbb{R}^d$ where $\operatorname{diam}(\Omega)< \pi/2$, and suppose that all measures are in a so-called admissible class $\mathcal{C}\subset \mathfrak{M}_+^\Gamma(\Omega)$, where all measures admit a Lebesgue density. Moreover, assume that all pairs $\mu, \mu' \in \mathcal C$
    admit a $\HK$ geodesic interpolant $\mu_t$ with tangent velocity $(v_t, \beta_t) = (\nabla \varphi_t, \varphi_t)$ such that the following conditions hold:
    \begin{enumerate}[label=\textbf{(A.\arabic*)}, leftmargin=*, align=left]
        \item for every $0 \leq t_1 < t_2 \leq 1$ the Monge map $T$ inducing the LET minimizing coupling $\pi^{t_1, t_2}$ of $(\mu_{t_1},\mu_{t_2})$ is $\Pi^0_\#\pi^{t_1,t_2}$-a.e. injective. 
        \medskip
        \item the tangent velocity family $(v_t, \beta_t)$ satisfies 
        \medskip
        \begin{enumerate}
        \item uniform boundedness of $\beta_t:$  \[-\infty < \beta_{\min} \leq \beta_t \leq \beta_{\max} < \infty \quad  \text{for all $t$.}\]
        \item uniform regularity of $(v_t, \beta_t):$  \begin{align*}\sup_{t \in [0,1]} \left(\|v_t\|_{W^{1,\infty}(U)} + \|\beta_t\|_{W^{1, \infty}(U)}\right) \leq M \quad \text{for some universal $M < \infty.$}
        \end{align*}
        \item square integrability:
        \[\int_0^1\|v_t\|_{L^2(\mu_t)}^2 \, dt < \infty \quad \text{and} \quad \int_0^1 \|\beta_t\|_{L^2(\mu_t)}^2 \, dt < \infty.\]
    \end{enumerate}
    \end{enumerate}
    
\end{asmpt}

The regularity described in \Cref{admissible class} is fairly standard, and we will show that it gives rise to lifted geodesics that have desirable properties. In particular, the diameter bound and absolute continuity with respect to Lebesgue immediately guarantee that $\mu_i^\perp = 0$ and $\pi$ is supported on a measurable map $T$, due to \Cref{clancy coupling supported on a map}. Moreover, we will also show that the boundedness of the reaction potentials $\beta_t$ guarantee that the lifted measures stay uniformly bounded away from the cone apex $\mathfrak{o}$, which implies $\mathfrak{C}_\Omega$ is a smooth Riemannian manifold with boundary.

\begin{prop}[Uniformly Bounded Lifts]
\label{uniformly bounded lifts}
    The isometric lift of the geodesic interpolant $\mu_t$ of any $\mu_0, \mu_1 \in \mathcal{C}$ is uniformly bounded away from the cone apex $\mathfrak{o}$. In particular, for some $r_{\min}, r_{\max}$ independent of $t$ we have $0 < r_{\min} \leq r_t \leq r_{\max} < \infty$ for all $t$. 
\end{prop}
We provide the proof of \Cref{uniformly bounded lifts} in \Cref{sec: proof of uniformly bounded lifts}. This property that the lifted measures stay uniformly bounded away from the cone apex ensures the existence of an open set on the cone where the metric is smooth and the set contains the supports of all lifted measures in the admissible class. This smoothness will be important for importing results from \citet{gigli2012second} regarding Wasserstein geometry on smooth Riemannian manifolds. 
We also have the following result, which guarantees that the lifted tangent velocity fields are spatially regular and integrable. 

\begin{prop}[Uniformly Regular Lifts]
\label{uniformly regular lifts}
    Let $\mu_0, \mu_1 \in \mathcal{C}$ admit a geodesic interpolant $\mu_t$ with tangent velocity $(v_t, \beta_t)$, and let $\lambda_t, V_t$ denote the isometrically lifted geodesic interpolant and its lifted velocity field. Then the lifted geodesic $\lambda_t$ is spatially regular in the sense that 
    \[  \int_0^1 \|V_t\|_{L^2(\lambda_t)}^2 \, dt < \infty  \quad \text{and} \quad\sup_{t\in[0,1]}\operatorname{Lip}_{\mathfrak{C}}(V_t) \leq L\]
    for some universal constant $L$.
\end{prop}

The proof of \Cref{uniformly regular lifts} is provided in \Cref{sec: proof of uniformly regular lifts}. As with uniform boundedness, this uniform spatial regularity of the lifted geodesic tangents will enable us to import results from Gigli's second order analysis on the Wasserstein space. In particular, we will see that this regularity guarantees the existence of parallel transport.

\section{Example: Hellinger-Kantorovich Parallel Transport} \label{sec: HK parallel transport}

As a concrete example of the richer connection between $(\mathfrak{M}_+^\Gamma(\Omega), \HK)$ and $(\mathfrak{C}_\Omega, W_2)$, we will use it to compute Hellinger-Kantorovich parallel transport; in particular, we will use recently developed tools for approximating \textit{Wasserstein} parallel transport, and then project back to the Hellinger-Kantorovich space. Since our lifting procedure yields an isometry of $T_{\mu}\mathfrak{M}_+^\Gamma$ and $S_{\lambda} \subset T_{\lambda}\mathcal{P}_2(\mathfrak{C}_\Omega)$, Wasserstein parallel transport along our lifted curves can be used to approximate Hellinger-Kantorovich parallel transport.  In the first part of this section we will describe the characterization of parallel transport via the covariant derivative. In the latter parts of this section, we will discuss the approximation scheme proposed by \citet{saidi2026wassersteinparalleltransportpredicting} and our proposed instantiation of it to approximate Hellinger-Kantorovich parallel transport.   

\subsection{Exact Parallel Transport via the Covariant Derivative.} 
Intuitively, a vector field along a curve is \say{unchanging} loosely-speaking if its derivative is zero. Thus, in the context of abstract manifolds, a vector field along a curve is the parallel transport of a source vector if its covariant derivative along the curve is zero.
\begin{defn}[\citet{lee2018introduction}]
\label{parallel transport on general manifolds}
    Let $\man$ be a smooth Riemannian manifold. A smooth vector field $X$ along a smooth curve $\gamma$ is said to be parallel along $\gamma$ with respect to the Levi-Civita connection if $\nabla_{\dot \gamma}X \equiv 0$.
\end{defn}
This characterization now allows us to define parallel transport on $(\mathcal{P}_2(\man), W_2)$ and $(\mathfrak{M}_+^\Gamma(\Omega), \HK)$ using the covariant derivative. Consider a smooth curve $\mu_t$ through $\mathcal{P}_2(\man)$ indexed by $t \in (0, 1)$ with the tangent field $\nabla \varphi_t$ driving its dynamics. 

\begin{prop}[Wasserstein Parallel Transport PDE, \citep{gigli2012second, saidi2026wassersteinparalleltransportpredicting}]
\label{Wasserstein parallel transport pde}
    An absolutely continuous (in the sense of \citet{gigli2012second}) tangent vector field $v_t$ along $\mu_t$ is parallel along a regular curve of measures $\mu_t$ (\Cref{strong regularity}) with respect to $\nabla_{(\nabla\varphi_t)}^{W_2}$ if and only if
    \begin{equation*}
        \operatorname{div}_g \left(\mu_t\left(\partial_t v_t + \nabla v_t \cdot \nabla \varphi_t\right)\right) = 0 \qquad \text{for a.e. $t \in (0,1)$}
    \end{equation*}
    in the sense of distributions on $\man$. 
\end{prop}
\noindent \textit{Proof.} Due to \Cref{parallel transport on general manifolds} and \Cref{Wasserstein covariant deriv} we know that $v_t$ is a parallel vector field along the curve $\mu_t$ if \[\Pi_{\mu_t}\left(\partial_tv_t + \nabla^M_{(\nabla \varphi_t)} v_t\right) = 0 \quad \text{ for almost every $t \in (0,1).$}\] This is tantamount to the requirement that $v_t$ solves $\operatorname{div}_g \left(\mu_t(\partial_tv_t + \nabla^M_{(\nabla \varphi_t)} v_t)\right) = 0$ distributionally, since  
\[T^\perp_{\mu_t}\mathcal{P}_2(\man) = \left\{w \in L^2(\mu_t) 
\,\big | \, \operatorname{div}_g(w\mu_t) = 0\right\}\]
as stated in Definition 1.29 of \citet{gigli2012second}.  \qed{}

We instantiate the same idea for the Hellinger-Kantorovich case. For a smooth curve (in some appropriate sense) $\mu_t$ parameterized by $t \in (0,1)$ with velocity field $(\nabla\varphi_t, \varphi_t)$ we show that the following PDE arises from the covariant derivative definition of parallel transport -- we provide the proof in \Cref{sec: proof of HK PDE}.

\begin{prop}[Hellinger-Kantorovich Parallel Transport PDE]
\label{HK parallel transport pde}
    Suppose $\mu_t$ is a curve of measures in the admissible class \Cref{admissible class}, and suppose that $(v_t, \beta_t)$ is a tangent vector field along $\mu_t$ whose isometric lift is absolutely continuous (in the sense of \citet{gigli2012second}). Then $(v_t, \beta_t)$ is parallel along $\mu_t$ with respect to $\nabla_{(\nabla\varphi_t, \varphi_t)}^{\HK}$ if and only if
    \begin{align*}
        -\nabla \cdot \left(\mu_t (\partial_tv_t + \nabla v_t \cdot \nabla \varphi_t + 2\varphi_t v_t + 2\beta_t \nabla\varphi_t)\right)  + 4\left( 
        \partial_t\beta_t + \frac{1}{2}\langle\nabla \beta_t, \nabla\varphi_t\rangle + 2\varphi_t \beta_t
        \right)\mu_t = 0
    \end{align*}
    for a.e. $t \in (0,1)$ in the sense of distributions.
\end{prop}

While \Cref{Wasserstein parallel transport pde} and \Cref{HK parallel transport pde} describe parallel transport along any curve, solving these PDEs in practice may be challenging, especially in high-dimensional settings. To this end, we will describe an alternative approximate approach for computing parallel transport along Hellinger-Kantorovich geodesics. To do this, we will lift the Hellinger-Kantorovich geodesic to the cone and use approximate Wasserstein parallel transport developed by \citet{saidi2026wassersteinparalleltransportpredicting}. The approach from this paper leverages the connection between parallel transport on a smooth, boundaryless and complete Riemannian manifold $\man$ and parallel transport on $\mathcal{P}_{2}(\man)$ -- a connection established by \citet{gigli2012second} -- to enable parallel transport on $\mathcal{P}_2(\man)$.

\subsection{Approximate Hellinger-Kantorovich Parallel Transport}

As alluded to, we will demonstrate the utility of the isometric lifting procedure in \Cref{alg:isometric-lift} by using it to compute approximate Hellinger-Kantorovich parallel transport along geodesics. Our procedure, described in \Cref{alg:HK-parallel-transport}, combines the lifting procedure with approximate Wasserstein parallel transport. As we will show, the fact that the lifting procedure gives an explicit isometry of the Wasserstein and Hellinger-Kantorovich tangent spaces along the entire geodesic, this procedure is indeed valid. 

To formalize the setting, we need to ensure enough regularity on the lifted geodesic of measures for parallel transport to even exist. A sufficient condition is \textit{strong regularity}, defined in \Cref{strong regularity}. 

\begin{defn}[Regular Curves, \citet{gigli2012second}]
\label{strong regularity}
    Let $(\mu_t)_{t \in [0,1]}$ be an absolutely continuous curve of measures in $\mathcal{P}_2(\man)$ where $\man$ is a $C^\infty$, boundaryless, connected and complete manifold. We say that $(\mu_t)$ is regular if its velocity vector field $(v_t)$ satisfies 
    \[\int_0^1\|v_t\|_{L^2(\mu_t)}^2\,dt < \infty \quad \text{and} \quad \int_0^1\operatorname{Lip}^\man(v_t)\, dt < \infty\]
    where $\operatorname{Lip}^\man(v_t)$ denotes the spatial Lipschitz constant of the field $v_t$. Moreover, we say that $(\mu_t)$ is strongly regular if it is regular and 
    \[\sup_{t\in [0,1]}\operatorname{Lip}^\man(v_t) < \infty.\]
\end{defn}

Fortunately, as we showed in \Cref{uniformly regular lifts}, the lifting procedure described in \Cref{alg:isometric-lift} yields a strongly regular geodesic when the endpoint measures $\mu_0$ and $\mu_1$ belong to the admissible class $\mathcal{C}$ defined in \Cref{admissible class}. One more detail to resolve, however, is the fact that the definition and results of \citet{gigli2012second} require $\man$ to be $C^\infty$, boundaryless, connected and complete. The domain $\mathfrak{C}_\Omega$ when restricted to radial coordinates in the range $[r_{\min}, r_{\max}]$ coming from \Cref{uniformly bounded lifts} is indeed a smooth and complete manifold, but it has a topological boundary. Fortunately, we can circumvent this through the following result.  

\begin{thm}[Localized ambient setting for lifted geodesics]
\label{cone ambient manifold}
Let $(\mu_t)_{t\in[0,1]}$ be an admissible HK geodesic with velocity $(v_t, \beta_t)$ and let $(\lambda_t,V_t)$ denote its deterministic isometric lift.
Then there exists an open manifold $\mathcal U$ and a complete smooth metric $\widetilde g_{\mathfrak C}$ on $\mathcal U$
such that:
\begin{enumerate}
    \item $\bigcup_{t\in[0,1]} \operatorname{supp}\lambda_t \Subset \mathcal U$.
    \item $\widetilde g_{\mathfrak C}$ agrees with the cone metric $g_{\mathfrak C}$ on a neighborhood of
    $\bigcup_t \operatorname{supp}\lambda_t$.
    \item the lifted curve $(\lambda_t)$ is a regular curve in $\mathcal P_2(\mathcal U,\widetilde g_{\mathfrak C})$, with velocity field $V_t$.
\end{enumerate}
Consequently, Gigli's parallel transport theory applies to $(\lambda_t)$, and the resulting objects depend only on the original cone geometry along the lifted curve.
\end{thm}

\begin{rmk}
    Strictly speaking, the application of Gigli's theory is made to the normalized curve \(\bar\lambda_t=\lambda_t/m\in\mathcal P_2(\mathcal U)\) where $\lambda_t(\mathcal{U}) \equiv m$. As mentioned in \Cref{lifting}, the relevant tangent spaces, projections, covariant derivative, and parallel transport equation are invariant under this constant rescaling. We therefore write the construction directly for \(\lambda_t\) without changing the resulting transported fields.
\end{rmk}

We provide the proof of \Cref{cone ambient manifold} in \Cref{sec: proof of cone ambient manifold}. This result explicitly constructs a complete boundaryless Riemannian manifold with an interior compact subset whose geometry coincides exactly with that of the cone restricted to the support of all lifted measures in the admissible class. The construction is simple: take an open subset of $\mathfrak{C}_\Omega$ that subsumes the support the entire lifted interpolant and choose a \textit{complete} Riemannian metric that coincides with the $\mathfrak{C}_\Omega$ metric on its interior (which provably exists). Thus, we will consider the geometry of $\mathcal{P}_2(\mathcal{U}, \tilde g_{\mathfrak{C}})$ in the following results to ensure compatibility with Gigli's theory. We now recall the following approximation result, which establishes that one can approximate $\mathcal{P}_2(\man)$ parallel transport by parallel transporting the tangent field along the Lagrangian paths on $\man$.

\begin{corollary}[\citet{gigli2012second}]
\label{approximation scheme}
    Let $(\man, g)$ be a smooth, complete and boundaryless Riemannian manifold, and suppose $
    \mu_0, \mu_1 \in  \mathcal{P}_{2}(\man)$ are probability measures that admit a strongly regular geodesic interpolant $\mu_t$ with a tangent velocity field $V_t$. For any $v \in T_{\mu_0}\mathcal{P}_2(\man)$ it holds that 
    \[\left\|\Pi_{\mu_t}(\PT^\man_{\mu_0 \rightarrow \mu_t}(v)) - {\PT}_{\mu_0 \rightarrow \mu_t}^{W_2, \man}(v)\right\|_{L^2(\mu_t)} \leq  \left(e^{\int_0^1{\operatorname{Lip}}(V_s)\,ds} - 1\right)^2\|v\|_{L^2(\mu_0)}\left(\int_0^t{\operatorname{Lip}}(V_s)\,ds\right)^2\]
    where $\PT^\man_{\mu_0 \rightarrow \mu_t}(v)$ denotes the pointwise parallel transport (with respect to $\man$) of $v$ along the Lagrangian path defined by the $\mathcal{P}_2(\man)$ geodesic between $\mu_0$ and $\mu_1$. Moreover, the strong regularity of $\mu_t$ guarantees that $\Pi_{\mu_t}(\PT^\man_{\mu_0 \rightarrow \mu_t}(v))$ approximates ${\PT}_{\mu_0 \rightarrow \mu_t}^{W_2, \man}(v)$, the Wasserstein parallel transport of $v$ along the geodesic $\mu_t$, in the sense that
    \[\left\|\Pi_{\mu_t}(\PT^\man_{\mu_0 \rightarrow \mu_t}(v)) - {\PT}_{\mu_0 \rightarrow \mu_t}^{W_2, \man}(v)\right\|_{L^2(\mu_t)} \leq Ct^2 \quad \text{for some $C>0$ independent of $t$.}\]
\end{corollary}
\begin{rmk}
    Note that if $\mu_0(\man)= \mu_1(\man) = m\neq 1$ then the constant $C$ picks up a dependence on $m$, but of course the statement is the same. 
\end{rmk}

\citet{saidi2026wassersteinparalleltransportpredicting} leverage this result to devise an approximation scheme that circumvents solving the parallel transport PDE. In particular, given an approximation resolution $N$, the approximate achieves an $O(N^{-1})$ error in an $L^2$ sense. 
\begin{corollary}[\citet{saidi2026wassersteinparalleltransportpredicting}]
\label{approximation scheme full geodesic}
    Let $v \in T_{\mu_0}\mathcal{P}_2(\man)$, let $N \in \mathbb{N}$, define $s = 1/N$ and define the maps \[\widehat \PT_k \triangleq \left(u\mapsto  \Pi_{\mu_{ks}}(\PT^\man_{\mu_{(k-1)s} \rightarrow \mu_{ks}}(u))\right)\quad \text{with $u \in T_{\mu_{(k-1)s}}\mathcal{P}_2(\man)$ for $k \in \{1, \dots, N\}$}.\]
    Then we have the following approximation bound,
    \[\left\|\left(\widehat \PT_N \circ \cdots \circ \widehat \PT_1\right)(v) - \PT_{\mu_0 \rightarrow \mu_1}^{W_2, \man}(v)\right\|_{L^2(\mu_1)} = O(N^{-1}).\]
\end{corollary}
Thus, one can circumvent solving the Wasserstein parallel transport PDE in \Cref{Wasserstein parallel transport pde} through this approach, which instead requires computation of parallel transport on $\man$. This is typically much easier, as it usually amounts to solving an ODE -- in some cases, including ours, parallel transport is available in closed form. With this Wasserstein parallel transport approximation result, it seems natural to compute Hellinger-Kantorovich parallel transport by lifting to $\mathfrak{C}_\Omega$ and doing Wasserstein parallel transport there. The validity of this idea is formalized below in \Cref{parallel transport approx entire geodesic}. But before that, we need the following result.

\begin{prop}[Pullback connection along a deterministic lift]
\label{prop: pullback connection}
Let $(\mu_t)_{t\in[0,1]}$ be an admissible HK geodesic with tangent field
$(\nabla\varphi_t,\varphi_t)$, and let $(\lambda_t)_{t\in[0,1]}$ be its
deterministic isometric lift. Write
\[
\mathcal L_t \equiv \mathcal L_{\mu_t,\lambda_t},
\qquad
\mathcal P_t \equiv \mathcal P_{\lambda_t},
\qquad
V_t \equiv \mathcal L_t(\nabla\varphi_t,\varphi_t).
\]
For any  HK tangent field $\mathbf u_t = (\nabla \psi_t, \psi_t)$ along $(\mu_t)$ such that $\psi_t \in C^{2,1}(\Omega \times (0,1))$,
define the pullback connection
\[
\widetilde\nabla_{(\nabla\varphi_t,\varphi_t)}\mathbf u_t
\;\triangleq\;
\left(\Pi_{\mu_t} \circ \mathcal{P}_t \right)\!\left(
\mathbf{D}_t^{W_2, \mathcal{U}}\big(\mathcal L_t\mathbf u_t\big)
\right)
\]
where $\mathbf{D}_t^{W_2, \mathcal{U}}$ is the Wasserstein \textit{total} derivative from \Cref{Wasserstein covariant deriv}. Then $\widetilde\nabla_{(\nabla \varphi_t, \varphi_t)} \mathbf{u}_t = \nabla ^{\HK}_{(\nabla \varphi_t, \varphi_t)}\mathbf{u}_t$.
\end{prop}

We provide the proof of \Cref{prop: pullback connection} in \Cref{sec: pullback connection equals HK}. With this result, we are now in a position to establish the validity of our proposed procedure, which we explicitly write out in \Cref{alg:HK-parallel-transport}. Moreover, we describe the approximation error bound in \Cref{parallel transport approx entire geodesic}. We remark that Step 6 of \Cref{alg:HK-parallel-transport} differs from the
approximation scheme of \citet{saidi2026wassersteinparalleltransportpredicting}
only in that the Wasserstein tangent projection
$\Pi_{\lambda_{t_{k+1}}}$ is not applied explicitly. This omission does not
affect the resulting update: as shown in \Cref{supplemental equivalence}, the
subsequent projection onto $S_{\lambda_{t_{k+1}}}$ satisfies
\[
    \Pi^S_{t_{k+1}}
    =
    \Pi^S_{t_{k+1}}\circ \Pi_{\lambda_{t_{k+1}}},
\]
and therefore already removes any component orthogonal to the Wasserstein
tangent space.

\begin{algorithm}[h]
\caption{Approximate HK parallel transport via cone transport}
\label{alg:HK-parallel-transport}
\begin{algorithmic}[1]
\Require Endpoints $\mu_0, \mu_1 \in \mathfrak{M}_+^\Gamma(\Omega)$, source tangent
$\mathbf u_0 \in T_{\mu_0}\mathfrak{M}_+^\Gamma$, approximation resolution $N$.
\Ensure Approximation
$\widehat{\mathbf u}_{1,N}$ of
$\PT^{\mathrm{HK}}_{\mu_0\rightarrow\mu_1}(\mathbf u_0)$.
\State Set $\Delta t = \frac{1}{N}$ and $
t_k = ks$ for $k=0,\dots, N.$
\State Compute lifts $\left((\lambda_i)_{i = 0}^{N-1}, (V_i)_{i = 0}^{N-1}\right) = \operatorname{IsometricLift}(\mu_0, \mu_1, N)$ via \Cref{alg:isometric-lift}.
\State Lift the source tangent:
\[
U_{0} = \mathcal{L}_{\mu_0,\lambda_0}[\mathbf u_0]
\in T_{\lambda_0}\mathcal P_2(\mathcal U).
\]
\For{$k=0$ to $N-1$}
    \State Let $T_k \triangleq (z \mapsto \exp^\mathfrak{C}_z(\Delta t V_k))$ be the lifted map sending
    $\lambda_{t_{k}}$ to $\lambda_{t_{k+1}}$.
    \State Form local approximation described in \Cref{approximation scheme} ,
    \[ \widetilde U_{k+1} = \left((x,r) \mapsto \PT^{\mathcal{U}}_{(x,r) \rightarrow T_k(x,r)}\left(\widehat  U_{k}\right) \right)\]
    \hspace{\algorithmicindent}using cone parallel transport (see \Cref{sec: cone parallel transport}).
    \State Project onto the lifted subspace $S_{\lambda_{t_{k+1}}} = \mathcal{L}_{\mu_{t_{k+1}}, \lambda_{t_{k+1}}}(T_{\mu_{t_{k+1}}}\mathfrak{M}_+^\Gamma)$ with 
    \[\widehat U_{k+1} = \left(\mathcal{L}_{\mu_{t_{k+1}}, \lambda_{t_{k+1}}} \circ \Pi_{\mu_{t_{k+1}}} \circ\mathcal{P}_{\lambda_{t_{k+1}}}\right)(\widetilde U_{k+1}).\]
\EndFor

\State Map back to the HK tangent space with $\widehat{\mathbf u}_{1, N}
=
\mathcal{P}_{\lambda_1}\big(\widehat U_{N}\big).$

\State \Return $\widehat{\mathbf u}_{1, N}$.
\end{algorithmic}
\end{algorithm}

\begin{thm}[Approximation of HK parallel transport by cone transport]
\label{parallel transport approx entire geodesic}
    Let $(\mu_t)_{t\in[0,1]}$ be an admissible HK geodesic with tangent velocity $\mathbf{v}_t = (\nabla \varphi_t, \varphi_t)$ and let
    $(\lambda_t,V_t)$ denote its deterministic isometric lift. Let $\mathbf u_0 \in T_{\mu_0}\mathfrak{M}_+^\Gamma$ and let
    $\widehat{\mathbf u}_{1, N}$ be the output of
    \Cref{alg:HK-parallel-transport} with input $\mathbf u_0$, and assume that the map
    \[t \mapsto \left(\PT^{W_2, \mathcal{U}}_{\lambda_t \rightarrow \lambda_0} \circ \,\mathcal{L}_{t} \circ \Pi_{\mu_t} \circ \mathcal{P}_t \circ \PT^{W_2, \mathcal{U}}_{\lambda_0 \rightarrow \lambda_t}\right) \quad \text{ is $C^2$ in operator norm.}\]
    Then
    \[
    \left\|
    \widehat{\mathbf u}_{1, N}
    -
    \PT^{\mathrm{HK}}_{\mu_0\rightarrow\mu_1}(\mathbf u_0)
    \right\|_{T_{\mu_1}\mathfrak{M}_+^\Gamma}
    = O(N^{-1}).
    \]
\end{thm}

We provide a proof of \Cref{parallel transport approx entire geodesic} in \Cref{proof of parallel transport approx entire geodesic}. We remark that the procedure requires one to compute parallel transport on the cone space $\mathfrak{C}_\Omega$. In \Cref{sec: cone parallel transport} we derive the covariant derivative on $\mathfrak{C}_\Omega$ and we use it to compute closed form parallel transport equations. Our derivation utilizes the theory of \textit{warped-product} manifolds -- product manifolds whose metric tensor consists of a function of the metric tensors of the factors -- where the geometric objects of the product can be readily obtained from the individual factors.  

\subsection{Simulations}

We will now show some simulations to demonstrate the properties of parallel curves of measures in the Hellinger-Kantorovich space. In particular, we will see that parallel curves of measures have intuitive properties when it comes to the evolution of their total mass and their central moments, but they have an unintuitive dependence on positional changes. We provide details regarding our numerical simulations and implementation in \Cref{sec: implementation details}. 
\\

\par{\textbf{Parallel Reaction.}} Our experimental setup is as follows. We consider three measures $\mu_1, \mu_2, \mu_3 \in \mathfrak{M}_+(\Omega)$ where $M_i \triangleq \mu_i(\Omega)$, $\text{diam}(\Omega) < \pi/2$ and we compute the tangent $\mathbf{u} \triangleq (v, \beta)$ that pushes $\mu_2$ to $\mu_3$, i.e. $\mu_3 = \mathbf{exp}_{\mu_2}(\mathbf{u}).$ Then, we compute $\mathbf{u}^* = \PT^{\HK}_{\mu_2 \rightarrow \mu_1}(\mathbf{u})$, the approximate geodesic parallel transport of $\mathbf{u}$ from $\mu_2$ to $\mu_1$, and push $\mu_1$ in the direction $\mathbf{u}^*$, defining the destination measure $\mu_4$. From a geometric perspective, we are tracing out two geodesic curves in $(\mathfrak{M}_+, \HK)$ with parallel initial velocities and studying the properties of their destination measures. In these experiments we sample $M_i$ samples from the normalized measures $\bar\mu_i = \mu_i/M_i$ (which are regarded as probability measures) and assign each sample a mass of $1$. We remark that the experiments that follow incur discretization error, and therefore act as a proxy for the true phenomenona underlying the geometry. Having said that, we believe that the approximation error is small enough for the experiments to provide intuition.

\begin{figure}[h]
     \centering
     \begin{subfigure}{\textwidth}
         \centering
         \includegraphics[width=\linewidth]{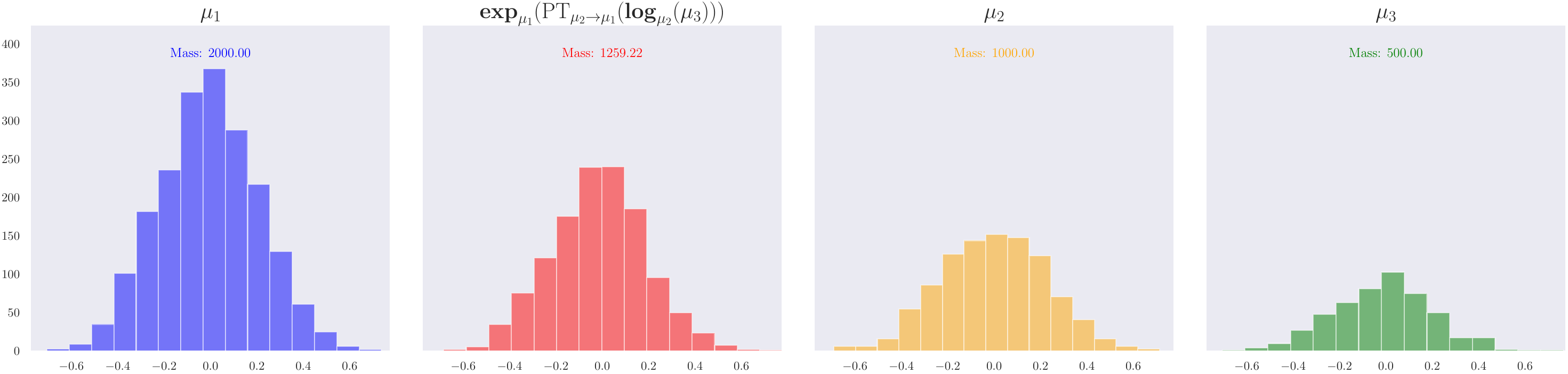}
         \caption{The measures $\mu_i = \sum_{i = 1}^{M_i}\delta_{x_i}$ where $x_i\sim \bar \mu_i = N(0, 2)$ (rescaled to have diameter $< \pi/2$), where $M_1 = 2000, M_2 = 1000, M_3 = 500$. }
     \end{subfigure}
     
     \vspace{1em} 

     \begin{subfigure}{\textwidth}
         \centering
         \includegraphics[width=\linewidth]{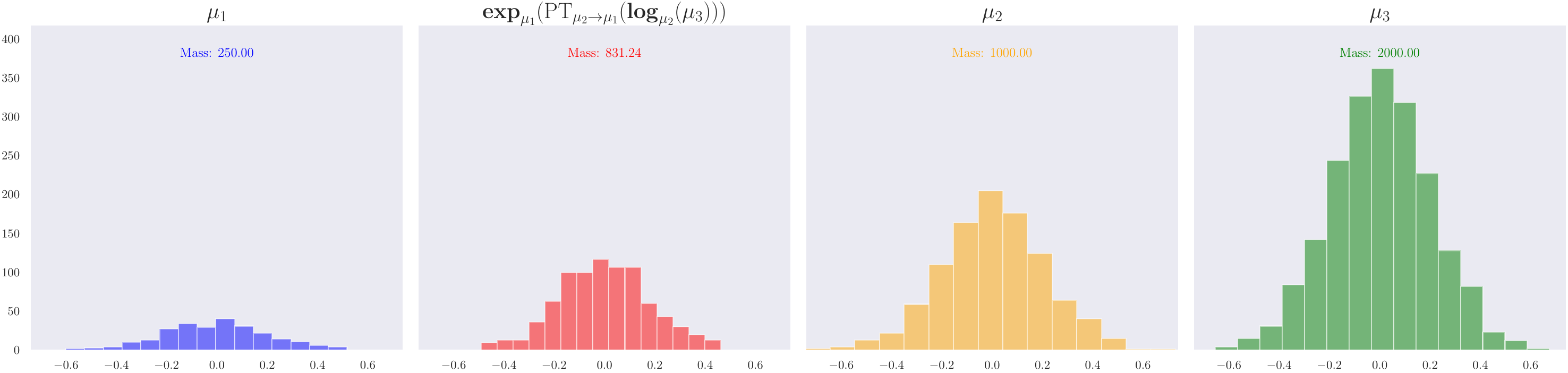}
         \caption{The measures $\mu_i = \sum_{i = 1}^{M_i}\delta_{x_i}$ where $x_i\sim \bar \mu_i = N(0, 2)$ (rescaled to have diameter $< \pi/2$), where $M_1 = 250, M_2 = 1000, M_3 = 2000$. }
     \end{subfigure}
     \caption{Visualization of parallel HK geodesics where $\mu_1, \mu_2, \mu_3$ consist of a variable number of samples from $N(0,2)$. We see that geodesics in $(\mathfrak{M}_+, \HK)$ with parallel initial velocities experience mass change in the same direction (i.e. growth or shrinkage). }
     \label{fig:mass_growth_death}
\end{figure}

In \Cref{fig:mass_growth_death} we see that Hellinger-Kantorovich geodesics with parallel initial velocities indeed experience analogous mass changes, as one might have guessed. However, the mass does not scale proportionally, and a heuristic calculation supports this. In particular, in the example in \Cref{fig:mass_growth_death} one might expect the tangent pushing $\mu_2$ to $\mu_3$ to be of the form $\mathbf{log}_{\mu_2}(\mu_3) = (0, \beta_{23})$ for some constant $\beta_{23}$ in the limit. The results in the experiment also suggest that the parallel transported vector $\PT_{\mu_2 \rightarrow \mu_1}^{\HK}(\mathbf{log}_{\mu_2}(\mu_3))$ takes the same form $(0, \beta^{\PT})$ for some constant $\beta_{\PT}$. By the isometry of parallel transport the norms of these two tangents should coincide, which implies the equation
\[\beta_{\PT} = \beta_{23}\sqrt{\frac{M_2}{M_1}}.\]
Using \Cref{HK exponential map closed form}, one would find that panel (a) of \Cref{fig:mass_growth_death} should have 
\[\beta_{23} = \frac{1}{2}\left(\sqrt{\frac{1}{2}} - 1\right) \implies \beta_{\PT} = \frac{1}{2\sqrt{2}}\left(\sqrt{\frac{1}{2}} - 1\right) \implies \mu_4(\Omega) = 2000\left(\frac{2}{2\sqrt{2}}\left(\sqrt{\frac{1}{2}} - 1\right) + 1\right)^2 \approx 1257.4\]
which is almost exactly what we see. The same calculation for panel (b) reveals that
\[\beta_{23} = \frac{1}{2}\left(\sqrt{2} - 1\right) \implies \beta_{\PT} = \left(\sqrt{2} - 1\right) \implies \mu_4(\Omega) = 250\left(2\sqrt{2} - 1\right)^2 \approx 835.8\]
yielding the same conclusion. Thus, parallel curves of measures that experience only changes in \textit{mass} change in the same direction, but not with the same magnitude. \\

\par{\textbf{Parallel Moments.}} We now present simulations illustrating the behavior of low-order moments along parallel measure-valued evolutions. The picture is somewhat less intuitive here, especially when the measures have nonstationary means. This is illustrated in \Cref{fig:mean_shift}, where a shift in location from $\mu_1$ to $\mu_2$ is transported in parallel from $\mu_1$ to $\mu_4$, producing a transformation that involves (1) a change in location, (2) a change in total mass, and (3) a change in the second central moment. As a sanity check, one may verify that $\HK(\mu_1,\mu_4)=\HK(\mu_2,\mu_3)$, which must hold due to the isometry of parallel transport. Fortunately, this strange behavior becomes much more natural when viewed through the geometry of the lifted measures on the cone, which clarifies the parallel geodesics in \Cref{fig:mean_shift}. \Cref{fig:cone-topdown-pt} illustrates the mechanism using a single particle in $\mathbb{R}$ lifted to the cone. The figure shows that a displacement vector $u_p$ that is purely spatial can be parallel transported to a vector at the destination that has both spatial and radial components. In other words, a tangent representing only spatial displacement may transport to a tangent of the same norm whose effect combines spatial displacement with mass variation. This is precisely the phenomenon underlying \Cref{fig:mean_shift}.  

\begin{figure}[h]
    \centering
    \includegraphics[width=0.6\linewidth]{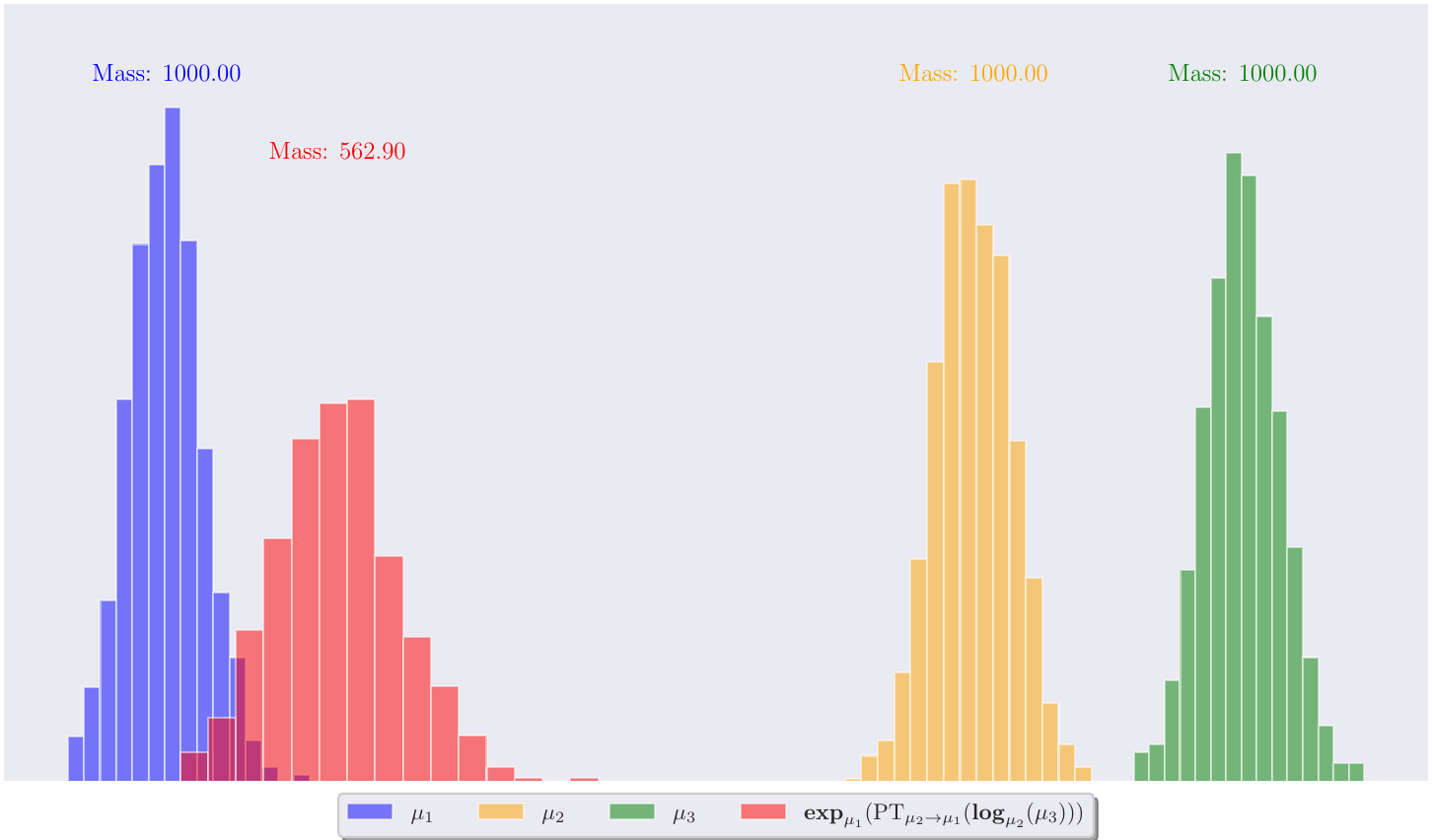}
    \caption{Visualization of parallel HK geodesics where $\mu_1, \mu_2, \mu_3$ consist of $1000$ samples from Gaussian distributions with a fixed variance and varying means. We see that parallel geodesics in $(\mathfrak{M}_+, \HK)$ with parallel initial velocities might \textit{not} experience similar changes in location.}
    \label{fig:mean_shift}
\end{figure}

\begin{figure}[h]
    \centering
    \includegraphics[width=\textwidth]{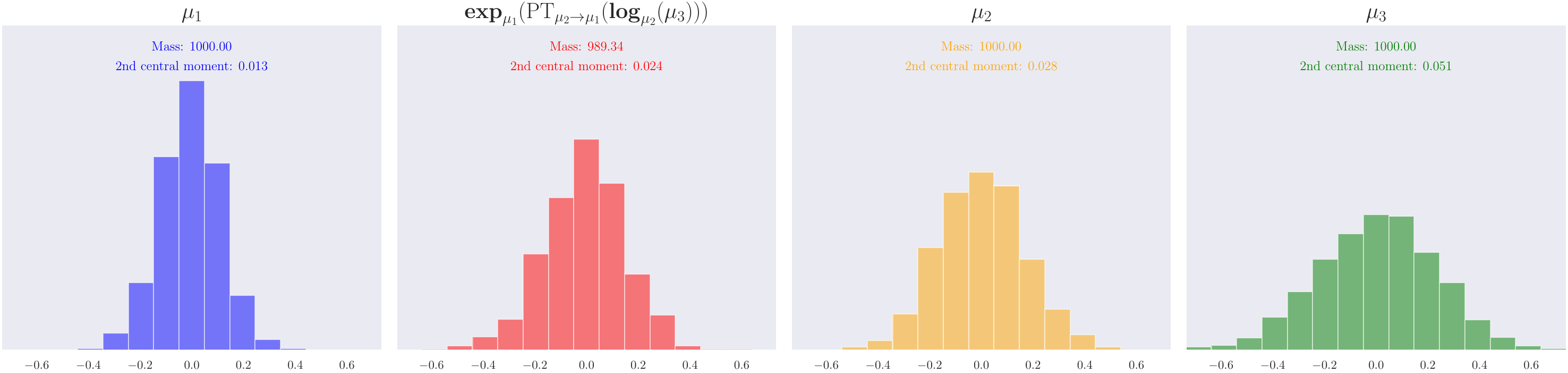}
    \caption{Visualization of parallel HK geodesics where $\mu_1, \mu_2, \mu_3$ consist of $1000$ samples from Gaussian distributions with a fixed mean and differing variances. Parallel geodesics in $(\mathfrak{M}_+, \HK)$ with parallel initial velocities experience similar changes in their second central moment.}
    \label{fig:cov_change}
\end{figure}

\begin{figure}[h]
    \centering
    \begin{tikzpicture}[
        >=Latex,
        line cap=round,
        line join=round,
        point/.style={circle,fill=black,inner sep=1.2pt},
        spatial/.style={->,thick,gray},
        radial/.style={->,thick,gray},
        total/.style={->,very thick,black},
        geodesic/.style={orange!85!black,very thick},
        helper/.style={densely dotted,gray},
        ring/.style={gray!55,densely dashed}
    ]

    \def\R{3.2}            
    \def\rcommon{2.0}      
    \def\thetap{170}       
    \def\thetaq{45}        

    \def\Ltot{0.95}        

    \pgfmathsetmacro{\phiPT}{\thetap - 90}

    \pgfmathsetmacro{\RadComp}{\Ltot*cos(\phiPT-\thetaq)}
    \pgfmathsetmacro{\SpaComp}{\Ltot*cos(\phiPT-(\thetaq+90))}

    \coordinate (O) at (0,0);

    \path[use as bounding box] (-\R,0) rectangle (\R,\R);

    \draw[thick] (-\R,0) arc[start angle=180,end angle=0,radius=\R];
    \draw[thick] (-\R,0) -- (\R,0);

    \draw[ring] (-1.0,0) arc[start angle=180,end angle=0,radius=1.0];
    \draw[ring] (-2.0,0) arc[start angle=180,end angle=0,radius=2.0];

    \node[right] at (\R,0) {$r$};

    \coordinate (P) at (\thetap:\rcommon);
    \coordinate (Q) at (\thetaq:\rcommon);

    \node[point,label={[above left=-1pt]$p$}] at (P) {};
    \node[point,label={[below right=-1pt]$q$}] at (Q) {};

    \draw[helper] (O) -- (P);
    \draw[helper] (O) -- (Q);

    \draw[geodesic] (P) -- (Q)
        node[midway,above] {$\gamma$};

    \coordinate (Psp) at ($(P)+({\Ltot*cos(\phiPT)},{\Ltot*sin(\phiPT)})$);
    \draw[->, very thick, black] (P) -- (Psp)
        node[
            overlay,
            above left=4pt,
            line width=0.4pt,
            draw=black,
            fill=white,
            inner sep=2pt
        ] {$u_p \in \mathrm{span}\{\partial_\theta\}$};

    \coordinate (Qtot) at ($(Q)+({\Ltot*cos(\phiPT)},{\Ltot*sin(\phiPT)})$);

    \coordinate (Qrad) at ($(Q)+({\RadComp*cos(\thetaq)},{\RadComp*sin(\thetaq)})$);
    \coordinate (Qsp)  at ($(Q)+({\SpaComp*cos(\thetaq+90)},{\SpaComp*sin(\thetaq+90)})$);

    \draw[radial] (Q) -- (Qrad)
        node[midway,right=4pt] {$b_q\,\partial_r$};

    \draw[spatial] (Q) -- (Qsp)
        node[midway,above=2pt,left=4pt] {$a_q$};

    \draw[->, very thick, black] (Q) -- (Qtot)
        node[
            overlay,
            above right=2pt,
            line width=0.4pt,
            draw=black,
            fill=white,
            inner sep=2pt
        ] {$\mathrm{PT}^{\mathfrak{C}}_{p \rightarrow q}(u_p)$};

    \draw[helper] (Qrad) -- (Qtot);
    \draw[helper] (Qsp) -- (Qtot);

    \def\rx{1.5}
    \def\thetaxstart{0}
    \def\thetaxend{45}
    
    \draw[->, thick, black]
        (\thetaxstart:\rx) arc[start angle=\thetaxstart, end angle=\thetaxend, radius=\rx]
        node[midway, right=3pt, fill=white, inner sep=1pt] {$x$};

    \end{tikzpicture}
    \caption{Parallel transport on the cone visualized in polar coordinates, where $x=\theta$. A purely spatial tangent vector $u_p$ at $p$ is parallel transported to a vector at $q$ with both spatial and radial components. Thus, a purely spatial displacement may transport to a displacement that combines motion in space with mass variation.}
    \label{fig:cone-topdown-pt}
\end{figure}
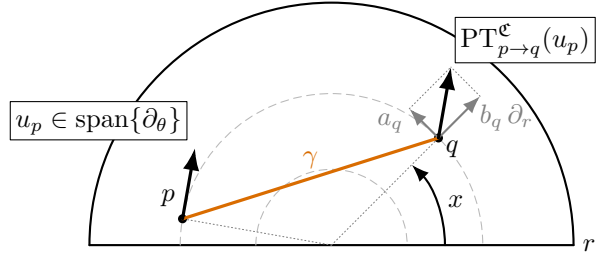

As a final experiment, we isolate the behavior of the generalized variance (that is, the second central moment) along geodesic curves of measures with parallel initial velocities; see \Cref{fig:cov_change}. As expected, a tangent vector that approximately doubles the second central moment of $\mu_2$ has a similar effect when parallel transported to $\mu_1$. At the same time, we observe a small change in total mass that we do not believe is due to discretization error. Rather, we attribute it to the same mechanism seen in \Cref{fig:mean_shift}: when spatial variation is parallel transported along a spatially varying geodesic, part of the tangent may be converted from spatial to radial variation. \Cref{fig:cone-topdown-pt} provides intuition for this effect. In \Cref{fig:cov_change}, the induced mass variation is small because the underlying spatial displacement is also small.

\section{Conclusion}

In this work, we further developed the theory and tractability of Hellinger--Kantorovich geometry by strengthening its connection with Wasserstein geometry on the associated metric cone. In particular, we gave an explicit construction of a lifted Wasserstein geodesic along which the Hellinger--Kantorovich Riemannian geometry is preserved. We then used this construction to compute parallel transport in Hellinger--Kantorovich space via recently developed tools from Wasserstein geometry. In the course of this development, we also derived the covariant derivative and closed-form parallel transport equations on the metric cone $\mathfrak{C}_\Omega$. While these results further clarify the structure of Hellinger--Kantorovich geometry, many important open problems remain. We hope that the framework developed here helps enable statistical and data-analytic methods for studying time-varying stochastic systems with nonconstant total mass. More broadly, tractable computational tools for Hellinger--Kantorovich geometry may support the development of statistical models and inferential procedures for distribution-level dynamics. One promising application area is genomics, where unbalanced optimal transport has already been used to infer gene-level trajectories of cellular populations during development \citep{schiebinger2019optimal}. Finally, we believe this work opens an interesting direction for future research on gradient flows and optimization on spaces of measures, including possible Hellinger--Kantorovich analogues of Riemannian optimization methods that rely on parallel transport.

\section*{Acknowledgments}
The authors would like to thank Larry Wasserman and Dejan Slep\v{c}ev for their helpful feedback and suggestions on the manuscript and theorem statements. The authors would also like to thank Maxfield Parson-Scherban for helpful discussions and suggestions regarding early versions of the work.
\clearpage
\appendix

\section{Notation} \label{sec:notation}

\begin{center}
\small
\renewcommand{\arraystretch}{1.18}
\setlength{\tabcolsep}{4pt}

\begin{longtable}{%
>{\centering\arraybackslash}p{0.30\textwidth}
>{\raggedright\arraybackslash}p{0.64\textwidth}
}
\caption{Notation used throughout the paper.}
\label{tab:notation}\\

\toprule
\textbf{Symbol} & \textbf{Definition} \\
\midrule
\endfirsthead

\toprule
\textbf{Symbol} & \textbf{Definition} \\
\midrule
\endhead

\midrule
\multicolumn{2}{r}{\emph{continued on next page}}\\
\endfoot

\bottomrule
\endlastfoot

\multicolumn{2}{l}{\textbf{Ambient spaces, measures, and metrics}}\\
\midrule

$\man$
& A smooth Riemannian manifold used for the abstract Wasserstein background. \\

$g$, $d_{\man}$, $\operatorname{vol}_g$
& The Riemannian metric tensor, induced distance, and Riemannian volume measure on $\man$. \\

$\Omega$, $\Omega^\circ$, $\partial\Omega$
& Compact convex subset of $\mathbb R^d$, its interior, and its boundary. \\

$\Gamma \Subset U \Subset \Omega^\circ$
& Localized support region $\Gamma$ and an open neighborhood $U$ compactly contained in $\Omega^\circ$. \\

$\mathcal P_2(\man)$
& Probability measures on $\man$ with finite second moment. \\

$\mathcal P_2(\mathfrak C_\Omega)$
& In this paper, finite nonnegative Radon measures on the cone with finite second moment; not necessarily probability measures. \\

$\mathcal P_{2,m}(\mathfrak C_\Omega)$
& Cone measures of fixed total mass $m>0$. \\

$\mathfrak M_+(\Omega)$
& Finite nonnegative Radon measures on $\Omega$. \\

$\mathfrak M_+^\Gamma(\Omega)$
& Measures in $\mathfrak M_+(\Omega)$ supported in $\Gamma$. \\

$\mu,\nu,\mu_t$
& Base-space measures, usually elements of $\mathfrak M_+(\Omega)$ or $\mathcal P_2(\man)$ depending on context. \\

$\lambda,\lambda_t$
& Cone lifts of base measures; typically $\mathfrak P\lambda_t=\mu_t$. \\

$\eta_t$
& Transported reference measure in the characteristic lift, usually $\eta_t=(X_t)_\#\mu_0$. \\

$W_2$
& Quadratic Wasserstein distance on a base manifold. \\

$W_{\mathfrak C}$
& Quadratic Wasserstein distance on the cone using the cone metric $d_{\mathfrak C}$. \\

$\HK$
& Hellinger--Kantorovich distance. \\

$\Gamma_{\mu,\nu}$
& Set of couplings between $\mu$ and $\nu$. \\

$T_\#\mu$
& Pushforward of $\mu$ by a measurable map $T$. \\

$\Pi^j_\#\pi$, $\pi_j$
& The $j$th marginal of a transport plan $\pi$. \\

$\delta_x$
& Dirac mass at $x$. \\

$\ll$, $\perp$
& Absolute continuity and singularity of measures. \\

\midrule
\multicolumn{2}{l}{\textbf{Dynamic formulations and tangent spaces}}\\
\midrule

$v_t$
& Wasserstein velocity field along a curve $\mu_t$; in HK, the transport component of a tangent field. \\

$\beta_t$
& HK reaction component along a curve $\mu_t$. \\

$(v_t,\beta_t)$
& General HK velocity--reaction field. \\

$(\nabla\varphi_t,\varphi_t)$
& Potential representation of an HK tangent field. \\

$|\dot\mu_t|$
& Metric derivative of a curve of measures. \\

$T_\mu\mathcal P_2(\man)$
& Wasserstein tangent space at $\mu$, defined as the $L^2(\mu)$-closure of smooth gradient fields. \\

$T_\mu\mathfrak M_+^\Gamma$
& HK tangent space at $\mu$, defined as the $L^2(\mu)\times L^2(\mu)$-closure of fields $(\nabla\varphi,\varphi)$ with $\varphi\in C_c^\infty(U)$. \\

$\langle\cdot,\cdot\rangle_\mu$
& Riemannian metric tensor on the relevant measure tangent space. For HK, $\langle s_1,s_2\rangle_\mu=\int(\langle v_1,v_2\rangle+4\beta_1\beta_2)\,d\mu$. \\

$\Pi_{\mu_t}$
& Orthogonal projection onto the tangent space at $\mu_t$; context determines whether this is the Wasserstein or HK tangent projection. \\

$\mathbf D_t^{W_2}$
& Wasserstein total derivative along a measure curve. \\

$\mathbf D_t^{\HK}$
& HK total derivative along a measure curve. \\

$\nabla_{(\nabla\varphi_t)}^{W_2}$
& Wasserstein covariant derivative along a curve driven by $\nabla\varphi_t$. \\

$\nabla_{(\nabla\varphi_t,\varphi_t)}^{\HK}$
& HK covariant derivative along a curve driven by $(\nabla\varphi_t,\varphi_t)$. \\

$\operatorname{div}_g$, $\nabla\cdot$
& Riemannian divergence and Euclidean divergence. \\

$\nabla^\man$
& Levi-Civita connection on $\man$. \\

\midrule
\multicolumn{2}{l}{\textbf{Static, dynamic, and logarithmic HK objects}}\\
\midrule
$\pi$, $\pi^*$
& LET transport plan and an optimal LET transport plan. \\

$\mu_i=u_i\pi_i+\mu_i^\perp$
& Lebesgue decomposition of $\mu_i$ with respect to the LET marginal $\pi_i$; here $u_i$ is the Radon--Nikodym factor. \\

$\gamma_\pi$
& Cone transport plan induced by a LET optimizer $\pi$. \\

$\lambda_i=\Pi^i_\#\gamma_\pi$
& Endpoint cone lifts induced by $\gamma_\pi$. \\

$\mu_i^\perp$
& Singular part of $\mu_i$ relative to the LET marginal $\pi_i$. \\

$T$
& Monge map supporting an optimal LET coupling in the reaction--transport regime. \\

\midrule
\multicolumn{2}{l}{\textbf{Exponential and logarithmic maps}}\\
\midrule

$\mathbf{exp}_\mu$, $\mathbf{log}_\mu$
& Measure-space exponential and logarithmic maps. \\

$\exp_x$, $\log_x$
& Base-space Riemannian exponential and logarithmic maps at $x$. \\

$T_{\mu\to\nu}$
& Brenier or Brenier--McCann map from $\mu$ to $\nu$. \\

$\operatorname{id}$
& Identity map on $\mathbb R^d$. \\

$\overline{\log}$
& Normalized logarithmic vector. \\

\midrule
\multicolumn{2}{l}{\textbf{Cone geometry}}\\
\midrule

$\mathfrak C_\Omega$
& Metric cone over $\Omega$. \\

$\mathfrak o$
& Cone apex. \\

$z=(x,r)$
& Cone point with base coordinate $x$ and radial coordinate $r$. \\

$d_{\mathfrak C}$
& Cone distance. \\

$g_{\mathfrak C}$
& Cone metric tensor. \\

$\widetilde g_{\mathfrak C}$
& Complete smooth ambient metric agreeing with the cone metric near the lifted supports. \\

$\mathcal U$
& Smooth ambient manifold used to realize the cone geometry near the lifted curve. \\

$\partial_r$
& Canonical radial vector field on the cone. \\

$\nabla_{\mathfrak C}$
& Cone gradient. \\

$\exp^{\mathfrak C}_{z_0}$, $\log^{\mathfrak C}_{z_0}$
& Cone exponential and logarithmic maps at $z_0$. \\

$r_{\min},r_{\max}$
& Uniform lower and upper radial bounds for lifted measures. \\

\midrule
\multicolumn{2}{l}{\textbf{Measure lifts, tangent lifts, and projections}}\\
\midrule

$\mathfrak P$
& Measure projection from cone space to base space, defined by integrating $r^2\phi(x)$ against the cone measure. \\

$\mathcal L_{\mu,\lambda}$
& Tangent lifting operator from HK space to cone Wasserstein space. \\

$\mathcal L_{\mu,\lambda}(v,\beta)(x,r)=(v(x),2\beta(x)r)$
& Explicit formula for lifting an HK tangent field to the cone. \\

$\mathcal P_\lambda$
& Tangent projection operator from cone vector fields back to base HK fields. \\

$S_\lambda$
& Lifted tangent subspace $\mathcal L_{\mu,\lambda}(T_\mu\mathfrak M_+) \subset T_\lambda\mathcal P_2(\mathfrak C_\Omega)$. \\

$F_{\mathfrak P\lambda}$
& Space containing the image of the tangent projection $\mathcal P_\lambda$. \\

$u=a+b\partial_r$
& Generic cone vector field, with base component $a$ and radial component $b$. \\

$\lambda(r\,|\,x)$
& Conditional radial law of a cone measure given the base coordinate $x$. \\

\midrule
\multicolumn{2}{l}{\textbf{Parallel transport and approximation}}\\
\midrule

$\PT$
& Parallel transport operator. \\

$\PT^\man_{\mu_0\to\mu_t}$
& Pointwise parallel transport on the base manifold along Lagrangian paths. \\

$\PT^{W_2,\man}_{\mu_0\to\mu_t}$
& Wasserstein parallel transport along a curve in $\mathcal P_2(\man)$. \\

$\PT^{\HK}_{\mu_0\to\mu_1}$
& HK parallel transport along a curve in $\mathfrak M_+(\Omega)$. \\

$\operatorname{Lip}^\man(v_t)$
& Spatial Lipschitz constant of a vector field on $\man$. \\

$\operatorname{Lip}_{\mathfrak C}(V_t)$
& Spatial Lipschitz constant of a vector field on the cone. \\

$\widehat\PT_k$
& One-step projected ambient-transport approximation operator in the Wasserstein approximation scheme. \\

$\mathcal L_t$, $\mathcal P_t$
& Shorthand for $\mathcal L_{\mu_t,\lambda_t}$ and $\mathcal P_{\lambda_t}$. \\

$\mathbf u_t=(\nabla\psi_t,\psi_t)$
& HK tangent field along $\mu_t$ used in the pullback-connection argument. \\

$U_t=\mathcal L_t\mathbf u_t$
& Lift of the HK tangent field $\mathbf u_t$ to the cone. \\

$\widetilde\nabla$
& Pullback connection obtained by lifting to the cone, taking a Wasserstein total derivative, projecting back, and applying $\Pi_{\mu_t}$. \\

$\Pi_t^S$
& Projection onto the lifted subspace, defined by $\Pi_t^S=\mathcal L_t\circ\Pi_{\mu_t}\circ\mathcal P_t$. \\

$\overline U(t)$
& Pullback of a lifted tangent to the fixed tangent space at $\lambda_0$. \\

$\overline\Pi_t^S$
& Pulled-back lifted-subspace projection at time $t$. \\

$\widehat U_k$, $\widetilde U_{k+1}$
& Approximate lifted tangent iterate and its unprojected one-step transported version. \\

$\widehat{\mathbf u}_{1,N}$
& Final approximate HK parallel transport output after $N$ steps. \\

$T_k$
& Lifted map in the HK parallel-transport algorithm, usually $z\mapsto\exp_z^{\mathfrak C}(\Delta t\,V_k)$. \\

$T_{t\to t+h}$
& Lifted short-time map sending $\lambda_t$ to $\lambda_{t+h}$. \\

\midrule
\multicolumn{2}{l}{\textbf{Closed-form cone parallel transport}}\\
\midrule

$I\times_\rho \man$
& Warped product with warping function $\rho$. \\

$\mathfrak L(\man)$, $\mathfrak L(I)$
& Vertical and horizontal lifted vector fields in the warped-product language. 

\end{longtable}
\end{center}

\section{Parallel Transport on the Cone} \label{sec: cone parallel transport}

\subsection{The Covariant Derivative on the Cone}

In order to obtain a closed form expression for the parallel transport on $\mathfrak{C}_\Omega$, we first need to derive the covariant derivative. To obtain this object, we will appeal to the theory of \textit{warped-product} manifolds. In particular, consider the nonempty interior $\Omega^{\circ}$ of $\Omega$. We will
take $\Omega^{\circ}$ to be the so-called fiber of our cone construction, as the theory of warped product manifolds typically assumes no topological boundary. We will treat cone tangent vectors as $(a, b) \in T_{(x,r)}\mathfrak{C}_{\Omega^{\circ}}$ where $a \in T_x\Omega^{\circ} \cong \mathbb{R}^d$ and $b \in \mathbb{R}$. Equivalently, we can write a tangent vector as $a + b\partial_r$ where $\partial_r$ is the canonical radial vector field. 
 
\begin{defn}[Warped product, \citet{petersen2006riemannian}]
    Given a Riemannian metric $(\man, g)$, a warped product (over $I$) is defined as a metric on $I \times \man$, where $I \subset \mathbb{R}$ is an open interval with metric 
    \[g = dr^2 + \rho^2(r)g\]
    where $\rho > 0$ on $I$. This is sometimes denoted $I \times_{\rho} \man$, where $I$ is called the base space and $\man$ is called the fiber.
\end{defn}
Having established this definition, it becomes clear that we can view $\mathfrak{C}_{\Omega^{\circ}} \triangleq \Omega^{\circ} \times (r_{\min} - \varepsilon, r_{\max} +\varepsilon)$ as a warped product with $I = (r_{\min} - \varepsilon, r_{\max} +\varepsilon)$ for some $0 < r_{\min} < r_{\max} <\infty $ and some $\varepsilon > 0$ sufficiently small as the base space and $(\Omega^{\circ}, g_{\mathbb{R}^d})$ as the fiber. It also follows from \Cref{def: metric cone} that we should take $\rho(r) = r$. With this in mind, we can used the theory of warped products to study the cone metric and needed geometric objects on it. But before moving on, we need to establish the \say{lift} of a vector field on a manifold to a vector field on a product manifold. Note that we will henceforth denote $T\man$ as the space of all vector fields over the manifold $\man$. 

\begin{defn}[Lifting vector fields, \citet{o1983semi} Definition 7.33]
    Let $\pi: \mathcal{N} \times \man\rightarrow \mathcal{N}$ be the projection onto the first factor $\pi(p, q) = p$. 
    \begin{enumerate}
        \item If $x \in T_p\mathcal{N}$ and $q \in \mathcal{M}$ then the lift $\tilde{x}$ of $x$ to $(p,q)$ is the unique vector in $T_{(p,q)}(\mathcal{N} \times q)$ such that $d\pi(\tilde{x}) = x$. We denote the set of all such horizontal lifts $\mathfrak{L}(\mathcal{N}).$
        \item If $A \in T\mathcal{N}$, then the lift of $A$ to $\mathcal{N} \times \man$ is the vector field $B$ whose values at each $(p,q)$ is the lift of $A_p$ to $(p,q).$
    \end{enumerate}
    Vertical lifts are defined in the same way but using the projection onto the second factor $\sigma: \mathcal{N} \times \man \rightarrow \man$. The set of all vertical lifts are denoted $\mathfrak{L}(\man).$ 
\end{defn}

Now we are in a position to appeal to key results in the theory of warped product manifolds. In particular, said results will allow us to express important geometric objects on the warped product as augmented versions of that of the fiber and the base space. 

\begin{prop}[\citet{o1983semi} Proposition 7.35]
    Let $\man$ and $I$ be Riemannian manifolds with connections denoted $\nabla^{\man}$ and $\nabla^{I}$ respectively. On $I \times_\rho \man$, if $X, Y \in \mathfrak{L}(\man)$ and $V, W \in \mathfrak{L}(I)$ then
    \begin{enumerate}
        \item $\nabla_VW \in \mathfrak{L}(I)$ is the horizontal lift of $\nabla^{I}_VW$ on $I$.
        \item $\nabla_VX = \nabla_XV = \left(\frac{V\rho}{\rho}\right)X$ where $V\rho$ denotes the derivative of $\rho$ in the direction of $V$.
        \item $\nabla_XY = \nabla_X^{\man}Y - \left(\frac{\langle X, Y\rangle}{\rho}\right)\nabla\rho$ where $\nabla \rho$ is the Riemannian gradient of $\rho$.
    \end{enumerate}
\end{prop}
\noindent Now we can directly apply this proposition to compute the connection on $\mathfrak{C}_\Omega \setminus \mathfrak{o}.$ 

\begin{corollary}[The Covariant Derivative on $\mathfrak{C}_{\Omega^{\circ}}$]
\label{corollary: connection on cone}
    Let $X, Y \in \mathfrak{L}(\Omega^{\circ})$ and let $V, W \in \mathfrak{L}((r_{\min} - \varepsilon, r_{\max} + \varepsilon))$. In particular, let $X(x) = \sum_{i = 1}^d X^i(x)\partial_{x_i}$ and $Y(x) = \sum_{i = 1}^d Y^i(x)\partial_{x_i}$ be the expression of $X$ and $Y$ in standard coordinates and let $V = v(r)\partial_r$ and $W = w(r)\partial_r$. Also let $\nabla^{\mathfrak{C}}$ denote the covariant derivative on $\mathfrak{C}_{\Omega^{\circ}}$ and let $(x,r) \in \mathfrak{C}_{\Omega^{\circ}}$. It holds that 
    \begin{enumerate}
        \item $\nabla_V^{\mathfrak{C}}W$ is given by 
        \[\nabla_V^{\mathfrak{C}}W(x,r) = v(r)w'(r)\partial_r.\]
        \item $\nabla_V^{\mathfrak{C}}X(x,r) = \nabla_X^{\mathfrak{C}}V(x,r) = \frac{v(r)}{r}X(x)$.
        \item Finally,
        \[\nabla_X^{\mathfrak{C}}Y(x,r) = \sum_{i = 1}^dX^i(x) \partial_{x_i}Y(x)  - r\langle X(x), Y(x) \rangle\partial_r \]
    \end{enumerate}
    where $\langle\cdot, \cdot\rangle$ is taken to be the Euclidean inner product. 
\end{corollary}
\noindent \textit{Proof.} The proof of (1) is trivial as it follows from the definition of the directional derivative and the horizontal lift. For (2) observe that we have
\[V\rho = v(r)\frac{d}{d r}r = v(r)\]
which implies the result. For (3), the fiber component is simply the Euclidean directional derivative of $Y$ in the direction of $X$, while the base component expression follows from the fact that $\nabla\rho = \partial_r$ and $\langle X, Y \rangle_{\mathfrak{C}} = r^2\langle X, Y\rangle$.  \qed{}

With this definition of the covariant derivative, we can now revisit the geodesic equations described in \cref{eq: cone geodesics}. In particular, we can use the covariant derivative to obtain an important geodesic invariant that we will leverage in our efforts to compute parallel transport along cone geodesics. 

\begin{prop}[Geodesic characterizations]
\label{cone geodesic equations via covariant deriv}
    Let $\gamma: [0, 1] \rightarrow \mathfrak{C}_{\Omega^{\circ}}$ be a constant speed geodesic, denoted $\gamma(t) = \left(p(t), r(t)\right)$.
    Then the geodesic equations on the cone are 
    \[ \ddot p + 2\frac{\dot r \dot p}{r} = 0 \quad \text{and}\quad \ddot r  - r\|\dot p\|_2^2 = 0. \]
    Moreover, for some constant vector $q$ independent of $t$ we have that $r^2(t)\dot p(t) = q$ for all $t \in [0, 1]$.
\end{prop}
\noindent \textit{Proof.} We can write $\dot \gamma$ as $\dot\gamma(t) =  v(t) + \dot r(t) \partial_r$ where $v(t) \triangleq \dot p(t) = \sum_{i = 1}^n\dot p_i(t) \partial_{x_i}$ By the definition of geodesics, we know that $\nabla^{\mathfrak{C}}_{\dot \gamma}\dot \gamma = 0$. Expanding this, we see that the equivalence is tantamount to $\nabla^{\mathfrak{C}}_{v + \dot r\partial_r}(v + \dot r \partial_r) = 0.$ Expanding further yields
\begin{align*}
    \nabla^{\mathfrak{C}}_{v + \dot r\partial_r}(v + \dot r \partial_r) &= \nabla^{\mathfrak{C}}_vv + \nabla^{\mathfrak{C}}_v \dot r \partial_r + \dot r \nabla^{\mathfrak{C}}_{\partial_r}v +  \nabla^{\mathfrak{C}}_{\dot r\partial_r}(\dot r \partial_r).
\end{align*}
Now we will apply \Cref{corollary: connection on cone} to obtain expressions for each of the terms. In particular, applying result 3 yields
\[\nabla_{v}^{\mathfrak{C}}v = \sum_{i = 1}^d\dot p_i \partial_{x_i}\dot p_i - r\|v\|_2^2 \partial_r = \dot v - r\|v\|_2^2\partial_r.\]
For the second term, product rule and result 2 indicate that 
\[\nabla_v^{\mathfrak{C}} (\dot r\partial_r) = v(\dot r) + \dot r \nabla_v^{\mathfrak{C}}\partial_r = 0 + \dot rv/r.\]
Similarly, for the third term we have $\dot r \nabla^{\mathfrak{C}}_{\partial_{r}}v = \dot rv/r$. Finally, for the last term, the Leibniz rule and result 1 imply 
\[\nabla^{\mathfrak{C}}_{\dot r\partial_r}(\dot r\partial_r) = [(\dot r\partial_r)(\dot r )]\partial_r +\dot r \nabla^{\mathfrak{C}}_{\dot r\partial _r}\partial_r = \ddot r \partial_r + \dot r^2\nabla^{\mathfrak{C}} _{\partial_r}\partial_r = \ddot r\partial_r.\]
Combining everything results in the system of equations $\dot v + 2\frac{\dot rv}{r} = 0$ and $\ddot r  - r\|v\|_2^2 = 0$ which completes the proof of the first result. The second result follows from differentiating $r^2v$ yielding $2r\dot rv + r^2\dot v.$ Solving for $\dot v$ above and plugging in implies that $\frac{d}{dt}r^2v = 0$. \qed{}

\subsection{Parallel Transport}

Now let $u_t = a_t + b_t\partial_r$ be a vector field along $\gamma: [0,1] \rightarrow \Omega^\circ$, where
$a_t \in T_{p(t)}\Omega^\circ \simeq \mathbb{R}^d$ and $b_t \in \mathbb{R}$. Writing
$v_t \triangleq \dot p(t)$, we seek a field satisfying
\[
\nabla^\mathfrak{C}_{\dot \gamma}u_t = 0.
\]
Since $a_t$ and $b_t$ are only defined along $\gamma$, one can use arbitrary smooth extensions and then restrict back to the curve; the resulting expression is independent of the chosen extensions. Applying \Cref{corollary: connection on cone} gives
\begin{align*}
    \nabla^\mathfrak{C}_{v_t + \dot r_t\partial_r}(a_t + b_t\partial_r)
    &= \nabla^\mathfrak{C}_{v_t}a_t
       + \dot r_t\nabla^\mathfrak{C}_{\partial_r}a_t
       + \nabla^\mathfrak{C}_{v_t}(b_t\partial_r)
       + \dot r_t\nabla^\mathfrak{C}_{\partial_r}(b_t\partial_r) \\
    &= \Big(\dot a_t - r_t\langle a_t, v_t\rangle \partial_r\Big)
       + \frac{\dot r_t}{r_t}a_t
       + \dot b_t\,\partial_r
       + \frac{b_t}{r_t}v_t.
\end{align*}
Hence
\[
\nabla^\mathfrak{C}_{\dot \gamma}u_t
=
\left(\dot a_t + \frac{\dot r_t}{r_t}a_t + \frac{b_t}{r_t}v_t\right)
+
\left(\dot b_t - r_t\langle a_t, v_t\rangle\right)\partial_r.
\]
Therefore $u_t$ is parallel along $\gamma$ if and only if
\begin{equation}\label{eq: cone PT system}
    \dot a_t + \frac{\dot r_t}{r_t}a_t + \frac{b_t}{r_t}v_t = 0,
    \qquad
    \dot b_t - r_t\langle a_t, v_t\rangle = 0.
\end{equation}
We now solve this system explicitly in \Cref{cone PT explicit}, and we provide an example of parallel transport on $\mathfrak{C}_{\Omega^\circ}$ in \Cref{fig:cone_pt}. The figure highlights the fact that when $\Omega^\circ \subset \mathbb{R}$ then the cone geometry in polar coordinates coincides with Euclidean geometry. 

\begin{prop}[Explicit parallel transport along cone geodesics]
\label{cone PT explicit}
    Let $\gamma(t) = (p(t), r(t))$ be a constant speed geodesic on $\mathfrak{C}_{\Omega^\circ}$ with speed $s$, and let
    $u_0 = a_0 + b_0\partial_r$ be an initial tangent vector at $\gamma(0)$.
    Then the unique parallel transport $u_t = a_t + b_t\partial_r$ along $\gamma$ is given as follows. Note that \Cref{cone geodesic equations via covariant deriv} indicates that the quantity $q \triangleq r_t^2 v_t$ is independent of $t$, so set $c \triangleq \|q\|_2$.
    \begin{enumerate}
        \item If $q = 0$ (equivalently, $v_t \equiv 0$), then
        \[
        a_t = \frac{r_0}{r_t}a_0,
        \qquad
        b_t = b_0.
        \]
        \item If $q \neq 0$, let $e \triangleq q/\|q\|_2$, decompose $a_0 = a_0^\perp + \alpha_0 e,$ $\alpha_0 \triangleq \langle a_0, e\rangle,$ and $a_0^\perp \perp e,$ 
    and define
    \[\theta(t) \triangleq \arctan\left(\frac{s^2t + r_0\dot r_0}{c}\right)
    - \arctan\left(\frac{r_0\dot r_0}{c}\right).
    \]
    Then
    \begin{align}
        a_t &= \frac{r_0}{r_t}a_0^\perp
        + \frac{r_0\alpha_0\cos\theta(t) - b_0\sin\theta(t)}{r_t}\,e, \quad \text{and} \quad 
        b_t= r_0\alpha_0\sin\theta(t) + b_0\cos\theta(t).
        \label{eq: cone PT explicit a}
    \end{align}
    \end{enumerate}
\end{prop}
\noindent \textit{Proof.}
If $q=0$, then $v_t \equiv 0$, so \Cref{eq: cone PT system} reduces to
\[
\dot a_t + \frac{\dot r_t}{r_t}a_t = 0,
\qquad
\dot b_t = 0,
\]
which yields
\[
a_t = \frac{r_0}{r_t}a_0,
\qquad
b_t = b_0.
\]
Now assume $q \neq 0$, and write
\[
v_t = \frac{q}{r_t^2} = \frac{c}{r_t^2}e.
\]
Decompose $a_t = \alpha_t e + w_t $ with $w_t \perp e.$ Substituting this into \Cref{eq: cone PT system} and separating the $e$ and $e^\perp$ components gives
\begin{align}
    \dot w_t + \frac{\dot r_t}{r_t}w_t &= 0, \label{eq: cone PT w}\\
    \dot\alpha_t + \frac{\dot r_t}{r_t}\alpha_t + \frac{c}{r_t^3}b_t &= 0, \label{eq: cone PT alpha}\\
    \dot b_t - \frac{c}{r_t}\alpha_t &= 0. \label{eq: cone PT b}
\end{align}
Equation \eqref{eq: cone PT w} immediately yields
\[
w_t = \frac{r_0}{r_t}a_0^\perp.
\]
For the coupled system, define
\[
x_t \triangleq r_t\alpha_t.
\]
Then
\[
\dot x_t
=
\dot r_t\alpha_t + r_t\dot\alpha_t
=
-\frac{c}{r_t^2}b_t,
\qquad
\dot b_t = \frac{c}{r_t^2}x_t.
\]
Thus
\[
\frac{d}{dt}
\begin{bmatrix}
x_t \\ b_t
\end{bmatrix}
=
\frac{c}{r_t^2}
\begin{bmatrix}
0 & -1 \\
1 & 0
\end{bmatrix}
\begin{bmatrix}
x_t \\ b_t
\end{bmatrix}.
\]
This is a planar rotation with angular velocity $c/r_t^2$, so if
\[
\theta(t) = \int_0^t \frac{c}{r_s^2}\,ds,
\]
then
\[
x_t = x_0\cos\theta(t) - b_0\sin\theta(t),
\qquad
b_t = x_0\sin\theta(t) + b_0\cos\theta(t),
\]
where $x_0 = r_0\alpha_0$. Dividing by $r_t$ gives
\[
\alpha_t = \frac{r_0\alpha_0\cos\theta(t) - b_0\sin\theta(t)}{r_t},
\]
and hence
\[
a_t = w_t + \alpha_t e
= \frac{r_0}{r_t}a_0^\perp
+ \frac{r_0\alpha_0\cos\theta(t) - b_0\sin\theta(t)}{r_t}e.
\]
This proves \eqref{eq: cone PT explicit a}. Finally, since $\gamma$ has constant speed $s$,
\[
\dot r_t^2 + r_t^2\|v_t\|_2^2 = s^2.
\]
Using the radial geodesic equation $\ddot r_t - r_t\|v_t\|_2^2 = 0$, we obtain
\[
\frac{d^2}{dt^2}(r_t^2) = 2\dot r_t^2 + 2r_t\ddot r_t = 2s^2 \quad \implies \quad 
r_t^2 = t^2 + 2r_0\dot r_0\,t + r_0^2.
\]
Therefore
\[
\theta(t)
=
\int_0^t \frac{c}{s^2 + 2r_0\dot r_0\,s + r_0^2}\,ds
=
\arctan\left(\frac{s^2t + r_0\dot r_0}{c}\right)
-
\arctan\left(\frac{r_0\dot r_0}{c}\right),
\]
where we used the identity $c^2 = r_0^2(s^2 - \dot r_0^{\,2})$ coming from the constant speed condition at $t=0$. \qed{}

\begin{figure}
    \centering
    \includegraphics[width=\linewidth]{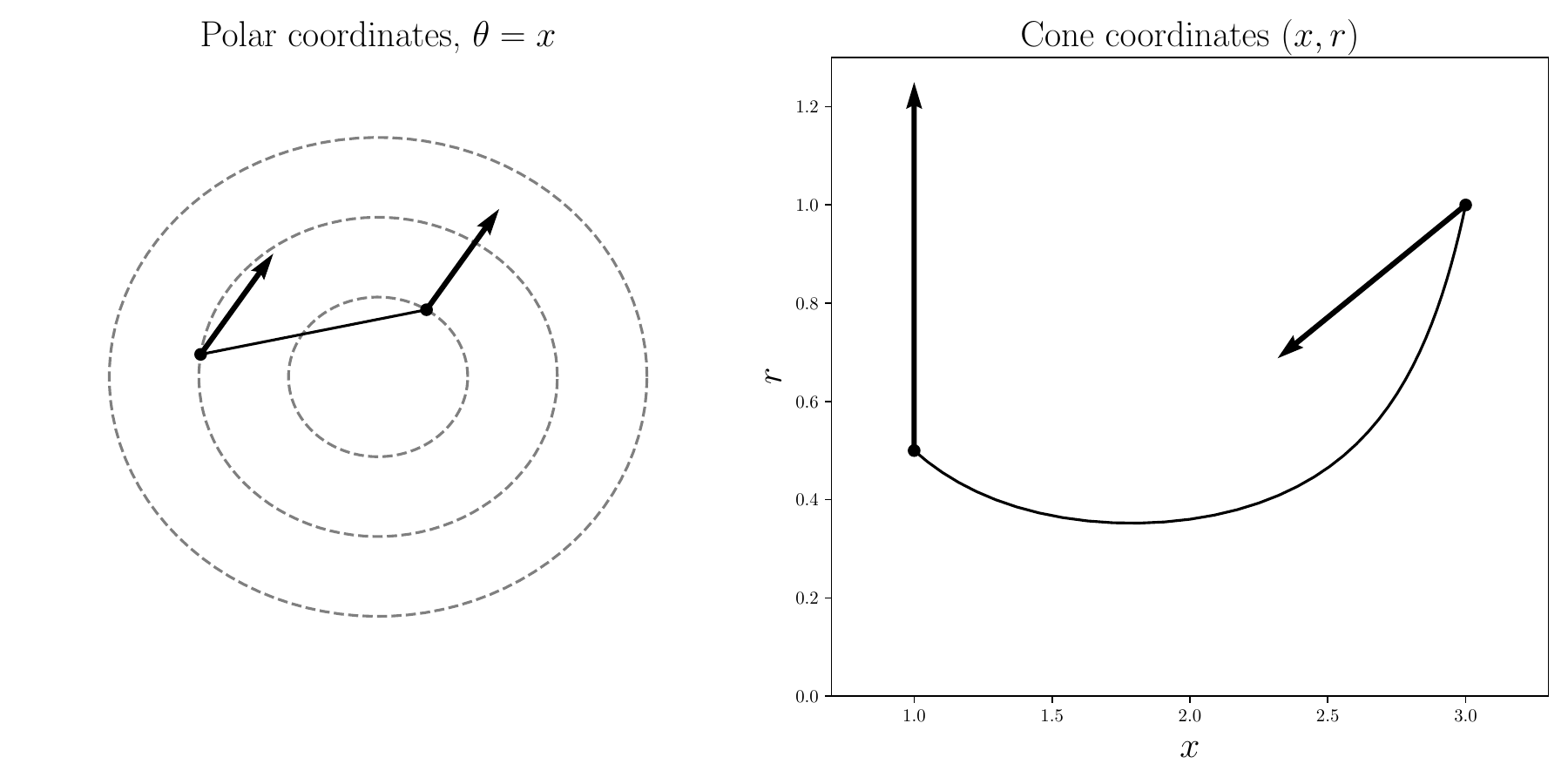}
    \caption{Visualization of geodesic parallel transport on $\mathfrak{C}_\Omega$ in polar coordinates (left) and cone coordinates (right). The visual consists of the parallel transport of $v = 0.75\partial_r$ from $(x_0, r_0) = (1.0, 0.5)$ to $(x_1, r_1) = (3.0, 1.0).$}
    \label{fig:cone_pt}
\end{figure}

\section{Proofs} \label{sec: proof of main results}

\subsection{Proof of \Cref{cone exponential and log map}.} \label{sec: proof of cone exponential and log map}

We desire a $z_1 \in \mathfrak{C}_\Omega$ such that $Z'(0^+; z_0, z_1)$, the right derivative at $0$, is equal to $(v_x, v_r).$ Letting $\theta = \|x_1 - x_0\|_2$, one can verify that 
\begin{align*}
    \rho'(s; z_0, z_1) = \frac{\sin^2(\theta)r_0r_1^2s}{\theta B(s)^{3/2}\sqrt{1 - \frac{A(s)^2}{B(s)}}}
\end{align*}
with 
\begin{equation*}
    A(s) = r_1s\cos(\theta) + r_0(1-s) \qquad B(s) = r_1^2s^2 + 2s(1-s)r_0r_1\cos(\theta) + r_0^2(1-s)^2.
\end{equation*}
We can directly evaluate $\lim_{s \rightarrow 0^+}B(s) = r_0^2$. For the radical,
\begin{align*}
    1 - \frac{A(s)^2}{B(s)} = \frac{\left(\frac{r_1s}{r_0(1-s)}\right)^2\sin^2(\theta)}{1 + 2\frac{r_1s}{r_0(1-s)}\cos(\theta) + \left(\frac{r_1s}{r_0(1-s)}\right)^2} = \left(\frac{r_1s}{r_0(1-s)}\right)^2\sin^2(\theta)\left(1 + u - u^2 + O(u^3)\right)
\end{align*}
where $u = 2\frac{r_1s}{r_0(1-s)}\cos(\theta) + \left(\frac{r_1s}{r_0(1-s)}\right)^2$. Plugging back in and canceling terms yields
\begin{align*}
    \lim_{s \rightarrow 0^+}\rho'(s; z_0, z_1) = \lim_{s\rightarrow 0^+}\left(\frac{\sin(\theta)r_1r_0^2(1-s)}{\theta B(s)^{3/2}\sqrt{1  + u - u^2 + O(u^3)}}\right) = \frac{\sin(\theta)r_1}{\theta r_0}
\end{align*}
since $\lim_{s \rightarrow 0^+}u = 0$, which implies $X'(0^+; z_0, z_1) = \frac{\sin(\theta)r_1(x_1 - x_0)}{\theta r_0}$. One can also easily verify that $R'(0^+; z_0, z_1) = r_1\cos(\theta) - r_0$. Now we have a system of equations 
\begin{equation*}
    v_x = \frac{\sin(\|x_1 - x_0\|_2)r_1(x_1 - x_0)}{\|x_1 - x_0\|_2 r_0} \quad \text{and} \quad  v_r = r_1\cos(\|x_1 - x_0\|_2) - r_0
\end{equation*}
for which we know $(x_0, r_0)$ and we want to solve for $(x_1, r_1)$. One can verify that the solution is $r_1 = \sqrt{\|v_x\|_2^2r_0^2 + (v_r+r_0)^2}$ and $x_1 = x_0 + \theta\frac{v_x}{\|v_x\|_2}$ with $\theta = \operatorname{atan2}(r_0\|v_x\|_2, v_r + r_0) \in [0, \pi)$. This completes the exponential map. The logarithmic map is given by
$\log^{\mathfrak{C}}_{z_0}(z_1) = Z'(0^+; z_0, z_1)$, the expressions for which we have already derived in computing the exponential map.
\qed{}

\subsection{Proof of \Cref{isometry}.} \label{sec: proof of isometry}
We will start by showing that 
\[\mathcal{L}_{\mathfrak{P}\lambda, \lambda}: T_{\mathfrak{P}\lambda}\mathfrak{M}_+^\Gamma \rightarrow S_\lambda \subset T_{\lambda}\mathcal{P}_2(\mathfrak{C}_\Omega) \quad \text{given by} \quad \mathcal{L}_{\mu, \lambda}[(v, \beta)](x,r) = (v(x), 2\beta(x)r)\]
is an isometry. Let $s_1 = (v_1, \beta_1)$ and $s_2 = (v_2, \beta_2)$ be tangent vectors in $T_{\mathfrak{P}\lambda}\mathfrak{M}_+^\Gamma$. Their inner product in the domain Hilbert space is given by 
\[\langle s_1, s_2\rangle_{\mathfrak{P}\lambda} = \int_{\Omega}\left( \langle v_1(x), v_2(x)\rangle + 4\beta_1(x)\beta_2(x)\right)d\mathfrak{P}\lambda(x) \]
where the inner product $\langle\cdot, \cdot\rangle$ is the standard Euclidean inner product on $\mathbb{R}^d$. For the lifted field, the inner product in the image is given by
\begin{align*}\langle \mathcal{L}[s_1], \mathcal{L}[s_2]\rangle_{\lambda} = \int_{\mathfrak{C}_\Omega} \langle \mathcal{L}[s_1], \mathcal{L}[s_2]\rangle_{\mathfrak{C}} d\lambda &= \int_{\mathfrak{C}_\Omega}r^2\left(4\beta_1(x)\beta_2(x) + \langle v_1(x), v_2(x) \rangle\right)d\lambda(x,r) \\
&= \int_\Omega \left(4\beta_1(x)\beta_2(x) + \langle v_1(x), v_2(x) \rangle\right)d\mathfrak{P}\lambda(x)
\end{align*}
by \Cref{eq: measure projection}. Thus, $\langle \mathcal{L}[s_1], \mathcal{L}[s_2]\rangle_{\lambda} = \langle s_1, s_2\rangle_{\mathfrak{P}\lambda}$ proving (a). Now we will prove (b). Let $u, v \in T_{\lambda}\mathcal{P}_2(\mathfrak{C}_\Omega)$ where $d\mathfrak{P}\lambda = r^2d\mu$,  $u = a + b\partial_r$ and $v = c + d\partial_r$. Their inner product in the domain Hilbert space is given by
\begin{align*}
    \langle u, v\rangle_{\lambda} &= \int_{\mathfrak{C}} \langle u, v \rangle_{\mathfrak{C}}\,d\lambda = \int_{\mathfrak{C}} \left(b(x,r)d(x,r) + r^2\langle a(x,r), c(x,r)\rangle \right)\,d\lambda(x,r)
\end{align*}
where the inner product on the right is the standard Euclidean inner product on $\mathbb{R}^d$. By the hypothesis, the radial law of $\lambda$ is deterministic given $x$. Thus, the projections of $u$ and $v$ are given by $\mathcal{P}_{\lambda}[u](x) = (a(x,r(x)), b(x,r(x))/2r(x))$ and $\mathcal{P}_{\lambda}[v](x) = (c(x,r(x)), d(x,r(x))/2r(x))$ and their inner product is
\begin{align*}
    \langle \mathcal{P}_{\lambda}[u], \mathcal{P}_{\lambda}[v]\rangle_{\mathfrak{P}\lambda} &= \int_{\Omega} \left(\frac{4b(x, r(x))d(x, r(x))}{4r^2(x)} +  \langle a(x, r(x)), c(x,r(x))\rangle\right)d\mathfrak{P}\lambda(x) \\
    &= \int_{\Omega}\left(\frac{b(x, r(x))d(x, r(x))}{r^2(x)} +  \langle a(x, r(x)), c(x,r(x))\rangle\right)d\mathfrak{P}\lambda(x) \\
    &= \int_{\mathfrak{C}_\Omega}\left(b(x, r)d(x, r) +  r^2\langle a(x, r), c(x,r)\rangle\right)d\lambda(x,r)
\end{align*}
which is equivalent to $\langle u, v \rangle_\lambda$ since $r = r(x)$ $\lambda$-almost everywhere. Finally, all that remains is to check the condition $\mathcal{P}_\lambda(\mathcal{L}_{\mathfrak{P}\lambda, \lambda}(v,\beta)) = (v, \beta)$ $\mathfrak{P}\lambda$-a.e. and $\mathcal{L}_{\mathfrak{P}\lambda}(\mathcal{P}_\lambda[u]) = u$ $\lambda$-almost everywhere. For the first condition, we have 
\begin{align*}
    \mathcal{P}_\lambda(\mathcal{L}_{\mathfrak{P}\lambda, \lambda}(v,\beta)) &= \mathcal{P}_\lambda \left[\left(v(x), 2\beta(x)r\right)\right] = \left(v(x), \frac{2\beta(x)r(x)}{2r(x)}\right) = (v(x), \beta(x))
\end{align*}
for $\mathfrak{P}\lambda$-a.e. $x.$ For the second condition, by the condional radial degeneracy of $\lambda$ we have 
\begin{align*}
    \mathcal{L}_{\mathfrak{P}\lambda, \lambda}\left(\mathcal{P_{\lambda}}[u]\right) = \mathcal{L}_{\mathfrak{P}\lambda, \lambda}\left(\left(a(x, r(x)), \frac{b(x, r(x))}{2r(x)}\right)\right) = (a, b).
\end{align*}
\qed{}

\subsection{Proof of \Cref{lifting by characteristics one}} \label{sec: proof of lifting by characteristics one}

Define $\nu_t$ such that $d\nu_t = r_t^2 d\eta_t$.  We will show that $(\nu_t)_{t\in[0,1]}$ solves the continuity-reaction equation with initial condition $\mu_0$. By uniqueness, this will imply $\nu_t=\mu_t$ for all $t$. Let $\psi\in C_c^1((0,1)\times \Omega)$. Then
\begin{align*}
    \int_\Omega \psi(t,y)\,d\nu_t(y)
    &= \int_\Omega \psi(t,y)r_t(y)^2\,d\eta_t(y) \\
    &= \int_\Omega \psi(t,X_t(x))r_t(X_t(x))^2\,d\mu_0(x) \\
    &= \int_\Omega \psi(t,X_t(x))
    \exp\left(4\int_0^t \beta_s(X_s(x))\,ds\right)\,d\mu_0(x).
\end{align*}
Differentiating in $t$ and using the chain rule yields
\begin{align*}
    \frac{d}{dt}\int_\Omega \psi(t,y)\,d\nu_t(y)
    &=
    \frac{d}{dt}\int_\Omega \psi(t,X_t(x))
    \exp\left(4\int_0^t \beta_s(X_s(x))\,ds\right)\,d\mu_0(x)
\end{align*}
Now observe that the function 
\[g(x) \equiv \sup_{(t,y) \in \operatorname{supp}(\psi)}\left(|\partial_t\psi(t,y)|e^{4|\beta_{\max}|} + \|\nabla \psi(t, y)\|_2\|v_t(y)\|_2 \cdot e^{4|\beta_{\max}|} + 4|\beta_{\max}e^{4|\beta_{\max}|}\psi(t,y)|\right) \triangleq C_\psi < \infty\]
dominates the time-derivative of the integrand of the right hand side above, and is integrable since (1) $\psi$ is $C^1$ and the supremum is taken over a compact set (and thus the supremum is finite), and (2) because $v_t$ is assumed to be uniformly bounded on $[0,1]\times \Omega$. Thus, by the dominated convergence theorem we can differentiate under the integral to say,
\begin{align*}
    \frac{d}{dt}\int_\Omega \psi(t,y)\,d\nu_t(y)
    &=
    \int_\Omega \frac{d}{dt}\left[
        \psi(t,X_t(x))
        \exp\left(4\int_0^t \beta_s(X_s(x))\,ds\right)
    \right]d\mu_0(x) \\
    &=
    \int_\Omega \Bigl(
        \partial_t\psi(t,X_t(x))
        + \nabla \psi(t,X_t(x))\cdot v_t(X_t(x))
    \Bigr)
    \exp\left(4\int_0^t \beta_s(X_s(x))\,ds\right)\,d\mu_0(x) \\
    &\hspace{31mm}
    +4\int_\Omega \beta_t(X_t(x))\psi(t,X_t(x))
    \exp\left(4\int_0^t \beta_s(X_s(x))\,ds\right)\,d\mu_0(x).
\end{align*}
Rewriting these terms in terms of $\nu_t$, we obtain
\begin{align*}
    \frac{d}{dt}\int_\Omega \psi(t,y)\,d\nu_t(y)
    &=
    \int_\Omega \partial_t\psi(t,y)\,d\nu_t(y)
    +\int_\Omega \nabla\psi(t,y)\cdot v_t(y)\,d\nu_t(y) 
    +4\int_\Omega \beta_t(y)\psi(t,y)\,d\nu_t(y).
\end{align*}
Integrating over $t\in[0,1]$, the fact that $\psi$ has compact support in $(0,1)\times\Omega$ implies that the left hand side of the panel above is zero, giving
\begin{align*}
    0 &=
    \int_0^1\int_\Omega \partial_t\psi(t,y)\,d\nu_t(y)\,dt
    +\int_0^1\int_\Omega \nabla\psi(t,y)\cdot v_t(y)\,d\nu_t(y)\,dt +4\int_0^1\int_\Omega \beta_t(y)\psi(t,y)\,d\nu_t(y)\,dt.
\end{align*}
Thus $(\nu_t)_{t\in[0,1]}$ is indeed a weak solution of the continuity-reaction equation with $v_t, \beta_t$. By the following result, we also know that $\nu_t$ is weakly continuous in $t$. 

\begin{lemma} \label{nu_weak_continuity}
    The curve of measures $(\nu_t)_{t \in [0,1]}$ is weakly continuous.
\end{lemma}
\begin{proof}
    To show that $\nu_t$ is weakly continuous, we need to show that
    $t \mapsto \int f\,d\nu_t$ is a continuous function of $t$ for all bounded
    and continuous test functions. In particular, for any bounded and continuous
    $f$, we have
    \[
        \int_{\Omega} f(y)\,d\nu_t(y)
        =
        \int_{\Omega} f(X_t(x))R_t(x)^2\,d\mu_0(x),
    \]
    where
    \[
        R_t(x) = r_t(X_t(x))
        =
        \exp\left(2\int_0^t \beta_s(X_s(x))\,ds\right).
    \]
    Since $X_t(x)$ is the characteristic ODE, it is continuous in $t$ for
    $\mu_0$-a.e. $x$. Similarly, since
    $s\mapsto \beta_s(X_s(x))$ belongs to $L^1(0,1)$ for $\mu_0$-a.e. $x$, the map
    \[
        t \mapsto \int_0^t \beta_s(X_s(x))\,ds
    \]
    is absolutely continuous, and hence continuous. Thus $R_t(x)$ is continuous
    in $t$ for $\mu_0$-a.e. $x$. Therefore, for every sequence $t\to t^*$ we have
    \[
        f(X_t(x))R_t(x)^2
        \to
        f(X_{t^*}(x))R_{t^*}(x)^2
    \]
    for $\mu_0$-a.e. $x$. Moreover, since $\beta_s \leq \beta_{\max}$,
    \[
        R_t(x)^2
        =
        \exp\left(4\int_0^t \beta_s(X_s(x))\,ds\right)
        \leq e^{4\beta_{\max}}.
    \]
    Hence
    \[
        |f(X_t(x))R_t(x)^2|
        \leq
        \|f\|_\infty e^{4\beta_{\max}},
    \]
    and this upper bound is $\mu_0$-integrable since $\mu_0(\Omega)<\infty$.
    Therefore, by the dominated convergence theorem,
    \[
        \lim_{t\to t^*}
        \int_{\Omega} f(y)\,d\nu_t(y)
        =
        \int_{\Omega} f(X_{t^*}(x))R_{t^*}(x)^2\,d\mu_0(x)
        =
        \int_{\Omega} f(y)\,d\nu_{t^*}(y).
    \]
    Thus, the curve of measures $\nu_t$ is weakly continuous.
\end{proof}

By the assumed uniqueness of weakly continuous distributional solutions with initial condition $\mu_0$, it follows that $\nu_t=\mu_t$ for all $t \in [0,1]$. To verify the lifting property, observe that for any continuous $\phi$
\begin{align*}
    \int_\Omega \phi \,d\mathfrak{P}\lambda_t = \int_{\mathfrak{C}_\Omega} r^2 \phi(x) \, d\lambda_t(x,r) = \int_\Omega r_t^2(x)\phi(x) \, d\eta_t(x) = \int_\Omega \phi \,d\mu_t
\end{align*}
which completes the proof.
\qed{}

\subsection{Proof of \Cref{radial update}} \label{sec: proof of radial update}
By \Cref{lifting by characteristics one}, we have that
\begin{align*}
r_{i+1}^2 &= \frac{d\mu_{i+1}}{d\eta_{i+1}}  = \frac{d(u_{i+1}((T_i)_\#\pi_0))}{d((T_{i})_\#\eta_i)}.
\end{align*}
Since $T_i$ is $\eta_i$-a.e. injective, we have
\begin{align*}
    r_{i+1}^2 = u_{i+1}\frac{d\pi_0}{d\eta_i} \circ T_i^{-1} = u_{i+1}\left(\frac{d\pi_0}{d\mu_i}\frac{d\mu_i}{d\eta_i}\right)\circ T_{i}^{-1} \qquad \eta_{i+1}\text{-a.e}.
\end{align*}
Since $\mu_i = u_i \pi_0$, we have
\begin{align*}
    r_{i+1}^2(x) = u_{i+1}(x)\left(\frac{r_i^2(T_i^{-1}(x))}{u_i(T_i^{-1}(x))}\right) \quad \text{and thus} \quad \frac{r_{i+1}^2(x)}{r_i^2(T_i^{-1}(x))} = \frac{u_{i+1}(x)}{u_i(T_i^{-1}(x))}.
\end{align*}
Applying the change of variables $y = T_i^{-1}(x)$ and taking the square root yields
\[\frac{r_{i+1}(T_i(y))}{r_i(y)} = \sqrt{\frac{u_{i+1}(T_i(y))}{u_i(y)}}\]
which completes the proof.
\qed{}

\subsection{Proof of \Cref{lifting by characteristics two}} \label{sec: proof of lifting by characteristics two}

Recall, $\lambda_t = (x \mapsto (x, r_t(x)))_\# \eta_t$. By construction, the conditional radial law of $\lambda_t$ is deterministic. By \Cref{isometry}, we know that $\mathcal{L}_{\mu_t, \lambda_t}$ and $\mathcal{P}_{\lambda_t}$ are isometric inverses from $T_{\mu_t}\mathfrak{M}_+^\Gamma(\Omega)$ to $S_{\lambda_t} \subset T_{\lambda_t}\mathcal{P}_2(\mathfrak{C}_\Omega)$. Now we will show that the lifted tangent field $V_t = \mathcal{L}_{\mu_t, \lambda_t}(v_t, \beta_t)$ and the lifted curve of measures $\lambda_t$ satisfy the continuity equation on $\mathfrak{C}_\Omega$ weakly. 
Let $\Psi \in C_c^1\bigl((0,1)\times \mathfrak{C}_\Omega\bigr).$ Define
\[
F_t(x)\triangleq \bigl(X_t(x),\, r_t(X_t(x))\bigr)\in \mathfrak{C}_\Omega.
\]
Then $\lambda_t=(F_t)_\#\eta_0$, and therefore
\[
\int_{\mathfrak{C}_\Omega}\Psi(t,z)\,d\lambda_t(z)
=
\int_\Omega \Psi\bigl(t,F_t(x)\bigr)\,d\eta_0(x).
\]
Differentiating in $t$ and using the chain rule gives
\begin{align*}
\frac{d}{dt}\int_{\mathfrak{C}_\Omega}\Psi(t,z)\,d\lambda_t(z)
&=
\int_\Omega \frac{d}{dt}\Psi\bigl(t,F_t(x)\bigr)\,d\eta_0(x) \\
&=
\int_\Omega \left[
\partial_t\Psi\bigl(t,F_t(x)\bigr)
+
\left\langle
\nabla_{\mathfrak{C}}\Psi\bigl(t,F_t(x)\bigr),
\frac{d}{dt}F_t(x)
\right\rangle_{\mathfrak{C}}
\right] d\eta_0(x).
\end{align*}
Since
\[
\frac{d}{dt}X_t(x)=v_t(X_t(x))
\qquad\text{and}\qquad
\frac{d}{dt}\bigl[r_t(X_t(x))\bigr]
=
2\,\beta_t(X_t(x))\,r_t(X_t(x)),
\]
we have
\[
\frac{d}{dt}F_t(x)
=
\Bigl(v_t(X_t(x)),\,2\,\beta_t(X_t(x))\,r_t(X_t(x))\Bigr)
=
V_t(F_t(x)).
\]
Hence
\begin{align*}
\frac{d}{dt}\int_{\mathfrak{C}_\Omega}\Psi(t,z)\,d\lambda_t(z)
&=
\int_\Omega
\left[
\partial_t\Psi\bigl(t,F_t(x)\bigr)
+
\left\langle
\nabla_{\mathfrak{C}}\Psi\bigl(t,F_t(x)\bigr),
V_t(F_t(x))
\right\rangle_{\mathfrak{C}}
\right]
d\eta_0(x) \\
&=
\int_{\mathfrak{C}_\Omega}
\left[
\partial_t\Psi(t,z)
+
\langle \nabla_{\mathfrak{C}}\Psi(t,z),V_t(z)\rangle_{\mathfrak{C}}
\right]
d\lambda_t(z).
\end{align*}
Integrating over $t\in(0,1)$, we obtain
\begin{align*}
\int_0^1 \frac{d}{dt}\left(\int_{\mathfrak{C}_\Omega}\Psi(t,z)\,d\lambda_t(z)\right)dt
&=
\int_0^1\int_{\mathfrak{C}_\Omega}
\left[
\partial_t\Psi(t,z)
+
\langle \nabla_{\mathfrak{C}}\Psi(t,z),V_t(z)\rangle_{\mathfrak{C}}
\right]
d\lambda_t(z)\,dt.
\end{align*}
Since $\Psi$ has compact support in time, the left-hand side is zero. Therefore
\[
\int_0^1\int_{\mathfrak{C}_\Omega}
\left[
\partial_t\Psi(t,z)
+
\langle \nabla_{\mathfrak{C}}\Psi(t,z),V_t(z)\rangle_{\mathfrak{C}}
\right]
d\lambda_t(z)\,dt
=0.
\]
This is exactly the weak form of the continuity equation on $\mathfrak{C}_\Omega$. To show that the lift is optimal, consider the following. We know $(\lambda_t, V_t)$ satisfy the continuity equation, so in order for them to be admissible we need the following result. 

\begin{lemma}
    The curve of measures $(\lambda_t)_{t \in [0,1]}$ is weakly continuous.
\end{lemma}
\begin{proof}
    The proof is nearly identical to that of \Cref{nu_weak_continuity}. To show that $\lambda_t$ is weakly continuous, we need to show that $t \mapsto \int f\, d\lambda_t$ is a continuous function of $t$ for all bounded and continuous test functions. In particular, for any bounded and continuous $f$, we have
    \[\int_{\mathfrak{C}_\Omega} f(x,r)\, d\lambda_t(x,r) = \int_{\Omega}f(X_t(x), R_t(x))\, d\mu_0(x)\]
    where $R_t(x) = r_t(X_t(x))$. Since $X_t(x)$ is the characteristic ODE, it is continuous in $t$. Similarly, by the definition of $r_t(X_t(x))$ and the fact that $\beta_t$ is bounded (and thus $\beta_s$ is integrable along the flow), we know that it too is continuous in $t$. Thus, for every sequence $t \rightarrow t^*$ we have 
    \[f(X_{t}(x), R_{t}(x)) \rightarrow f(X_{t^*}(x), R_{t^*}(x))\]
    for $\mu_0$-a.e. $x$. Thus, since $\|f\|_{\infty} < \infty$ and $\mu_0(\Omega) < \infty$, we can apply the dominated convergence theorem to say 
    \[\lim_{t \rightarrow t^*} \int_{\Omega}f(X_t(x), R_t(x))\, d\mu_0(x) \rightarrow \int_{\Omega} f(X_{t^*}(x), R_{t^*}(x))\, d\mu_{0}(x).\]
    Thus, the curve of measures $\lambda_t$ is indeed weakly continuous.
\end{proof}

Since $\lambda_t$ is indeed weakly continuous, we can apply the Benamou-Brenier theorem to say 
\[W_\mathfrak{C}(\lambda_0, \lambda_1)^2 \leq \int_0^1\|V_t\|_{L^2(\lambda_t)}^2 \, dt.\]
By the isometry of $\mathcal{P}_{\lambda_t}$, we have
\begin{align*}
    W_{\mathfrak{C}}(\lambda_0, \lambda_1)^2 \leq \int_0^1 \|(v_t, \beta_t)\|_{\mu_t}^2 \, dt &= \HK(\mu_0, \mu_1)^2
\end{align*}
since $\mu_t$ is a geodesic and $(v_t, \beta_t, \mu_t)$ solve the continuity-reaction equation. For the reverse inequality, the characterization given in \Cref{eq: cone wasserstein and hk equivalence} indicates that $W_\mathfrak{C}(\lambda_0, \lambda_1)^2 \geq \HK(\mu_0, \mu_1)^2$. Thus, $W_\mathfrak{C}(\lambda_0, \lambda_1) = \HK(\mu_0, \mu_1)$ and the lifts $\lambda_0, \lambda_1$ are optimal. Moreover, 
\[\int_0^1\|V_t\|_{L^2(\lambda_t)}^2 \, dt = W_{\mathfrak{C}}(\lambda_0, \lambda_1)^2\]
implying that $(\lambda_t, V_t)$ attains the Benamou-Brenier minimum --
it follows that $\lambda_t$ is a $W_\mathfrak{C}$ geodesic.
\qed{}

\subsection{Proof of \Cref{uniformly bounded lifts}} \label{sec: proof of uniformly bounded lifts}

Recall that $r_t$ along the characteristic flow $X_t$ is given by
\[r_t(X_t(x)) = \exp\left(2\int_0^t \beta_s(X_s(x))\, ds\right).\]
Thus, since $\beta_t$ satisfies
$\beta_t \geq \beta_{\min}$ and $\beta_t \leq \beta_{\max}$
for some $\beta_{\min} > -\infty, \beta_{\max} < \infty$ independent of $t$, then 
\[  \beta_{\max}t\geq \int_0^t\beta_s(X_s(x)) \, ds \geq \beta_{\min}t\]
and $r_{\max}\triangleq \exp(2\max\{\beta_{\max}, 0\})\geq r_t(X_t(x)) \geq \exp\left(2 \min\{\beta_{\min}, 0\}\right) \triangleq r_{\min}$.
\qed{}

\subsection{Proof of \Cref{uniformly regular lifts}} \label{sec: proof of uniformly regular lifts}

Recall the definition $V_t(x,r) = (v_t(x), 2\beta_t(x)r)$. We will start by proving the integrability result. By the isometry of the lifting procedure (\Cref{lifting by characteristics two}), we have
\begin{align*}
    \int_0^1 \|V_t\|_{L^2(\lambda_t)}^2 \, dt &= \int_0^1  \|(v_t, \beta_t)\|_{L^2(\mu_t)}^2 \, dt \\
    &= \int_0^1 \|v_t\|_{L^2(\mu_t)}^2\, dt + 4\int_0^1\|\beta_t\|^2_{L^2(\mu_t)}\, dt \\
    &< \infty 
\end{align*}
by \Cref{admissible class}. For the second result, note that we can bound the spatial Lipschitz constant of a vector field on a manifold using the covariant derivative \citet[Proposition 10.46]{boumal2023introduction}. In particular, $V_t$ is $L$-Lipschitz continuous if and only if $\|\nabla^{\mathfrak{C}}_UV_t\|_{\mathfrak{C}} \leq L\|U\|_{\mathfrak{C}}$ for all $U\in T\mathfrak{C}_\Omega$. Let $U = a + b\partial_r \in T\mathfrak{C}_\Omega$ be a tangent field and write $V_t = c_t + d_t \partial_r$ where $c_t(x,r) = v_t(x)$ and $d_t(x,r) = 2\beta_t(x)r$. The covariant derivative of $V_t$ in the direction $U$ is given by 
\begin{align*}
    \nabla^{\mathfrak{C}}_{U}V_t &= \nabla_a^{\mathfrak{C}}V_t + b \nabla ^{\mathfrak{C}}_{\partial_r}V_t \\
    &= \nabla_a^{\mathfrak{C}}c_t + \nabla_a^{\mathfrak{C}}(d_t\partial_r) + b\nabla ^{\mathfrak{C}}_{\partial_r}c_t + b\nabla_{\partial_r}^{\mathfrak{C}}(d_t\partial_r).
\end{align*}
Now we will apply \Cref{corollary: connection on cone}. The first term is
\begin{align*}
    \nabla_a^{\mathfrak{C}}c_t &= \nabla v_t^\top a - r\langle v_t, a\rangle \partial_r.
\end{align*}
The second term can be simplified by applying the Leibniz rule, which gives
\begin{align*}
    \nabla _a^\mathfrak{C}(d_t\partial_r) &= a(d_t)\partial_r + d_t \nabla_a^\mathfrak{C}\partial_r \\
    &= 2ra(\beta_t) \partial_r + 2\beta_ta.
\end{align*}
For the third term, we have $b\nabla_{\partial_r}^\mathfrak{C}c_t = bv_t/r$, and for the fourth term we have
\begin{align*}
    b\nabla_{\partial_r}^{\mathfrak{C}}d_t\partial_r &= b\left(\partial_r(d_t) + d_t\nabla^\mathfrak{C}_{\partial_r}\partial_r\right)\partial_r \\
    &= 2b \beta_t\partial_r.
\end{align*}
Thus,
\[\nabla^{\mathfrak{C}}_{U}V_t = \nabla v_t^\top a + 2 \beta_t a+ \frac{bv_t}{r} + \left(2ra(\beta_t) + 2b\beta_t - r\langle v_t, a\rangle\right)\partial_r.\]
Taking the norm and applying the triangle inequality yields
\begin{align}
    \left\| \nabla^\mathfrak{C}_U V_t\right\|_{\mathfrak{C}} &\le  r\left\|\nabla v_t \right\|_{\operatorname{op}}\|a\|_2 + 2r|\beta_t|\|a\|_2 + |b|\|v_t\|_2 + 2r|a(\beta_t)| + 2|b||\beta_t| + r\|v_t\|_2\|a\|_2. \label{cov deriv upper bound}
\end{align}
We'll continue to upper bound these terms until we arrive at something that is a constant multiple of $\|U\|_\mathfrak{C}$. Since $a$ has no radial component, we have
\begin{align*}
    |a(\beta_t)| &= |\langle \nabla^\mathfrak{C}\beta_t, a \rangle_{\mathfrak{C}}| = r^2|\langle r^{-2}\nabla \beta_t, a\rangle | \leq \|\nabla \beta_t\|_2\|a\|_2.
\end{align*}
Also observe that we have the following bounds $r\|a\|_2 \leq \|U\|_{\mathfrak{C}}$ and $|b| \leq \|U\|_{\mathfrak{C}}$; when applied to \cref{cov deriv upper bound} this yields
\begin{align*}
    \left\| \nabla^\mathfrak{C}_U V_t\right\|_{\mathfrak{C}} &\le  \left(\left\|\nabla v_t \right\|_{\operatorname{op}} + 2\|\nabla \beta_t\|_2 + 2\|v_t\|_2  + 4|\beta_t|\right)\|U\|_{\mathfrak{C}}
\end{align*}
and taking the supremum over $(x,r)$ yields
\begin{align*}
    \sup_{(x,r)\in \Gamma \times [r_{\min}, r_{\max}]}\left\| \nabla^\mathfrak{C}_U V_t\right\|_{\mathfrak{C}} &\le  \left(\left\|\nabla v_t \right\|_{L^\infty} + 2\|\nabla \beta_t\|_{L^\infty} + 2\|v_t\|_{L^\infty}  + 4\|\beta_t\|_{L^\infty}\right)\|U\|_{\mathfrak{C}}.
\end{align*}
The properties of the measures in the admissible class (\Cref{admissible class}) guarantee that the quantity in the parentheses is bounded by some $L < \infty$ for all $t$, completing the proof.
\qed{}

\subsection{Proof of \Cref{HK parallel transport pde}} \label{sec: proof of HK PDE}
By definition of the Hellinger--Kantorovich covariant derivative, if we write
\begin{align*}
    A_t 
    &\triangleq \partial_t v_t + \nabla v_t \cdot \nabla \varphi_t + 2\varphi_tv_t+ 2\beta_t \nabla \varphi_t 
    , \quad \text{and} \quad B_t
    \triangleq \partial_t \beta_t + \frac12 \langle \nabla \beta_t, \nabla \varphi_t \rangle + 2\varphi_t \beta_t,
\end{align*}
then the parallel transport condition is equivalent to
\begin{align*}
    \Pi_{\mu_t}\binom{A_t}{B_t} = 0.
\end{align*}
Since $\Pi_{\mu_t}$ is the orthogonal projection onto \(T_{\mu_t}\mathfrak{M}_+^\Gamma\), we have
\begin{align*}
    \Pi_{\mu_t}\binom{A_t}{B_t}=0
    \quad\Longleftrightarrow\quad
    \binom{A_t}{B_t}\perp T_{\mu_t}\mathfrak{M}_+^\Gamma.
\end{align*}
By density of $\left\{(\nabla\psi,\psi):\psi\in C_c^\infty(U)\right\}$
in $T_{\mu_t}\mathfrak{M}_+^\Gamma$, this is equivalent to
\begin{align*}
    \int_U
    \left(
        \langle A_t,\nabla \psi\rangle + 4 B_t \psi
    \right)\,d\mu_t
    = 0
    \qquad
    \forall \psi \in C_c^\infty(U).
\end{align*}
Substituting the definitions of \(A_t\) and \(B_t\), we conclude that the parallel condition is equivalent to
\begin{align*}
    \int_U
    \Bigg\langle
        \partial_t v_t
        + \nabla v_t \cdot \nabla \varphi_t
        + 2\varphi_tv_t+ 2\beta_t \nabla \varphi_t,
        \nabla \psi
    \Bigg\rangle
    \,d\mu_t
    +
    4\int_U
    \left(
        \partial_t \beta_t
        + \frac12 \langle \nabla \beta_t,\nabla \varphi_t\rangle
        + 2\varphi_t \beta_t
    \right)\psi
    \,d\mu_t
    = 0
\end{align*}
for every \(\psi \in C_c^\infty(U)\). Equivalently, interpreting the first term in the distributional sense,
\begin{align*}
    \int_U \langle A_t,\nabla\psi\rangle\,d\mu_t
    =
    -\left\langle \nabla\cdot(\mu_t A_t),\psi \right\rangle,
\end{align*}
so the above weak formulation may be written as $-\nabla\cdot(\mu_t A_t) + 4 B_t \mu_t = 0$. That is,
\begin{align*}
    -\nabla\cdot\left(
        \mu_t
        \left(
            \partial_t v_t
            + \nabla v_t \cdot \nabla \varphi_t
            + 2\varphi_tv_t+ 2\beta_t \nabla \varphi_t
        \right)
    \right)
    +
    4
    \left(
        \partial_t \beta_t
        + \frac12 \langle \nabla \beta_t,\nabla \varphi_t\rangle
        + 2\varphi_t \beta_t
    \right)\mu_t
    = 0
\end{align*}
in the sense of distributions.
\qed{}

\subsection{Proof of \Cref{cone ambient manifold}} \label{sec: proof of cone ambient manifold}

We will proceed with a proof by construction. By \Cref{uniformly bounded lifts} there exist constants $r_{\min}, r_{\max} > 0$ such that $\operatorname{supp} \lambda_t \subset \Gamma \times[r_{\min}, r_{\max}] \triangleq K$ for all $t \in [0,1]$. Let $U$ be an open connected neighborhood of $\Gamma$ with $\Gamma \Subset U \Subset \Omega^\circ$ and define 
\[\mathcal{U} \triangleq U \times (r_{\min}/2,2 r_{\max}) \subset \mathfrak{C}_\Omega.\]
Since $\mathcal{U}$ is contained in the smooth part of the cone and stays uniformly bounded away from the apex, it is a smooth, open (and thus boundaryless) manifold endowed with the cone metric $g_\mathfrak{C}$. Now, choose an open set $W$ such that $K \Subset W \Subset \mathcal U$. Since $\mathcal U$ is a smooth manifold without boundary, there exists a smooth complete
Riemannian metric $h$ on $\mathcal U$ \citep{nomizu1961existence}. Let $\chi \in C_c^\infty(\mathcal U)$ satisfy
\[
0 \leq \chi \leq 1, \qquad \chi \equiv 1 \text{ on } W.
\]
Define the augmented metric tensor
\[
\widetilde g_{\mathfrak C} \triangleq \chi g_{\mathfrak C} + (1-\chi)h.
\]
Since $g_{\mathfrak C}$ and $h$ are smooth positive-definite bilinear forms, so is
$\widetilde g_{\mathfrak C}$; hence $\widetilde g_{\mathfrak C}$ is a smooth Riemannian metric
on $\mathcal U$. Moreover, $\widetilde g_\mathfrak{C}$ coincides with $g_\mathfrak{C}$ on  a neighborhood of
$\bigcup_{t\in[0,1]}\operatorname{supp}\lambda_t \subset K$. We next show that $\widetilde g_{\mathfrak C}$ is complete. On the compact set
$\operatorname{supp}\chi$, the smooth metrics $g_{\mathfrak C}$ and $h$ are uniformly equivalent (due to positive definiteness),
so there exists $c>0$ such that
$g_{\mathfrak C} \geq c h$ on $\operatorname{supp}\chi.$ Therefore, on $\operatorname{supp}\chi$,
\[
\widetilde g_{\mathfrak C}
= \chi g_{\mathfrak C} + (1-\chi)h
\geq \chi c\, h + (1-\chi)h
\geq \min\{1,c\}\, h.
\]
Outside $\operatorname{supp}\chi$ we have $\widetilde g_{\mathfrak C}=h$. Hence, we have the global bound $\widetilde g_{\mathfrak C} \geq c_0 h$ for some constant $c_0 > 0$. This bound in combination with the completeness of $h$ guarantees that $\tilde g_\mathfrak{C}$ is complete. This proves (1) and (2). It remains to verify (3). By construction,
\[
\operatorname{supp}\lambda_t \subset K \Subset W
\qquad \forall t\in[0,1],
\]
and $\widetilde g_{\mathfrak C}=g_{\mathfrak C}$ on $W$. Hence all geometric quantities computed
along the lifted curve $(\lambda_t,V_t)$ agree whether they are evaluated using $g_{\mathfrak C}$
or $\widetilde g_{\mathfrak C}$. In particular, by \Cref{uniformly regular lifts},
$(\lambda_t)$ satisfies the continuity equation on $\mathcal U$ with velocity field $V_t$, and
\[
\int_0^1 \|V_t\|_{L^2(\lambda_t)}^2\,dt < \infty,
\]
where $\widetilde V_t \triangleq \zeta V_t$ and $\zeta \in C^\infty_c(\mathcal U)$, $\zeta \equiv 1$ on a neighborhood of $K$, and $\zeta$ has a bounded $\widetilde g_{\mathfrak{C}}$-Lipschitz constant. By definition, $\widetilde V_t = V_t$ on $\operatorname{supp} \lambda_t$. It remains to verify the Lipschitz condition of $\widetilde V_t$ on $\mathcal{U}$. To do so, observe that the uniform $W^{1,\infty}(U)$ control on $(v_t, \beta_t)$ implies uniform $W^{1, \infty}(\mathcal{U})$ control of $V_t(x,r) = v_t(x) + 2r\beta_t(x)\partial_r$. To see this, write $d_x$ as the differential, and observe that since \(r\in[r_{\min}/2,2r_{\max}]\) on \(\mathcal U\), we have that
\[
d_x(2r\beta_t(x))=2rd_x\beta_t(x)
\qquad \text{and} \qquad 
\partial_r(2r\beta_t(x))=2\beta_t(x),
\]
and thus the uniform \(W^{1,\infty}(U)\) bounds on \(v_t\) and \(\beta_t\) imply a
uniform $W^{1,\infty}(\mathcal U)$ bound on $V_t$. Denoting $\nabla^{\mathcal{U}}$ as the Levi-Civita connection on $\mathcal{U}$ and applying the Leibniz rule yields
\[\nabla^{\mathcal{U}}_{X}\widetilde V_t = X(\zeta)V_t + \zeta \nabla_X^{\mathcal{U}}V_t \quad \text{for every tangent field $X$ on $\mathcal{U}$}.\]
Thus
\[\|\nabla^{\mathcal{U}}_{X}\widetilde V_t\|_{\widetilde{g_{\mathfrak{C}}}} \leq \left(\|d\zeta\|_{\operatorname{op}}\|V_t\|_{\widetilde{g_{\mathfrak{C}}}} + \|\nabla^{\mathcal{U}}_XV_t\|_{\widetilde{g_{\mathfrak{C}}}}\right)\|X\|_{\widetilde{g_{\mathfrak{C}}}}\]
which, by \citet[Proposition 10.46]{boumal2023introduction} implies
\[\sup_{t \in [0,1]} \operatorname{Lip}^{\widetilde g_{\mathfrak C}}(\widetilde V_t) < \infty,\]
and $(\lambda_t)$ is indeed a regular curve in
$\mathcal P_2(\mathcal U,\widetilde g_{\mathfrak C})$. Finally, since $(\mathcal U,\widetilde g_{\mathfrak C})$ is a smooth and complete manifold without
boundary, Gigli's parallel transport theory applies to $(\lambda_t)$. Moreover, because
$\widetilde g_{\mathfrak C}$ agrees with $g_{\mathfrak C}$ on a neighborhood of the lifted curve,
any construction involving only the curve, its velocity field, and pointwise Riemannian geometry along the support is independent of the chosen completion.
\qed{}

\subsection{Proof of \Cref{prop: pullback connection}} \label{sec: pullback connection equals HK} We will prove the result by evaluating the pulled-back covariant derivative and showing that it coincides with the HK covariant derivative. Observe that $ V_t= (\nabla \varphi_t, 2r\varphi_t)$ and $\mathcal{L}_{t} \mathbf{u}_t = (\nabla \psi_t, 2r\psi_t).$ Then the lifted total derivative is given by 
\begin{align*}
    \mathbf{D}^{W_2, \mathcal{U}}_{t}(\mathcal{L}_t \mathbf{u}_t) &= \partial_t (\mathcal{L}_t\mathbf{u}_t) + \nabla^{\mathfrak{C}}_{V_t}(\mathcal{L}_t \mathbf{u}_t) \\
    &= \partial_t\left(\nabla \psi_t, 2r\psi_t\right) + \nabla^{\mathfrak{C}}_{(\nabla \varphi_t, 2r\varphi_t)}\left(\nabla \psi_t, 2r\psi_t\right).
\end{align*}
By linearity of the connection in its first argument we obtain 
\begin{align*}
    \nabla^{\mathfrak{C}}_{(\nabla \varphi_t, 2r\varphi_t)}\left(\nabla \psi_t, 2r\psi_t\right) &= \nabla^{\mathfrak{C}}_{\nabla \varphi_t + 2r\varphi_t \partial_r}(\nabla \psi_t + 2r\psi_t \partial_r) \\
    &= \nabla^{\mathfrak{C}}_{(\nabla \varphi_t)}\nabla \psi_t + 2r\varphi_t \nabla ^\mathfrak{C}_{(\partial_r)}\nabla \psi_t + \nabla^{\mathfrak{C}}_{(\nabla \varphi_t)}(2r\psi_t\partial_r) + 2r\varphi_t\nabla ^{\mathfrak{C}}_{(\partial_r)}(2r\psi_t \partial_r).
\end{align*}
Now we can apply the cone covariant derivative formulas from \Cref{corollary: connection on cone}. For the first term, we have $\nabla^{\mathfrak{C}}_{(\nabla \varphi_t)}\nabla \psi_t =  \nabla^2\psi_t \cdot \nabla \varphi_t  - r\langle \nabla \psi_t, \nabla \varphi_t \rangle \partial_r$. For the second term, we have $2r\varphi_t \nabla ^\mathfrak{C}_{(\partial_r)}\nabla \psi_t = 2\varphi_t \nabla \psi_t$. For the third term we first apply the Leibniz rule and then apply \Cref{corollary: connection on cone} to obtain
\begin{align*}
    \nabla^{\mathfrak{C}}_{(\nabla \varphi_t)} (2r\psi_t\partial_r) &= \nabla \varphi_t(2r\psi_t)\partial_r + 2r\psi_t\nabla_{(\nabla \varphi_t)}^{\mathfrak{C}}\partial_r \\
    &= 2r\langle \nabla \varphi_t, \nabla \psi_t\rangle \partial_r + 2\psi_t \nabla \varphi_t.
\end{align*}
Similarly, for the fourth term we first apply the Leibniz rule to say
\begin{align*}
    2r\varphi_t\nabla ^{\mathfrak{C}}_{(\partial_r)}(2r\psi_t \partial_r) &= 4r\varphi_t \psi_t\partial_r.
\end{align*}
Combining everything gives 
\begin{align*}
    \mathbf{D}^{W_2, \mathcal{U}}_{t}(\mathcal{L}_t \mathbf{u}_t) = \begin{pmatrix}
        \partial_t \nabla \psi_t  + \nabla^2 \psi_t \cdot \nabla \varphi_t  + 2\varphi_t \nabla \psi_t + 2\psi_t \nabla \varphi_t \\
        2r\left(\partial_t \psi_t +\frac{1}{2}\langle \nabla \psi_t, \nabla \varphi_t\rangle + 2\varphi_t \psi_t\right)
    \end{pmatrix}.
\end{align*}
Applying $\Pi_{\mu_t}\circ \mathcal{P}_t$ yields
\begin{align*}
    (\Pi_{\mu_t} \circ \mathcal{P}_t)\left(\mathbf{D}^{W_2, \mathcal{U}}_{t}(\mathcal{L}_t \mathbf{u}_t) \right)= \Pi_{\mu_t}\begin{pmatrix}
        \partial_t \nabla \psi_t  + \nabla^2 \psi_t \cdot \nabla \varphi_t  + 2\varphi_t \nabla \psi_t + 2\psi_t \nabla \varphi_t \\
        \partial_t \psi_t +\frac{1}{2}\langle \nabla \psi_t, \nabla \varphi_t\rangle + 2\varphi_t \psi_t
    \end{pmatrix}
\end{align*}
which coincides with $\nabla^{\HK}$ as desired. 
\qed{}

\subsection{Proof of \Cref{parallel transport approx entire geodesic}} \label{proof of parallel transport approx entire geodesic}

Let $\mathbf{u}_t$ be a tangent field along $\mu_t$ that is parallel, i.e. $\nabla^{\HK}_{\mathbf{v}_t} \mathbf{u}_t = 0$ for a.e. $t$, and write $U_t = \mathcal{L}_t \mathbf{u}_t$. By \Cref{prop: pullback connection}, we know that this is equivalent to the condition that 
\[\left(\Pi_{\mu_t} \circ \mathcal{P}_t\right)(\mathbf{D}_t^{W_2, \mathcal{U}}(U_t)) = 0 \iff \left(\mathcal{L}_t \circ \Pi_{\mu_t} \circ \mathcal{P}_t\right)(\mathbf{D}_t^{W_2, \mathcal{U}}(U_t)) = 0.\]
Write $\Pi_{t}^S \triangleq \mathcal{L}_t \circ \Pi_{\mu_t}\circ \mathcal{P}_t$. Now we will leverage the following supplemental results. 
\begin{lemma} \label{orthog proj}
    The map $\Pi_t^S: L^2(\lambda_t; T\mathfrak{C}_\Omega) \rightarrow S_{\lambda_t} \subset T_{\lambda_t}\mathcal{P}_2(\mathfrak{C}_\Omega)$ is an orthogonal projection onto $S_{\lambda_t}$ when $\lambda_t$ has a deterministic conditional radial law. 
\end{lemma}
\begin{proof}
    In order to prove this statement, we need to verify three properties: linearity, idempotence, and self-adjointness. Linearity follows trivially from the fact that $\mathcal{L}_t, \Pi_{\mu_t}$ and $\mathcal{P}_t$ are linear, and thus the composition is linear. For idempotence, fix a $V \in L^2(\lambda_t; T\mathfrak{C}_\Omega)$ and observe that 
    \begin{align*}
        \Pi_t^S(\Pi_t^S(V)) = \Pi_t^S(Z)
    \end{align*}    
    for some $Z \in S_{\lambda_t}$. Thus, $\mathcal{L}_t^{-1}(Z) \in T_{\mu_t}\mathfrak{M}_+^\Gamma$ due to \Cref{lifted image}. Thus
    \begin{align*}
        \Pi_t^S (Z) &= \left(\mathcal{L}_t \circ \mathcal{L}_t^{-1}\right)(Z) = Z \quad \implies \quad \Pi_t^S(\Pi_t^S(V)) = Z = \Pi_t^S(V)
    \end{align*}
    for any $V\in L^2(\lambda_t; T\mathfrak{C}_\Omega),$ completing the proof of idempotence. For self-adjointness, observe the following. Fix a $\mathbf{u} \in T_{\mu_t}\mathfrak{M}_+^\Gamma$ and observe that for any $Z = a(x,r) + b(x,r)\partial_r \in L^2(\lambda_t; T\mathfrak{C}_\Omega)$ we have
    \begin{align*}
        \langle\mathbf{u}, \mathcal{P}_t(a + b\partial_r) \rangle_{\mu_t} &= \langle\mathbf{u}, (a(\cdot), b/2r(\cdot)) \rangle_{\mu_t} \tag{deterministic $\lambda(x|r)$}\\ 
        &= \langle \mathcal{L}_t(\mathbf{u}), \mathcal{L}_t(a(\cdot), b/2r(\cdot))\rangle_{\lambda_t} \tag{\Cref{isometry}}\\
        &= \langle \mathcal{L}_t(\mathbf{u}), Z\rangle_{\lambda_t}
    \end{align*}
    indicating that $\mathcal{P}_t$ is the adjoint of $\mathcal{L}_t$, i.e. $\mathcal{P}_t = \mathcal{L}_t^*$ on the relevant $L^2$ spaces. Now consider the adjoint of $\Pi_t^S$,
    \[(\Pi_t^S)^* = \left(\mathcal{L}_t \circ \Pi_{\mu_t}\circ \mathcal{P}_t\right)^* = \mathcal{L}_t\circ (\Pi_{\mu_t})^* \circ \mathcal{P}_t = \left(\mathcal{L}_t \circ \Pi_{\mu_t}\circ \mathcal{P}_t\right)\]
    since $\Pi_{\mu_t}$ is an orthogonal projection. Thus $\Pi_t^S$ is self adjoint, and therefore its an orthogonal projection onto $S_{\lambda_t}$. 
\end{proof}

\begin{lemma} \label{supplemental equivalence}
    Let $Z_t \in T_{\lambda_t}\mathcal{P}_2(\mathfrak{C}_\Omega)$ be an absolutely continuous cone vector field along a curve of measures $\lambda_t$ with tangent velocity $V_t$. Then we have the following equivalence,
    \[\Pi_{t}^S \left( \mathbf{D}_t^{W_2, \mathcal{U}}Z_t\right) = \Pi_{t}^S\left(\nabla^{W_2, \mathcal{U}}_{V_t}Z_t \right).\]
\end{lemma}
\begin{proof}
    Define $\widetilde Z_t = \mathbf{D}_t^{W_2, \mathcal{U}}Z_t$ and observe that
    \begin{align*}
        \widetilde Z_t = \Pi_{\lambda_t}(\widetilde Z_t) + (1 - \Pi_{\lambda_t})(\widetilde Z_t) = \nabla^{W_2, \mathcal{U}}_{V_t}Z_t + (1 - \Pi_{\lambda_t})(\widetilde Z_t).
    \end{align*}
    Plugging this in, we see that 
    \begin{align*}
        \Pi_{t}^S \left( \mathbf{D}_t^{W_2, \mathcal{U}}Z_t\right) &= \Pi_t^S\left(\nabla_{V_t}^{W_2, \mathcal{U}}Z_t\right) + \left(\Pi_t^S \circ (\operatorname{id} - \Pi_{\lambda_t})\right)\left(\mathbf{D}_t^{W_2, \mathcal{U}}Z_t\right).
    \end{align*}
    The second term on the right side of the equation above is $0$ since $(\operatorname{id} - \Pi_{\lambda_t})$ is a projection onto $T_{\lambda_t}\mathcal{P}_2(\mathfrak{C}_\Omega)^{\perp }$, while $\Pi_{t}^S$ is an orthogonal projection onto $S_{\lambda_t} \subset T_{\lambda_t}\mathcal{P}_2(\mathfrak{C}_\Omega)$. This completes the proof. 
\end{proof}
\Cref{supplemental equivalence} allows us to say $\mathbf{u}_t$ is HK-parallel if $\left(\mathcal{L}_t \circ \Pi_{\mu_t} \circ \mathcal{P}_t\right)(\nabla^{W_2, \mathcal{U}}_{V_t}U_t) = 0$. Now we will pull all objects back to the fixed tangent space at the measure $\lambda_0$. Write $\overline U(t) = \PT^{W_2, \mathcal{U}}_{\lambda_t \rightarrow \lambda_0}(U_t)$ and apply the limit definition of the Wasserstein covariant derivative \citep[Definition 5.1]{gigli2012second} to say
\begin{align*}
    \nabla_{V_t}^{W_2, \mathcal{U}}U_t &= \PT^{W_2, \mathcal{U}}_{\lambda_0 \rightarrow \lambda_t}\left(\PT^{W_2, \mathcal{U}}_{\lambda_t \rightarrow \lambda_0} \left(\nabla_{V_t}^{W_2, \mathcal{U}}U_t\right)\right) \\
    &= \PT^{W_2, \mathcal{U}}_{\lambda_0 \rightarrow \lambda_t}\left(\PT^{W_2, \mathcal{U}}_{\lambda_t \rightarrow \lambda_0} \left(\lim_{h \rightarrow 0^+}\frac{\PT^{W_2, \mathcal{U}}_{\lambda_{t + h}\rightarrow \lambda_{t}}U_{t+h} - U_t}{h}\right)\right) \\
    &= \PT^{W_2, \mathcal{U}}_{\lambda_0 \rightarrow \lambda_t} \left(\lim_{h \rightarrow 0^+}\frac{\PT^{W_2, \mathcal{U}}_{\lambda_{t + h}\rightarrow \lambda_{0}}U_{t+h} - \PT _{\lambda_t \rightarrow \lambda_0}U_t}{h}\right) \\
    &= \PT^{W_2, \mathcal{U}}_{\lambda_0 \rightarrow \lambda_t} \left(\dot{\overline{U}}(t)\right).
\end{align*}
Thus, the parallel condition on $\mathbf{u}_t$ is tantamount to 
\begin{align*}
    \Pi_t^S\left(\PT^{W_2, \mathcal{U}}_{\lambda_0 \rightarrow \lambda_t}\left(\dot{\overline{U}}(t)\right)\right) = 0 \iff \underbrace{\left( \PT^{W_2, \mathcal{U}}_{\lambda_t \rightarrow \lambda_0} \circ \,\Pi_t^S  \circ \PT^{W_2, \mathcal{U}}_{\lambda_0 \rightarrow \lambda_t} \right)}_{\triangleq \overline{\Pi}_t^S}\left(\dot{\overline{U}}(t)\right) = 0.
\end{align*}
So $\mathbf{u}_t$ is HK parallel if and only if
\begin{equation}
    \overline{\Pi}_t^S \left(\dot{\overline{U}}(t)\right) = 0, \quad \text{and}\quad U_t \in \operatorname{Im}(\mathcal{L}_t) \label{eq: lift and pullback parallel cond}
\end{equation}
where $\mathbf{u}_t = \mathcal{P}_t(U_t).$ It's quite easy to check that the second condition is the same as requiring \begin{equation}\overline{U}(t) = \overline{\Pi}_t^S(\overline{U}(t)). \label{pullback ode}
\end{equation} Differentiating this condition in time gives
\begin{align}
    \dot{\overline{U}}(t) &= \dot{\overline{\Pi}}_t^S\left(\overline{U}(t)\right) + \overline{\Pi}_t^S \left(\dot{\overline{U}}(t)\right) = \dot{\overline{\Pi}}_t^S\left(\overline{U}(t)\right) \label{pullback ode two}
\end{align}
by \Cref{eq: lift and pullback parallel cond}. Taylor expanding $\overline{\Pi}^S_t$ in time using the $C^2$ regularity of $t \mapsto \overline{\Pi}_t^S$ gives 
\begin{align*}
    \overline{\Pi}^S_{t+h}\left(\overline{U}(t)\right) &= \underbrace{\overline{\Pi}_t^S\left(\overline{U}(t)\right)}_{\mathclap{ \overline{U}(t) \text{ from \Cref{pullback ode}}}} + h \dot{\overline{\Pi}}_t^S\left(\overline{U}(t)\right) + O(h^2) \\
    &= \overline{U}(t) + h \underbrace{\dot{\overline{\Pi}}_t^S\left(\overline{U}(t)\right)}_{\mathclap{\dot{\overline{U}}(t) \text{ from \Cref{pullback ode two}}}} + O(h^2) \\
    &= \overline{U}(t) + h\dot{\overline{U}}(t) + O(h^2)
\end{align*}
where the $O(h^2)$ is in $L^2(\lambda_0)$ and uniform over $t \in [0, 1-h].$ Note that \cref{pullback ode two} guarantees that 
\[\ddot{\overline{U}}(t) = \ddot{\overline{\Pi}}_t^S\left(\overline{U}(t)\right) + \dot{\overline{\Pi}}_t^S \left(\dot{\overline{U}}(t)\right) = \ddot{\overline{\Pi}}_t^S\left(\overline{U}(t)\right) + \dot{\overline{\Pi}}_t^S \left(\dot{\overline{\Pi}}_t^S\left(\overline{U}(t)\right)\right)\]
which guarantees that $\overline{U}(t)$ inherits the same $C^2$ regularity as $\overline{\Pi}_t^S$. Therefore, we can Taylor expand $\overline{U}(t)$ in time to obtain
\begin{align*}
    \overline{U}(t+h) &= \overline{U}(t) + h \dot{\overline{U}}(t) + O(h^2).
\end{align*}
Taking the difference of the last two equations gives
\begin{equation} \label{quadratic bound}
\overline{\Pi}^S_{t+h}\left(\overline{U}(t)\right) - \overline{U}(t+h) = O(h^2).
\end{equation}
Now consider the approximation scheme in \Cref{alg:HK-parallel-transport}, which iterates \[\widehat U_{t+h} = \Pi_{t+h}^S\left(\PT^{\mathcal{U}}_{z \rightarrow T_{t \rightarrow t+h}(z)}(\widehat U_t)\right)\]
where $T_{t \rightarrow t+h}$ is the Brenier map from $\lambda_{t}$ to $\lambda_{t+h}$. By \Cref{orthog proj} we know that this is equivalent to 
\[\widehat U_{t+h} = \Pi_{t+h}^S\left(\PT^{\mathcal{U}}_{z \rightarrow T_{t \rightarrow t+h}(z)}(\widehat U_t)\right) = \Pi_{t+h}^S\left(\Pi_{\lambda_{t+h}}\left(\PT^{\mathcal{U}}_{z \rightarrow T_{t \rightarrow t+h}(z)}(\widehat U_t)\right)\right).\]
Now note that, due to \Cref{approximation scheme}, we have
\[\left\|\Pi_{\lambda_{t+h}}\left(\PT^{\mathcal{U}}_{z \rightarrow T_{t \rightarrow t+h}(z)}(\widehat U_t)\right) -  \PT^{W_2, \man}_{\lambda_t \rightarrow \lambda_{t+h}}(\widehat U_t)\right\|_{L^2(\lambda_{t+h})} \leq C h^2\|\widehat U_t\|_{L^2(\lambda_t)}.\]
\Cref{admissible class} and \Cref{uniformly regular lifts} together guarantee that $\sup_{t \in [0,1]}C_t \leq C$. By non-expansiveness of $\Pi_{t+h}^S$, 
\begin{equation} \label{quadratic approx two}
\left\|\Pi_{t+h}^S\left(\Pi_{\lambda_{t+h}}\left(\PT^{\mathcal{U}}_{z \rightarrow T_{t \rightarrow t+h}(z)}(\widehat U_t)\right) -  \PT^{W_2, \mathcal{U}}_{\lambda_t \rightarrow \lambda_{t+h}}(\widehat U_t)\right)\right\|_{L^2(\lambda_{t+h})} \leq C h^2\|\widehat U_t\|_{L^2(\lambda_t)}.\end{equation}
Now write 
\[R_{t+h} \triangleq \Pi_{t+h}^S\left(\Pi_{\lambda_{t+h}}\left(\PT^{\mathcal{U}}_{z \rightarrow T_{t \rightarrow t+h}(z)}(\widehat U_t)\right) -  \PT^{W_2, \mathcal{U}}_{\lambda_t \rightarrow \lambda_{t+h}}(\widehat U_t)\right)\]
and observe that
\begin{align*}
    \widehat{U}_{t+h} = \Pi_{t+h}^S\left(\PT^{W_2, \mathcal{U}}_{\lambda_t \rightarrow \lambda_{t+h}}(\widehat U_t)\right) + R_{t+h}.
\end{align*}
Pulling back to $t = 0$ and defining $\overline{U}_t = \PT^{W_2, \mathcal{U}}_{\lambda_t \rightarrow \lambda_0}(\widehat U_t)$ gives
\begin{align*}
    \overline{U}_{t+h} &= \PT^{W_2, \mathcal{U}}_{\lambda_{t+h} \rightarrow \lambda_0}\left(\Pi_{t+h}^S\left(\PT^{W_2, \mathcal{U}}_{\lambda_t \rightarrow \lambda_{t+h}}(\widehat U_t)\right)\right) + \PT^{W_2, \mathcal{U}}_{\lambda_{t+h} \rightarrow \lambda_0}\left(R_{t+h}\right)\\
    &= \PT^{W_2, \mathcal{U}}_{\lambda_{t+h} \rightarrow \lambda_0}\left(\Pi_{t+h}^S\left(\PT^{W_2, \mathcal{U}}_{\lambda_0 \rightarrow \lambda_{t+h}}(\overline U_t)\right)\right) + \PT^{W_2, \mathcal{U}}_{\lambda_{t+h} \rightarrow \lambda_0}\left(R_{t+h}\right)
\end{align*}
which can be written like so,
\begin{equation}
    \overline{U}_{t+h} = \overline{\Pi}^S_{t+h}\left(\overline{U}_t\right) + \PT^{W_2, \mathcal{U}}_{\lambda_{t+h} \rightarrow \lambda_0}\left(R_{t+h}\right).
\end{equation}
Subtracting off $\overline{U}(t + h)$ yields
\begin{align*}
    \overline{U}_{t+h} - \overline{U}(t+h) &= \overline{\Pi}^S_{t+h}\left(\overline{U}_t\right) - \overline{U}(t+h) + \PT^{W_2, \mathcal{U}}_{\lambda_{t+h} \rightarrow \lambda_0}\left(R_{t+h}\right) \\
    &=  \overline{\Pi}^S_{t+h}\left(\overline{U}_t - \overline{U}(t)\right) + \left[\overline{\Pi}^S_{t+h} \left(\overline{U}(t)\right) - \overline{U}(t+h)\right] + \PT^{W_2, \mathcal{U}}_{\lambda_{t+h} \rightarrow \lambda_0}\left(R_{t+h}\right).
\end{align*}
Taking the $\|\cdot \|_{L^2(\lambda_0)}$ norm and applying Minkowski's inequality yields
\begin{align*}
    \left\|\overline{U}_{t+h} - \overline{U}(t+h)\right\|_{L^2(\lambda_0)} &\lesssim \left\|\overline{\Pi}^S_{t+h}\left(\overline{U}_t - \overline{U}(t)\right) \right\|_{L^2(\lambda_0)} + \left\|\overline{\Pi}^S_{t+h} \left(\overline{U}(t)\right) - \overline{U}(t+h)\right\|_{L^2(\lambda_0)}\\
    &\hspace{75mm}+ \left\|\PT^{W_2, \mathcal{U}}_{\lambda_{t+h} \rightarrow \lambda_0}\left(R_{t+h}\right)\right\|_{L^2(\lambda_0)} \\
    &= \left\|\overline{\Pi}^S_{t+h}\left(\overline{U}_t - \overline{U}(t)\right) \right\|_{L^2(\lambda_0)}  + O(h^2) + \left\|\PT^{W_2, \mathcal{U}}_{\lambda_{t+h} \rightarrow \lambda_0}\left(R_{t+h}\right)\right\|_{L^2(\lambda_0)} \tag{\Cref{quadratic bound}} \\
    &\leq \left\|\overline{U}_t - \overline{U}(t)\right\|_{L^2(\lambda_0)}  + O(h^2)  + \left\|\PT^{W_2, \mathcal{U}}_{\lambda_{t+h} \rightarrow \lambda_0}\left(R_{t+h}\right)\right\|_{L^2(\lambda_0)}. \tag{non-expansiveness} 
\end{align*}
By the isometry of parallel transport and \Cref{quadratic approx two}, 
\begin{align*}
    \left\|\PT^{W_2, \mathcal{U}}_{\lambda_{t+h} \rightarrow \lambda_0}\left(R_{t+h}\right)\right\|_{L^2(\lambda_0)} \leq Ch^2\|\widehat U_t\|_{L^2(\lambda_t)}.
\end{align*}
To bound this quantity, we need one final supplemental result.
\begin{lemma} \label{norm control iterate}
    The iterates $\widehat U_t$ produced via \Cref{alg:HK-parallel-transport} satisfy $\|\widehat U_t\|_{L^2(\lambda_t)} \leq \|\widehat U_0\|_{L^2(\lambda_0)}$.
\end{lemma}
\begin{proof}
    Recall the iteration scheme, $\widehat U_{t+h} =  \Pi_{t+h}^S\left(\PT^{\mathcal{U}}_{z \rightarrow T_{t \rightarrow t+h}(z)}(\widehat U_t)\right)$. We will first show the pushforward identity $(T_{t \rightarrow t+h})_\# \lambda_t = \lambda_{t+h}$, and then we will use this to prove the lemma statement. We have 
    \begin{align*}
       \lambda_{t+h} &= \left(x \mapsto \left(x, q_{t}(T_t^{-1}(x))\,r_t(T_t^{-1}(x))\right)\right)_\#\eta_{t+h} \\
       &= \left(x \mapsto \left(x, \sqrt{\frac{u_{t+h}(x)}{u_t\left(T_{t\rightarrow t+h}^{-1}(x)\right)}}\,r_t\left(T_{t\rightarrow t+h}^{-1}(x)\right)\right)\right)_\#\eta_{t+h} \tag{\Cref{alg:isometric-lift}} \\
       &= \left(x \mapsto \left(T_{t \rightarrow t+h}(x), \sqrt{\frac{u_{t+h}(T_{t \rightarrow t+h}(x))}{u_t\left(x\right)}}\,r_t\left(x\right)\right)\right)_\#\eta_{t}.
    \end{align*}
    Since $\lambda_t$ almost every $x$ satisfies $r_t(x) = r$, we have 
    \begin{align*}
       \lambda_{t+h}  &= \left((x,r) \mapsto \left(T_{t \rightarrow t+h}(x), \sqrt{\frac{u_{t+h}(T_{t \rightarrow t+h}(x))}{u_t\left(x\right)}}\,r\right)\right)_\#\lambda_{t} \\
       &=\left((x,r) \mapsto \left(T_{t+h}(x), q_t(x)\,r_t(x)\right)\right)_\#\lambda_{t}.
    \end{align*}
    By \Cref{cone exponential and log map} and the definition of $V_t$, we have $\exp^\mathfrak{C}_{(x,r)}(hV_t(x,r)) = (T_{t+h}(x), q_t(x)r_t(x))$, and thus $\lambda_{t+h} = (z \mapsto\exp_{z}^\mathfrak{C}(V_t(x)))_\# \lambda_t = (T_{t\rightarrow t+h})_\# \lambda_t$. Now, by the non-expansiveness of $\Pi_{t+h}^S$,
    \begin{align*}
        \|\widehat U_{t+h}\|_{L^2(\lambda_{t+h})}^2 &\leq \|\PT^{\mathcal{U}}_{z \rightarrow T_{t \rightarrow t+h}(z)}(\widehat U_t)\|_{L^2(\lambda_{t+h})}^2 \\
        &=\int_{\mathcal{U}}\left\|\PT^{\mathcal{U}}_{z \rightarrow T_{t \rightarrow t+h}(z)}(\widehat U_t)\right\|_{\mathfrak{C}}^2\, d\lambda_{t+h} \\
        &= \int_{\mathcal{U}}\|\widehat U_t\|_{\mathfrak{C}}^2 \,d\lambda_t \tag{due to $(T_{t\rightarrow t+h})_\#\lambda_t = \lambda_{t+h}$} \\
        &= \|\widehat U_t\|_{L^2(\lambda_t)}^2.
    \end{align*}
    Iterating gives the desired result.
\end{proof}
\Cref{norm control iterate} therefore allows us to say 
\begin{align*}
    \left\|\overline{U}_{t+h} - \overline{U}(t+h)\right\|_{L^2(\lambda_0)} &\lesssim  \left\|\overline{U}_t - \overline{U}(t)\right\|_{L^2(\lambda_0)}  + O(h^2). 
\end{align*}
Now summing over iterates from \Cref{alg:HK-parallel-transport} and using $h = O(N^{-1})$ yields
\begin{align*}
    \left\|\overline{U}_{1} - \overline{U}(1)\right\|_{L^2(\lambda_0)} &\leq \sum_{i = 1}^NO(N^{-2}) = O(N^{-1}).
\end{align*}
By the isometry parallel transport, we can parallel transport to $\lambda_1$ and obtain the same rate, $\|\widehat U_1 - U_1\|_{L^2(\lambda_1)} = O(N^{-1})$. Moreover, the isometry of $\mathcal{P}_t$ on $S_{\lambda_t}$ allows us to conclude
\[ \|\mathcal{P}_1(\widehat U_1) - \PT_{\mu_0 \rightarrow \mu_1}^{\HK}(\mathbf{u}_0)\|_{T_{\mu_1}\mathfrak{M}_+^\Gamma} = O(N^{-1}).\]
\qed{}

\section{Parallel Transport Implementation Details} \label{sec: implementation details}

The empirical implementation described in this section is available in the public repository
\[
\texttt{\url{https://github.com/TristanSaidi/HKPT}}.
\]
Its numerical structure follows the same general philosophy as that described in the implementation details appendix in \citet{saidi2026wassersteinparalleltransportpredicting}: one first replaces a non-deterministic empirical transport plan by a deterministic map through barycentric projection, and then uses the coupling itself to aggregate transported tangent information whenever several source atoms contribute to a common target atom. For empirical measures we write
\[
\hat\mu_k = \sum_{i=1}^{n_k} a_i^{(k)}\,\delta_{x_i^{(k)}},
\]
and, on each local step of the lifted path, we solve the discrete logarithmic entropy transport problem to obtain a coupling matrix $\pi^{(k)} \in \mathbb{R}_+^{n_k \times n_{k+1}}$. The code then converts this local plan into an empirical HK tangent, lifts that tangent to the cone, transports it on the cone, and finally pushes it back to the base space.
\medskip

\paragraph{\textbf{Barycentric projection for the input tangent.}}
In the parallel-transport implementation, barycentric projection enters only at the initial stage where one constructs the empirical HK tangent to be transported. More precisely, suppose one begins with two empirical measures and wishes to transport the empirical HK logarithmic map from the source to the target. If the corresponding LET optimizer returns a coupling that is not supported on a map, then the implementation replaces this non-deterministic plan by a deterministic source-supported tangent through barycentric averaging. Given a source atom $x_i$, the corresponding row of the coupling is normalized to a conditional distribution and one sets
\[
\bar y_i
=
\frac{\sum_{j = 1}^{n_1} \pi_{ij} y_j}{\sum_{j = 1}^{n_1} \pi_{ij}}.
\]
Thus the mass splitting encoded by $\pi$ is summarized by the conditional barycenter of its target support. In the default approximation used in the code, one does not stop at averaging the target locations. Instead, for each active pair $(x_i,y_j)$ in the support of $\pi$, one computes the corresponding edgewise HK logarithmic-map quantities and then averages these with the conditional weights
\[
P_{ij}
=
\frac{\pi_{ij}}{\sum_{j = 1}^{n_1} \pi_{ij}}.
\]
More precisely, if $\Delta t$ denotes the local time step and
\[
u_i^{\mathrm{src}} = \frac{a_i}{(\pi \mathbf 1)_i},
\qquad
u_j^{\mathrm{tgt}} = \frac{b_j}{(\pi^\top \mathbf 1)_j},
\]
then for each edge one forms the discrete HK scaling factor
\[
q_{ij} = \sqrt{\frac{u_j^{\mathrm{tgt}}}{u_i^{\mathrm{src}}}},
\]
the associated edge velocity $v_{ij}$, and the edge reaction coefficient $\beta_{ij}$ using the closed-form HK logarithmic-map formulas from the cone model. The local tangent field at $x_i$ is then the barycentric average
\[
v_i = \sum_{j = 1}^{n_1} P_{ij} v_{ij},
\qquad
\beta_i = \sum_{j = 1}^{n_1} P_{ij} \beta_{ij}.
\]
This produces an empirical HK tangent $(v_i,\beta_i)$ on the source support, which is the tangent subsequently lifted and transported. By contrast, once the parallel-transport recursion itself is started, the transition couplings along the lifted path are \emph{not} barycentrically projected: they are retained as couplings and used directly to aggregate incoming cone tangents at each step.
\medskip

\paragraph{\textbf{Coupling-based aggregation of parallel transport.}}
The second important implementation point is the aggregation of transported tangents according to the coupling. This is the HK analogue of the weighted aggregation step described in Appendix~D of \citet{saidi2026wassersteinparalleltransportpredicting}. After constructing the lifted path $(\lambda_k)_{k=0}^N$, the initial HK tangent $(v_0,\beta_0)$ is lifted to a cone tangent via
\[
\mathcal L(v_0,\beta_0)(x,r) = \bigl(v_0(x), 2\beta_0(x)r\bigr),
\]
and is then transported recursively on the cone. When the lifted path is driven by the discrete local LET plans, each active edge $(i,j)$ of the local coupling contributes an edgewise parallel transport from the source cone atom $z_i^{(k)}$ to the target cone atom $z_j^{(k+1)}$. The contribution of this transported edge tangent is weighted by the edge reference mass
\[
m_{ij}^{(k)}
=
\lambda_k(i)\,\frac{\pi_{ij}^{(k)}}{\sum_{j' = 1}^{n_{k+1}} \pi_{ij'}^{(k)}}.
\]
Hence the tangent assigned to a target atom is the coupling-weighted average of all incoming parallel-transported edge tangents:
\[
U_{k+1}(j)
=
\frac{1}{\sum_{i = 1}^{n_k} m_{ij}^{(k)}}
\sum_{i = 1}^{n_k} m_{ij}^{(k)}\,\PT^{\mathfrak{C}}_{z_i^{(k)} \to z_j^{(k+1)}}\bigl(U_k(i)\bigr).
\]
This is the precise sense in which the implementation aggregates parallel transport according to the coupling: if several source atoms send mass to the same target atom, one does not select a preimage arbitrarily, but instead averages all incoming transported tangents using the masses induced by the local LET plan. 
\medskip

\paragraph{\textbf{No tangent-space projection step.}}
In the theoretical HK parallel transport algorithm (\Cref{alg:HK-parallel-transport}) the projection operator $\Pi_{\mu_t}$ is used. Our implementation does \emph{not} apply this step for the following reason: given an empirical point cloud $\{x_1, \dots, x_n\} \subset \mathbb{R}^d$ and empirical transport-reaction observations $\{(v_i, \beta_i)\}_{i = 1}^n$, one can construct an HK tangent $(\nabla \phi, \phi)$ with $\phi \in C^\infty_c(\mathbb{R}^d)$ such that $\phi(x_i) = \beta_i$ and $\nabla \phi(x_i) = v_i$. In particular, let $\varepsilon = \min_{x_i \neq x_j} \|x_i - x_j\|/4$ and let $\eta_i \in C^\infty_c(\mathbb{R}^d)$ satisfy
\[ \eta_i = 1 \quad \text{on a neighborhood of $x_i$, and }  \operatorname{supp} \eta_i \subset B(x_i, \varepsilon)\]
where $B(x_i, \varepsilon)$ is the metric ball of radius $\varepsilon$ centered at $x_i$. Define 
\[\phi(x) = \sum_{i = 1}^n \eta_i(x) \left(\beta_i + \langle v_i, x - x_i\rangle\right). \]
Then $(\nabla \phi(x_i), \phi(x_i)) = (v_i, \beta_i)$ for all $i$. We refer to this construction as justification for \textit{not} applying $\Pi_{\mu_t}$ on empirical observations of transport reaction fields, as for any finite $n$ there exists an HK tangent that coincides with the observations. 

\bibliographystyle{abbrvnat}
\bibliography{main}

@article{chizat2018interpolating,
  title={An interpolating distance between optimal transport and Fisher--Rao metrics},
  author={Chizat, Lenaic and Peyr{\'e}, Gabriel and Schmitzer, Bernhard and Vialard, Fran{\c{c}}ois-Xavier},
  journal={Foundations of Computational Mathematics},
  volume={18},
  number={1},
  pages={1--44},
  year={2018},
  publisher={Springer}
}

@book{ambrosio2005gradient,
  title={Gradient flows: in metric spaces and in the space of probability measures},
  author={Ambrosio, Luigi and Gigli, Nicola and Savar{\'e}, Giuseppe},
  year={2005},
  publisher={Springer}
}

@book{petersen2006riemannian,
  title={Riemannian geometry},
  author={Petersen, Peter},
  year={2006},
  publisher={Springer}
}

@phdthesis{clancy2021interpolating,
  author = {Clancy, Julien},
  title = {Interpolating Spline Curves of Measures},
  school = {Massachusetts Institute of Technology},
  year = {2021}
}

@book{lee2018introduction,
  title={Introduction to Riemannian manifolds},
  author={Lee, John M},
  volume={2},
  year={2018},
  publisher={Springer}
}

@book{gigli2012second,
  title={Second Order Analysis on {$\mathscr{P}_2(M)$}},
  author={Gigli, Nicola},
  year={2012},
  publisher={American Mathematical Society}
}

@article{liero2016optimal,
  title={Optimal transport in competition with reaction: The Hellinger--Kantorovich distance and geodesic curves},
  author={Liero, Matthias and Mielke, Alexander and Savar{\'e}, Giuseppe},
  journal={SIAM Journal on Mathematical Analysis},
  volume={48},
  number={4},
  pages={2869--2911},
  year={2016},
  publisher={SIAM}
}

@book{o1983semi,
  title={Semi-Riemannian geometry with applications to relativity},
  author={O'neill, Barrett},
  volume={103},
  year={1983},
  publisher={Academic press}
}

@article{liero2018optimal,
  title={Optimal entropy-transport problems and a new Hellinger--Kantorovich distance between positive measures},
  author={Liero, Matthias and Mielke, Alexander and Savar{\'e}, Giuseppe},
  journal={Inventiones mathematicae},
  volume={211},
  number={3},
  pages={969--1117},
  year={2018},
  publisher={Springer}
}

@misc{chewi2024statisticaloptimaltransport,
      title={Statistical optimal transport}, 
      author={Sinho Chewi and Jonathan Niles-Weed and Philippe Rigollet},
      year={2024},
      eprint={2407.18163},
      archivePrefix={arXiv},
      primaryClass={math.ST},
      url={https://arxiv.org/abs/2407.18163}, 
}

@misc{kondratyev2016newoptimaltransportdistance,
      title={A new optimal transport distance on the space of finite Radon measures}, 
      author={Stanislav Kondratyev and Léonard Monsaingeon and Dmitry Vorotnikov},
      year={2016},
      eprint={1505.07746},
      archivePrefix={arXiv},
      primaryClass={math.AP},
      url={https://arxiv.org/abs/1505.07746}, 
}

@book{burago2001course,
  title={A course in metric geometry},
  author={Burago, Dmitri and Burago, Yuri and Ivanov, Sergei and others},
  volume={33},
  year={2001},
  publisher={American Mathematical Society Providence}
}

@article{cai2022linearized,
  title={The Linearized Hellinger--Kantorovich Distance},
  author={Cai, Tianji and Cheng, Junyi and Schmitzer, Bernhard and Thorpe, Matthew},
  journal={SIAM Journal on Imaging Sciences},
  volume={15},
  number={1},
  pages={45--83},
  year={2022},
  publisher={SIAM}
}

@article{clancy2022wasserstein,
  title={Wasserstein-fisher-rao splines},
  author={Clancy, Julien and Suarez, Felipe},
  journal={arXiv preprint arXiv:2203.15728},
  year={2022}
}

@article{mccann2001polar,
  title={Polar factorization of maps on Riemannian manifolds},
  author={McCann, Robert J},
  journal={Geometric \& Functional Analysis GAFA},
  volume={11},
  number={3},
  pages={589--608},
  year={2001},
  publisher={Springer}
}

@incollection{ambrosio2012user,
  title={A user’s guide to optimal transport},
  author={Ambrosio, Luigi and Gigli, Nicola},
  booktitle={Modelling and Optimisation of Flows on Networks: Cetraro, Italy 2009, Editors: Benedetto Piccoli, Michel Rascle},
  pages={1--155},
  year={2012},
  publisher={Springer}
}

@book{boumal2023introduction,
  title={An introduction to optimization on smooth manifolds},
  author={Boumal, Nicolas},
  year={2023},
  publisher={Cambridge University Press}
}

@misc{saidi2026wassersteinparalleltransportpredicting,
      title={Wasserstein Parallel Transport for Predicting the Dynamics of Statistical Systems}, 
      author={Tristan Luca Saidi and Gonzalo Mena and Larry Wasserman and Florian Gunsilius},
      year={2026},
      eprint={2603.23736},
      archivePrefix={arXiv},
      primaryClass={stat.ML},
      url={https://arxiv.org/abs/2603.23736}, 
}

@article{nomizu1961existence,
  title={The existence of complete Riemannian metrics},
  author={Nomizu, Katsumi and Ozeki, Hideki},
  journal={Proceedings of the American Mathematical Society},
  volume={12},
  number={6},
  pages={889--891},
  year={1961},
  publisher={JSTOR}
}

@article{schiebinger2019optimal,
  title={Optimal-transport analysis of single-cell gene expression identifies developmental trajectories in reprogramming},
  author={Schiebinger, Geoffrey and Shu, Jian and Tabaka, Marcin and Cleary, Brian and Subramanian, Vidya and Solomon, Aryeh and Gould, Joshua and Liu, Siyan and Lin, Stacie and Berube, Peter and others},
  journal={Cell},
  volume={176},
  number={4},
  pages={928--943},
  year={2019},
  publisher={Elsevier}
}

@article{sejourne2023unbalanced,
  title={Unbalanced optimal transport, from theory to numerics},
  author={S{\'e}journ{\'e}, Thibault and Peyr{\'e}, Gabriel and Vialard, Fran{\c{c}}ois-Xavier},
  journal={Handbook of Numerical Analysis},
  volume={24},
  pages={407--471},
  year={2023},
  publisher={Elsevier}
}

@article{monsaingeon2021new,
  title={A new transportation distance with bulk/interface interactions and flux penalization},
  author={Monsaingeon, L{\'e}onard},
  journal={Calculus of Variations and Partial Differential Equations},
  volume={60},
  number={3},
  pages={101},
  year={2021},
  publisher={Springer}
}

@article{chizat2016scaling,
  title={Scaling algorithms for unbalanced transport problems},
  author={Chizat, Lenaic and Peyr{\'e}, Gabriel and Schmitzer, Bernhard and Vialard, Fran{\c{c}}ois-Xavier},
  journal={arXiv preprint arXiv:1607.05816},
  year={2016}
}

@article{zhang2024learning,
  title={Learning stochastic dynamics from snapshots through regularized unbalanced optimal transport},
  author={Zhang, Zhenyi and Li, Tiejun and Zhou, Peijie},
  journal={arXiv preprint arXiv:2410.00844},
  year={2024}
}

@article{chewi2023log,
  title={Log-concave sampling},
  author={Chewi, Sinho},
  journal={Book draft available at https://chewisinho. github. io},
  volume={9},
  pages={17--18},
  year={2023}
}

@article{cloninger2025linearized,
  title={Linearized Wasserstein dimensionality reduction with approximation guarantees},
  author={Cloninger, Alexander and Hamm, Keaton and Khurana, Varun and Moosm{\"u}ller, Caroline},
  journal={Applied and Computational Harmonic Analysis},
  volume={74},
  pages={101718},
  year={2025},
  publisher={Elsevier}
}

@article{laschos2019geometric,
  title={Geometric properties of cones with applications on the Hellinger--Kantorovich space, and a new distance on the space of probability measures},
  author={Laschos, Vaios and Mielke, Alexander},
  journal={Journal of Functional Analysis},
  volume={276},
  number={11},
  pages={3529--3576},
  year={2019},
  publisher={Elsevier}
}

@article{gallouet2025regularity,
  title={Regularity theory and geometry of unbalanced optimal transport},
  author={Gallou{\"e}t, Thomas and Ghezzi, Roberta and Vialard, Fran{\c{c}}ois-Xavier},
  journal={Journal of Functional Analysis},
  volume={289},
  number={7},
  pages={111042},
  year={2025},
  publisher={Elsevier}
}

@misc{ponnoprat2026minimaxoptimalestimationtransportgrowth,
      title={Minimax Optimal Estimation of Transport-Growth Pairs in Unbalanced Optimal Transport}, 
      author={Donlapark Ponnoprat and Noboru Isobe and Masaaki Imaizumi},
      year={2026},
      eprint={2605.08705},
      archivePrefix={arXiv},
      primaryClass={math.ST},
      url={https://arxiv.org/abs/2605.08705}, 
}

\end{document}